\documentclass[aos,preprint]{imsart}
\RequirePackage[OT1]{fontenc}
\RequirePackage{amsmath, amsfonts, amssymb, amsthm, amsbsy, mathrsfs, amscd, bm}
\RequirePackage[numbers]{natbib}
\RequirePackage[colorlinks,citecolor=blue,urlcolor=blue]{hyperref}
\usepackage{color}
\usepackage{url}
\usepackage{graphicx}
\usepackage{ifthen}
\usepackage{xspace}
\usepackage{algorithm, algorithmic}
\usepackage{subcaption}
\usepackage{xr}
\usepackage[normalem]{ulem}

\arxiv{arXiv:1503.08393}

\startlocaldefs
\numberwithin{equation}{section}
\theoremstyle{plain}
\newtheorem{theorem}{Theorem}[section]
\newtheorem{lemma}[theorem]{Lemma}
\newtheorem{corollary}[theorem]{Corollary}
\newtheorem{proposition}[theorem]{Proposition}

\newtheorem{fact}{Fact}[section]

\theoremstyle{definition}
\newtheorem{definition}{Definition}[section]
\newcommand{\R}{\mathbb{R}}

\def\iid{i.i.d.\xspace}                %   independent, identically distributed
\def\normaldist{\mathcal{N}}        %   normal distribution

\newcommand{\<}{\langle}
\renewcommand{\>}{\rangle}

\newcommand{\goto}{\rightarrow}
\newcommand{\sgn}{\textrm{sgn}}
\renewcommand{\P}{\operatorname{\mathbb{P}}}
\newcommand{\E}{\operatorname{\mathbb{E}}}
\newcommand{\norm}[1]{{\|#1\|}}

\renewcommand{\vec}[1]{{\boldsymbol{#1}}}
\newcommand{\supp}[1]{\operatorname{supp}(#1)}

\newcommand{\half}{{\textstyle\frac{1}{2}}}

\definecolor{eac}{RGB}{200,50,50}
\definecolor{gosia}{RGB}{50,200,50}
\definecolor{ejc}{RGB}{255,0,0}
\definecolor{wjs}{RGB}{100,100,0}
\definecolor{vdb}{RGB}{0,0,255}
\definecolor{srb}{RGB}{200,10,255}

\newcommand{\BH}{BHq\xspace}

\newcommand{\slope}[1]{\operatorname{prox}_{\bm\lambda}\left( #1 \right)}
\newcommand{\slopex}[2]{\operatorname{prox}_{#1}\left( #2 \right)}

\newcommand{\prox}[1]{\operatorname{prox}_{#1}}
\newcommand{\order}[1]{#1_{(i)}}
\newcommand{\preceqm}{\preceq}

\newcommand\lambe{\bm{\lambda}_{\epsilon}}
\newcommand{\lambex}[1]{\ifthenelse{\equal{#1}{}}{\lambda_{\epsilon}}{\lambda_{\epsilon, #1}}}
\newcommand\lambbh{\lambda^{\textnormal{\tiny BH}}}
\newcommand\ceil[1]{\lceil #1 \rceil}
\newcommand\floor[1]{\lfloor #1 \rfloor}
\renewcommand\d{\mathrm{d}}

\newcommand{\lb}{\left(}
\newcommand{\rb}{\right)}

\newcommand{\K}{k^{\star}}
\def\X{\bm{X}}
\def\Beta{\bm{\beta}}
\renewcommand{\hat}[1]{\widehat{#1}}
\renewcommand{\tilde}[1]{\widetilde{#1}}
\def\eqd{\,{\buildrel d \over =}\,} 
\newcommand{\barS}{{\overline{S}}}

\newcommand\e[1]{\mathrm{e}^{#1}}
\newcommand{\ee}[1]{\ifthenelse{\equal{#1}{1}}{\mathrm{e}}{\mathrm{e}^{#1}}}

\graphicspath{{./figs/}}

\numberwithin{equation}{section}

\endlocaldefs

\begin{document}

\begin{frontmatter}
\title{SLOPE is Adaptive to Unknown Sparsity and Asymptotically Minimax}
\runtitle{Adaptivity and Minimaxity of SLOPE}
%\thankstext{T1}{Footnote to the title with the ``thankstext'' command.}

\begin{aug}
\author{\fnms{Weijie} \snm{Su}\thanksref{t1}\ead[label=e1]{wjsu@stanford.edu}}
\and
\author{\fnms{Emmanuel} \snm{Cand\`es}\thanksref{t2}\ead[label=e2]{candes@stanford.edu}}
%\ead[label=u1,url]{http://www.foo.com}

\thankstext{t1}{Supported in part by a General Wang Yaowu Stanford Graduate Fellowship.}
\thankstext{t2}{Supported in part by NSF under grant CCF-0963835 and by the Math + X Award from the Simons Foundation.}
%\thankstext{t3}{Corresponding author.}
%\thankstext{t3}{Second supporter of the project}
\runauthor{W.~Su and E.~Cand\`es}

\affiliation{Stanford University}

\address{Departments of Statistics\\
Stanford University\\
Stanford, CA 94305\\
USA\\
\printead{e1}}

\address{Departments of Statistics and Mathematics\\
Stanford University\\
Stanford, CA 94305\\
USA\\
\printead{e2}}
\end{aug}

\begin{abstract}
  We consider high-dimensional sparse regression problems in which we
  observe $\bm y = \X \bm \beta + \bm z$, where $\X$ is an $n \times
  p$ design matrix and $\bm z$ is an $n$-dimensional vector of
  independent Gaussian errors, each with variance $\sigma^2$. Our
  focus is on the recently introduced SLOPE estimator \cite{slope},
  which regularizes the least-squares estimates with the
  rank-dependent penalty $\sum_{1 \le i \le p} \lambda_i |\hat
  \beta|_{(i)}$, where $|\hat \beta|_{(i)}$ is the $i$th largest
  magnitude of the fitted coefficients. Under Gaussian designs, where
  the entries of $\X$ are i.i.d.~$\mathcal{N}(0, 1/n)$, we show that
  SLOPE, with weights $\lambda_i$ just about equal to $\sigma \cdot
  \Phi^{-1}(1-iq/(2p))$ ($\Phi^{-1}(\alpha)$ is the $\alpha$th
  quantile of a standard normal and $q$ is a fixed number in $(0,1)$)
  achieves a squared error of estimation obeying
\[
\sup_{\|\bm\beta\|_0 \le k} \,\, \P \left(\|
  \hat{\bm\beta}_{\textnormal{\tiny SLOPE}} - \bm\beta \|^2
  >   (1+\epsilon) \, 2\sigma^2 k \log(p/k) \right)
\longrightarrow 0
\] as the dimension $p$ increases to $\infty$, and where $\epsilon >
0$ is an arbitrary small constant. This holds under a weak
  assumption on the $\ell_0$-sparsity level, namely, $k/p \goto 0$ and
  $(k\log p)/n \goto 0$, and is sharp in the sense that this is the
  best possible error {\em any} estimator can achieve.  A remarkable
  feature is that SLOPE does not require any knowledge of the degree
  of sparsity, and yet automatically adapts to yield optimal total
  squared errors over a wide range of $\ell_0$-sparsity classes. We
are not aware of any other estimator with this property.
\end{abstract}

\begin{keyword}[class=MSC]
\kwd[Primary ]{62C20}
%%\kwd{60K35}
%% http://www.ams.org/mathscinet/msc/msc2010.html?t=62-XX&s=&btn=Search&ls=s
\kwd[; secondary ]{62G05}
\kwd{62G10}
\kwd{62J15}.
\end{keyword}


\end{frontmatter}

\section{Introduction}
\label{sec:introduction}
Twenty years ago, Benjamini and Hochberg proposed the false discovery
rate (FDR) as a new measure of type-I error for multiple testing,
along with a procedure for controlling the FDR in the case of
statistically independent tests \cite{BH95}. In words, the FDR is the
expected value of the ratio between the number of false rejections and
the total number of rejections, with the convention that this ratio
vanishes in case no rejection is made.  To describe the
Benjamini-Hochberg procedure, henceforth referred to as the \BH
procedure, imagine we observe a $p$-dimensional vector $\bm y \sim
\mathcal{N}(\bm \beta, \sigma^2 \bm{I}_p)$ of independent statistics
$\{y_i\}$, and wish to test which means $\beta_i$ are nonzero.  Begin
by ordering the observations as $|y|_{(1)} \ge |y|_{(2)} \ge \cdots
\ge |y|_{(p)}$---that is, from the most to the least significant---and
compute a data-dependent threshold given by
\[
\hat{t}_{\textnormal{\tiny FDR}} = |y|_{(R)},  
\]
where $R$ is the last time $|y|_{(i)}/\sigma$ exceeds a critical curve
$\lambbh_i$: formally,
\begin{equation}
\label{eq:BH}
R \triangleq \max\left\{i : |y|_{(i)}/\sigma \ge \lambbh_i \right\}
\text{ with } \lambbh_i = \Phi^{-1}\left( 1-iq/(2p) \right);
\end{equation}
throughout, $0 < q < 1$ is a target FDR level and $\Phi$ is the
cumulative distribution function of a standard normal random
variable. (The chance that a null statistic $z \sim \mathcal{N}(0,1)$
exceeds $\lambbh_i$ is $\P(|z| \ge \lambbh_i) = q \cdot i/p$.)  Then
\BH rejects all those hypotheses with
$|y_i| \ge \hat{t}_{\textnormal{\tiny FDR}}$ and makes no rejection in
the case where all the observations fall below the critical curve,
i.e.~when the set $\{i : |y|_{(i)}/\sigma \ge \lambbh_i\}$ is
empty. In short, the hypotheses corresponding to the $R$ most
significant statistics are rejected. Letting $V$ be the number of
false rejections, Benjamini and Hochberg proved that this procedure
controls the FDR in the sense that
\[
\mathrm{FDR} = \E \left[ \frac{V}{R \vee 1}\right] = \frac{qp_0}{p}
\le q,
\]
where $p_0 = |\{i: \beta_i = 0\}|$ is the total number of nulls.
Unlike a Bonferroni procedure---see e.g.~\cite{bonferroni}---where the
threshold for significance is fixed in advance, a very appealing
feature of the \BH procedure is that the threshold is adaptive as it
depends upon the data $\bm y$. Roughly speaking, this threshold is
high when there are few discoveries to be made and low when there are
many.

Interestingly, the acceptance of the FDR as a valid error measure has
been slow coming, and we have learned that the FDR criterion initially
met much resistance. Among other things, researchers questioned
whether the FDR is the right quantity to control as opposed to more
traditional measures such as the familywise error rate (FWER), and
even if it were, they asked whether among all FDR controlling
procedures, the \BH procedure is powerful enough. Today, we do not
need to argue that this step-up procedure is a useful tool for
addressing multiple comparison problems, as both the FDR concept and
this method have gained enormous popularity in certain fields of
science; for instance, they have influenced the practice of genomic
research in a very concrete fashion. The point we wish to make is,
however, different: as we discuss next, if we look at the multiple
testing problem from a different point of view, namely, from that of
estimation, then FDR becomes in some sense the right notion to
control, and naturally appears as a valid error measure.

% By now, both the FDR concept and the Benjamini-Hochberg step-up
% procedure have gained enormous popularity in certain fields of science
% and have, for instance, influenced the practice of medical research in
% a very concrete fashion.

Consider estimating $\bm \beta$ from the same data $\bm y \sim
\mathcal{N}(\bm \beta, \sigma^2 \bm{I}_p)$ and suppose we have reasons
to believe that the vector of means is sparse in the sense that most
of the coordinates of $\bm \beta$ may be zero or close to zero, but
have otherwise no idea about the number of `significant' means. It is
well known that under sparsity constraints, thresholding rules can far
outperform the maximum likelihood estimate (MLE). A key issue is thus
how one should determine an appropriate threshold.  Inspired by the
adaptivity of \BH, Abramovich and Benjamini \cite{fdrthreshold}
suggested estimating the mean sequence by the following {\em
  testimation} procedure:\footnote{See \cite{adke1987two} for the use
  of this word.} use \BH to select which coordinates are worth
estimating via the MLE and which do not and can be set to
zero. Formally, set $0 < q < 1$ and define the FDR estimate as
\begin{equation}
\label{eq:FDRThresh}
\hat{\beta}_i = 
\begin{cases}
  y_i, \quad & |y_i| \ge \hat{t}_{\textnormal{\tiny FDR}},\\
  0, \quad & \text{otherwise}.
\end{cases}
\end{equation}
The idea behind the FDR-thresholding procedure is to automatically
adapt to the unknown sparsity level of the sequence of means under
study. Now a remarkably insightful article \cite{ABDJ} published ten
years ago rigorously established that this way of thinking is
fundamentally correct in the following sense: if one chooses a
constant $q \in (0, 1/2]$,
% \wjs{In ABDJ, the assumption of $q_p$ is that it has a limit in $[0, 1/2]$ and $q_p \gtrsim 1/\log p$. Indeed, our SLOPE also allows a slowly vanishing $q$.} 
then the FDR estimate is asymptotically minimax over the class of
$k$-sparse signals as long as $k$ is neither too small nor too large.
More precisely, take any $\bm \beta \in \R^p$ with a number $k$ of
nonzero coordinates obeying $\log^5 p \le k \le p^{1-\delta}$ for any
constant $\delta > 0$. Then as $p \goto \infty$, it holds that
\begin{equation}
  \label{eq:ABDJ}
  \text{MSE} =  \E \|\hat{\bm \beta} - \bm \beta\|^2 \le (1+o(1)) \, 2\sigma^2 k \log(p/k). 
\end{equation}
It can be shown that the right-hand side is the asymptotic minimax
risk over the class of $k$-sparse signals (\cite{ABDJ} provides other asymptotic minimax results for $\ell_p$
balls) and, therefore, there is a sense in which the FDR estimate
asymptotically achieves the best possible mean-square error
(MSE). This is remarkable because the FDR estimate is not given any
information about the sparsity level $k$ and no matter this value in
the stated range, the estimate will be of high quality. To a certain
extent, the FDR criterion strikes the perfect balance between bias and
variance. Pick a higher threshold/or a more conservative testing
procedure and the bias will increase resulting in a loss of
minimaxity. Pick a lower threshold/or use a more liberal procedure and
the variance will increase causing a similar outcome. Thus we see that
the FDR criterion provides a fundamentally correct answer to an
estimation problem with squared loss, which is admittedly far from
being a pure multiple testing problem.

For the sake of completeness, we emphasize that the FDR thresholding
estimate happens to be very close to penalized estimation procedures
proposed earlier in the literature, which seek to regularize the
maximum likelihood by adding a penalty term of the form
\begin{equation}
\label{eq:penMLE}
\underset{\bm b}{\mbox{argmin}} ~ \|\bm y - \bm b\|_2^2 + \sigma^2
\operatorname{Pen}(\|\bm b\|_0),
\end{equation}
where $\operatorname{Pen}(k) = 2k \log(p/k)$ see \cite{FosterStine}
and \cite{BirgeMassart,TibshiraniKnight} for related ideas. In fact,
\cite{ABDJ} begins by considering the penalized MLE with
\[
\operatorname{Pen}(k) = \sum_{i \le k} (\lambbh_i)^2  = (1+o(1))\, 2k
\log(p/k),
\]
which is different from the FDR thresholding estimate, and shown to
enjoy asymptotic minimaxity under the restrictions on the sparsity
levels listed above. In a second step, \cite{ABDJ} argues that the FDR
thresholding estimate is sufficiently close to this penalized MLE so
that the estimation properties carry over.

% This penalized FDR thresholding is different from the step-up FDR thresholding in that it hard-thresholds at a slightly different level. Then \cite{ABDJ} shows (i) this penalized FDR thresholding enjoys minimaxity under the sparsity assumption, and (ii) the step-up FDR thresholding and the step-down FDR thresholding, which sandwich the penalized FDR thresholding, are uniformly close in the numbers of rejections. Hence, the asymptotic minimaxity conclusions for penalized rule can be seamlessly carried over to the rest two. Going through the proof, the minimaxity is mainly due to a delicate trade-off, provided by the control of FDR, between biases and variances resulting from estimation.

\subsection{SLOPE}
\label{sec:extent-linea-model}

Our aim in this paper is to extend the link between estimation and
testing by showing that a procedure originally aimed at controlling
the FDR in variable selection problems enjoys optimal estimation
properties. We work with a linear model, which is far more general
than the orthogonal sequence model discussed up until this point;
here, we observe an $n$-dimensional response vector obeying
% Naturally, it is interesting to see an analog of \BH in the setting of linear models, which capture far more real life scenarios than the Gaussian sequence model. This is, to a large extent, equivalent to extending \BH to arbitrary correlations, where provable FDR control remains challenging. Imagine that an $n$-dimensional response $\bm y$ is generated by a linear model of the form
\begin{equation}\label{eq:linear_model}
\bm y = \bm X \bm\beta + \bm z,
\end{equation}
where $\bm X \in \R^{n \times p}$ is a design matrix, $\bm\beta \in
\R^p$ is a vector of regression coefficients and $\bm z \sim
\mathcal{N}(\bm 0, \sigma^2 \bm{I}_n)$ is an error term.  

On the testing side, finding finite sample procedures that would test
the $p$ hypotheses $H_j: \beta_j = 0$ while controlling the FDR---or
other measures of type-I errors---remains a challenging topic. When $p
\le n$ and the design $\bm X$ has full column rank, this is equivalent
to testing a vector of means under arbitrary correlations since the
model is equivalent to $\hat{\bm{\beta}}_{\text{LS}} \sim
\mathcal{N}(\bm \beta, \sigma^2 (\bm{X}'\bm{X})^{-1})$
($\hat{\bm{\beta}}_{\text{LS}}$ is the least-squares
estimate). Applying \BH procedure to the least-squares estimate (1) is
not known to control the FDR (the positive regression dependency
\cite{prds} does not hold here), and (2) suffers from high variability
in false discovery proportions due to correlations
\cite{slope}. Having said this, we are aware of recent significant
progress on this problem including the development of the knockoff
filter \cite{Knockoffs}, which is a powerful FDR controlling method
working when $p \le n$, and other innovative ideas
\cite{fan2012estimating,liugaussian,lockhart2012significance,ji2014rate} relying on assumptions,
which may not always hold.

On the estimation side, there are many procedures available for
fitting sparse regression models and the most widely used is the Lasso
\cite{Tibs96}. When the design is orthogonal, the Lasso simply applies
the same soft-thresholding rule to all the coordinates of the
least-squares estimates. This is equivalent to comparing all the
$p$-values to a {\em fixed} threshold. In the spirit of the adaptive
\BH procedure, \cite{slope} proposed a new fitting strategy called
{SLOPE}, a short-hand for Sorted L-One Penalized Estimation: fix a
nonincreasing sequence $\lambda_1 \ge \lambda_2 \ge \cdots \ge
\lambda_p \ge 0$ not all vanishing; then SLOPE is the solution to
\begin{equation}\label{eq:slope_lambda}
  \underset{\bm b}{\mbox{minimize}} \quad  \frac12 \|\bm y - \bm X \bm b\|^2 + \lambda_1|b|_{(1)} + \lambda_2|b|_{(2)} + \cdots + \lambda_p|b|_{(p)}, 
\end{equation}
where $|b|_{(1)} \ge |b|_{(2)} \ge \cdots \ge |b|_{(p)}$ are the order
statistics of $|b_1|, |b_2|, \ldots, |b_p|$. The regularization is a
{\em sorted $\ell_1$ norm}, which penalizes coefficients whose
estimate is larger more heavily than those whose estimate is smaller.
This reminds us of the fact that in multiple testing procedures,
larger values of the test statistics are compared with higher
thresholds. In particular, recall that \BH compares $|y|_{(i)}/\sigma$
with $\lambbh_i = \Phi^{-1}(1-iq/2p)$---the $(1-iq/2p)$th quantile of
a standard normal (for information, the sequence $\bm \lambbh$ shall
play a crucial role in the rest of this paper). SLOPE is a convex
program and \cite{slope} demonstrates an efficient solution algorithm
(the computational cost of solving a SLOPE problem is roughly the same
as that of solving the Lasso).

To gain some insights about SLOPE, it is helpful to consider the
orthogonal case, which we can take to be the identity without loss of
generality.  When $\bm X = \bm{I}_p$, the SLOPE estimate is the
solution to
\begin{equation}\label{eq:prox_intro}
  \slope{\bm y} \triangleq \underset{\bm b}{\mbox{argmin}} ~\textstyle{\frac12} \|\bm y - \bm b\|^2 + \lambda_1|b|_{(1)}  + \cdots + \lambda_p|b|_{(p)}; 
\end{equation} 
in the literature on optimization, this solution is called the prox to
the sorted $\ell_1$ norm evaluated at $\bm y$, hence the notation in
the left-hand side. (In the case of a general orthogonal design in
which $\bm{X}' \bm X = \bm{I}_p$, the SLOPE solution is
$\slope{\bm{X}'\bm y}$.)  Suppose the observations are nonnegative and already ordered, i.e.~$y_1 \ge y_2 \ge \cdots \ge y_p \ge
0$.\footnote{For arbitrary data, the solution can be obtained as
  follows: let $\bm P$ be a permutation that sorts the magnitudes
  $|\bm y|$ in a non-increasing fashion. Then $\slope{\bm y} =
  \sgn({\bm y}) \odot \bm P^{-1} \slope{\bm P |\bm y|}$, where $\odot$
  is componentwise multiplication.  In words, we can replace the
  observations by their sorted magnitudes, solve the problem and,
  finally, undo the ordering and restore the signs.} Then by
\cite[Proposition 2.2]{slope} SLOPE can be recast as the solution to
\begin{equation}
  \label{eq:slopeQP}
\begin{array}{ll} 
  \text{minimize} & \quad \textstyle{\frac{1}{2}} \|\bm y - \bm \lambda - \bm b\|^2 = \textstyle{\frac{1}{2}} \sum_i (y_i - \lambda_i - b_i)^2 \\
  \text{subject to} &  \quad b_1 \ge b_2 \ge \cdots \ge b_p \ge 0
\end{array} 
\end{equation}
so that it is equivalent to solving an isotonic regression
  problem with data $\bm y - \bm\lambda$. Hence, methods like the
  pool adjacent violators algorithm (PAVA) \cite{kruskal64,barlow72}
  are directly applicable. Further, two observations are in order:
the first is that the fitted values have the same signs and ranks as
the original observations; for any pair $(i,j)$, $y_i \ge y_j$ implies
that $\hat \beta_i \ge \hat \beta_j$.  The second is that the fitted
values are as close as possible to the shrunken observations $y_i -
\lambda_i$ under the ordering constraint. Hence, SLOPE is a sort of
soft-thresholding estimate in which the amount of thresholding is data
dependent and such that the original ordering is preserved.

To emphasize the similarities with the \BH procedure, assume that we
work with $\lambda_i = \sigma \cdot \lambbh_i$ and that we use SLOPE
as a multiple testing procedure rejecting $H_i : \beta_i = 0$ if and
only if $\hat{\beta}_i \neq 0$. Then this procedure rejects all the
hypotheses the \BH step-down procedure would reject, and accepts all
those the step-up procedure would accept. Under independence,
i.e.~$\bm y \sim \mathcal{N}(\bm \beta, \sigma^2 \bm{I}_p)$, SLOPE
controls the FDR \cite{slope}, namely,
$\operatorname{FDR}(\text{SLOPE}) \le {qp_0}/{p}$, where again $p_0$
is the number of nulls, i.e.~of vanishing means.

Figure \ref{fig:adapt} displays SLOPE estimates for two distinct data
sets, with one set containing many more stronger signals than the
other. We see that SLOPE sets a lower threshold of significance when
there is a larger number of strong signals. We can also see that SLOPE
tends to shrink less as observations decrease in magnitude.  In
summary, SLOPE encourages sparsity just as the Lasso, but unlike the
Lasso its degree of penalization is adaptive to the unknown sparsity
level.
\begin{figure}[!htbp]
\centering
\hfill
\begin{subfigure}[b]{0.49\textwidth}
\centering
\includegraphics[width=\textwidth]{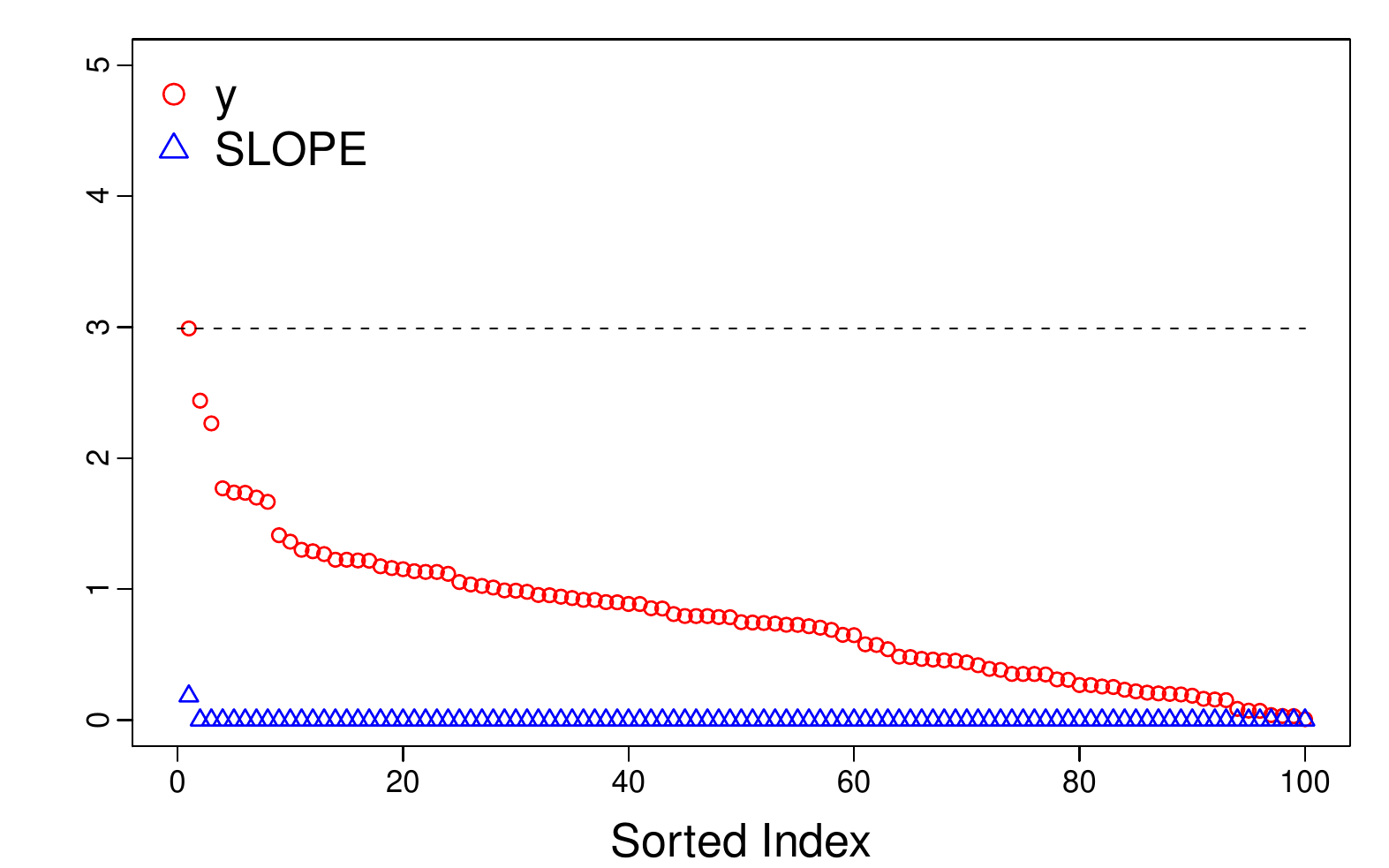}
\caption{Weak signals.}
\label{fig:adapt_k5}
\end{subfigure}
%\hfill
\begin{subfigure}[b]{0.49\textwidth}
\centering
\includegraphics[width=\textwidth]{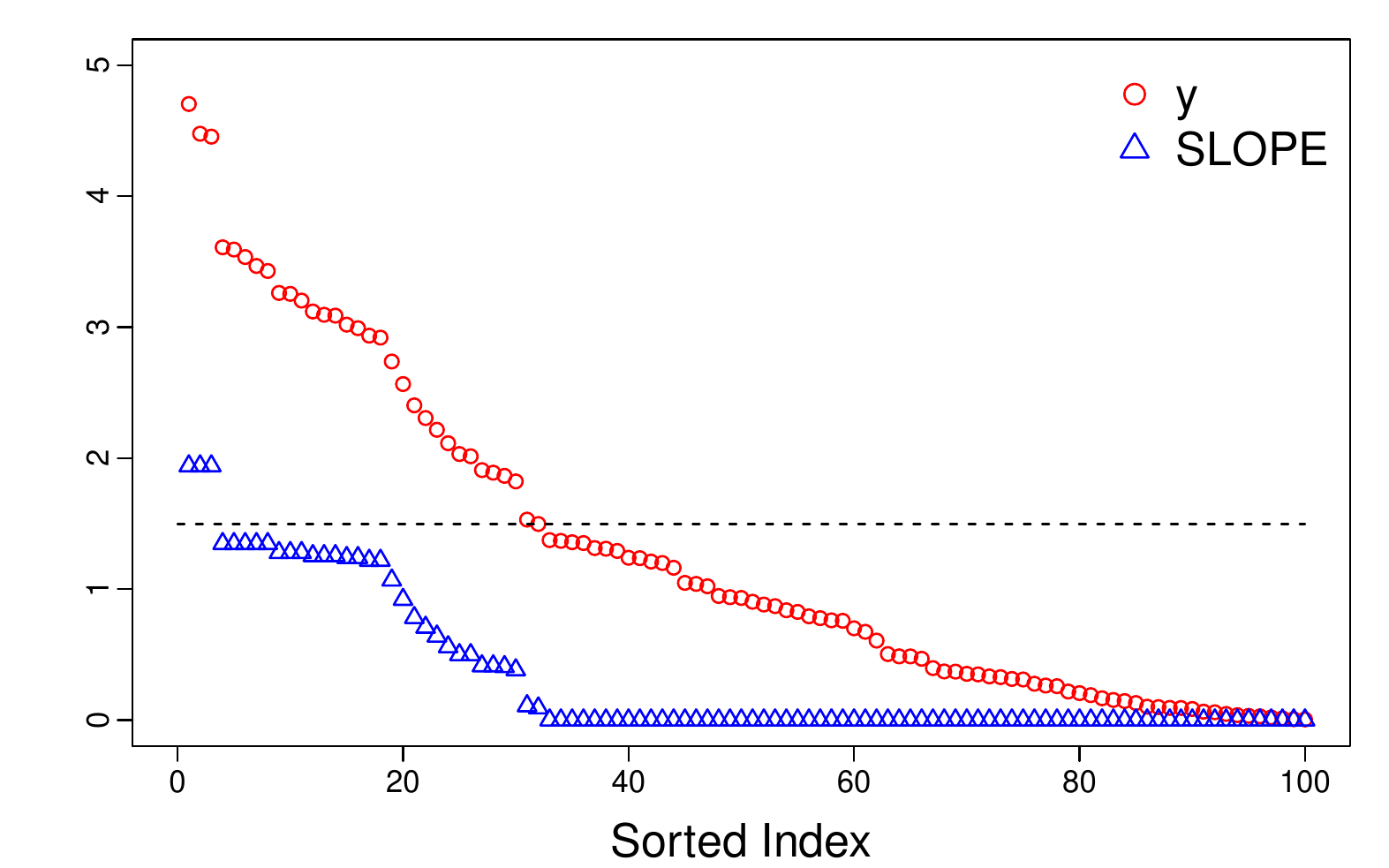}
\caption{Strong signals.}
\label{fig:adapt_k15}
\end{subfigure}
\hfill
\caption{Illustrative examples of original observations and SLOPE estimates with the identity design. All observations below the threshold
  indicated by the dotted line are set to zero; this threshold is data
  dependent.}
\label{fig:adapt}
\end{figure}

\subsection{Orthogonal designs}
\label{sec:slope-orth-design-intro}

We now turn to estimation properties of SLOPE and begin by considering
orthogonal designs.  Multiplying both sides of \eqref{eq:linear_model}
by $\bm X'$ gives the statistically equivalent Gaussian sequence
model,
\[
\bm y = \bm\beta + \bm{z},
\]
where $\bm z \sim \mathcal{N}(\bm 0, \sigma^2\bm{I}_p)$.  Estimating a
sparse mean vector from Gaussian data is a well-studied problem with a
long line of contributions, see
\cite{bickel1981minimax,donoho1994a,foster1994,BirgeMassart,cai2009,
  johnstone} for example.  Among other things, we have already mentioned
that the asymptotic risk over sparse signals is known: consider a
sequence of problems in which $p\goto \infty$ and $k/p \goto 0$, then
\[
R_p(k) \triangleq \inf_{\hat{\bm\beta}} \,\, \sup_{\|\bm\beta\|_0 \le
  k} \E \| \hat{\bm\beta} - \bm\beta\|^2 = (1+o(1))\, 2\sigma^2
k\log(p/k),
\]
where the infimum is taken over all measurable estimators, see
\cite{donoho1994b} and \cite{johnstone}. Furthermore, both soft or
hard-thresholding at the level of $\sigma\sqrt{2\log(p/k)}$ are
asymptotically minimax.  Such estimates require knowledge of the
sparsity level ahead of time, which is not realistic. Our first result is that SLOPE also achieves asymptotic minimaxity \textit{without} this knowledge.
\begin{theorem}\label{thm:orth_main_general_intro}
  Let $\bm X$ be orthogonal and assume that $p \goto \infty$ with $k/p
  \goto 0$. Fix $0 < q < 1$. Then SLOPE with $\lambda_i = \sigma \cdot
  \Phi^{-1}(1-iq/2p) = \sigma \cdot \lambbh_i$ obeys
\begin{equation}
\label{eq:orth_main_general_intro}
\sup_{\|\bm\beta\|_0 \le k}\E \| \hat{\bm\beta}_{\textnormal{\tiny SLOPE}} - \bm\beta \|^2 = (1+o(1))\, 2 \sigma^2k\log(p/k).
\end{equation}
\end{theorem}
Hence, no matter how we select the parameter $q$ controlling the FDR
level in the range $(0,1)$, we get asymptotic minimaxity (in practice
we would probably stick to values of $q$ in the range $[0.05, 0.30]$).
There are notable differences with the result from \cite{ABDJ} we
discussed earlier. First, recall that to achieve minimaxity in that
work, the nominal FDR level needs to obey $q \le 1/2$ (the MSE is
larger otherwise) and the sparsity level is required to obey $\log^5 p
\le k \le p^{1-\delta}$ for a constant $\delta > 0$, i.e.~the signal
cannot be too sparse nor too dense. The lower bound on sparsity has
been improved to $\log^{4.5} p$ \cite{wuzhou}.  In contrast, there are
no restrictions of this nature in
Theorem~\ref{thm:orth_main_general_intro}; this has to do with the
fact that SLOPE is a continuous procedure whereas FDR thresholding is
highly discontinuous; small perturbations in the data can cause the
FDR thresholding estimates to jump. This idea may also be found in the recent work \cite{jiangadaptive} in which the authors prove that some smooth-thresholding procedures uniformly achieve asymptotic minimaxity under the same assumptions as in Theorem~\ref{thm:orth_main_general_intro}. They also establish some optimality results for these thresholding rules at a fixed $\bm\beta$. Second, SLOPE effortlessly extends to
linear models while it is not clear how one would extend FDR
thresholding ideas in a computationally tractable fashion.

%% \footnote{Th paper \cite{jiangadaptive} also develops a strong notion of adaptive ratio optimality to analyze performance at a fixed $\bm\beta$. \ejc{Not very clear}.}

One can ask which vectors $\bm \beta$ achieve the equality in
\eqref{eq:orth_main_general_intro}, and it is not very hard to see
that equality holds if the $k$ nonzero entries of $\bm \beta$ are very
large. Suppose for simplicity that $\beta_1 \gg \beta_2 \gg \cdots \gg
\beta_k \gg 1$ and that $\beta_{k+1} = \cdots = \beta_p = 0$. Spacing
the nonzero coefficients sufficiently far apart will insure that $y_j
- \lambda_j$, $1 \le j \le k$, is nonincreasing with high probability
so that the SLOPE estimate is obtained by rank-dependent
soft-thresholding:
\[
\hat{\beta}_{\textnormal{\tiny
    SLOPE},j} = y_j - \sigma \lambbh_j. 
\]
Informally, since the mean-square error is the sum of the squared bias
and variance, this gives
\[
\E (\hat{\beta}_{\textnormal{\tiny
    SLOPE},j} - \beta_j)^2 \approx \sigma^2 \cdot ((\lambbh_j)^2+1). 
\]
Since $\sum_{1 \le j \le k} (\lambbh_j)^2 = (1+o(1))\,
2k\log(p/k)$,\footnote{This relation follows from $\Phi^{-1}(1 -c) =
  (1 + o(1)) \sqrt{2\log(1/c)}$ when $c \searrow 0$ and applying
  Stirling's approximation.} summing this approximation over the first
$k$ coordinates gives
\[
\E \sum_{1 \le j \le k} (\hat{\beta}_{\textnormal{\tiny SLOPE},j} -
\beta_j)^2 \approx \sigma^2 \cdot \Bigl(k + \sum_{1 \le j \le k}
(\lambbh_j)^2\Bigr) = (1+o(1))\, 2 \sigma^2 k \log(p/k),
\]
where the last inequality follows from the condition $k/p \goto
0$. Theorem \ref{thm:orth_main_general_intro} states that in
comparison, the $p-k$ vanishing means contribute a negligible MSE.

We pause here to observe that if one hopes SLOPE with weights
$\lambda_j$ to be minimax, then they will need to satisfy
\[
\sum_{j = 1}^k \lambda_j^2 = (1+o(1)) \, 2k\log(p/k) 
\]
for all $k$ in the stated range. Since $\lambda_j^2 = \sum_{i =
    1}^j \lambda_i^2 - \sum_{i = 1}^{j-1} \lambda_i^2$, we have that
  $\lambda_j^2$ is roughly the derivative of $f(x) = 2x \log(p/x)$ at
  $x = j$ yielding $\lambda_j^2 \approx f'(j) = 2\log p - 2\log j -
  2$, or
\[
\lambda_j \approx \sqrt{2 \log(p/j)} \approx \Phi^{-1}(1 - jq/2p).
\]
% where we recall the second order approximation of normal quantiles
% \[
% \Phi^{-1}(1 - jq/2p) \sim \sqrt{2\log \frac{p}{jq\sqrt{\log (p/jq)}} }.
% \]
As a remark, all our results---e.g.~Theorems
\ref{thm:orth_main_general_intro} and
\ref{thm:gauss_minimax}---continue to hold if we replace
$\lambda_j^\text{BH}(q)$ with $\sqrt{2\log(p/j)}$.
% \wjs{(1) We might briefly comment that $\lambbh_i$ is just a little
%   bit above the $i$th order statistic of $\bm z$. (2) It is not clear
%   that these critical values are `being on the edge' solely by the
%   fact that BH controls FDR. But, the minimax property induced by
%   $\bm\lambbh$ suggests this choice is really \textit{sharp}. So I
%   think it is helpful to say that this magnitude is `being on the
%   edge' once the minimax result is introduced.}

We speculate that Theorem \ref{thm:orth_main_general_intro}---and to some extent Theorem \ref{thm:gauss_minimax} below---extend to other loss functions.  For instance, from the proofs of Theorem \ref{thm:orth_main_general_intro} we believe that for $r \ge 1$,
\[
\sup_{\|\bm\beta\|_0 \le k} \E \| \hat{\bm\beta}_{\textnormal{\tiny
    SLOPE}} - \bm\beta \|_r^r = (1+o(1)) \cdot k \cdot 
\bigl(2\sigma^2\log(p/k)\bigr)^{r/2}
\]
holds. Furthermore, examining the proof of
Theorem \ref{thm:orth_main_general_intro} reveals that for all $k$ not
necessarily obeying $k/p \goto 0$ (e.g.~$k = p/2$),
\[
\frac{\sup_{\|\bm\beta\|_0 \le k} \E \|
  \hat{\bm\beta}_{\textnormal{\tiny SLOPE}} - \bm\beta \|^2}{R_p(k)}
\le C(q), 
\]
where $C(q)$ is a positive numerical constant that only depends on
$q$.

\subsection{Random designs}
\label{sec:slope-under-general-intro}

We are interested in getting results for sparse regression that would
be just as sharp and precise as those presented in the orthogonal
case. In order to achieve this, we assume a tractable model in which
$\bm X$ is a Gaussian random design with $X_{ij}$
i.i.d.~$\mathcal{N}(0, 1/n)$ so that the columns of $\bm X$ have just
about unit norm. Random designs allow to analyze fine structures of
the models of interest with tools from random matrix theory and large
deviation theory, and are very popular for analyzing regression
methods in the statistics literature. An incomplete list of works
working with Gaussian designs would include
\cite{brown2002,baraud2002,birge2004,bunea2007,wainwright2009, lassorisk, donohorobust}. On the one hand, Gaussian designs are
amenable to analysis while on the other, they capture some of the
features one would encounter in real applications.

% Random designs allow to analyze fine structure of the models
% of interest with tools from random matrix theory and large deviation
% theory, thus are popular in analyzing regression in statistics
% literature, see
% \cite{brown2002,baraud2002,birge2004,bunea2007,kulik2009} for a far
% incomplete list of work utilizing random designs. While significantly
% simplifying life, random design models also capture a lot of sense in
% real applications. For instance, in genome-wide association studies
% design matrices behave very closely to random designs \cite{slope}.

% To compensate the variance inflation introduced by between-column correlations, which scale $O(1/\sqrt{n})$, we use $\bm\lambda = \sigma\lambe = \sigma(1+\epsilon)\bm\lambbh$, where the inflation factor $\epsilon \ge 0$ is small. Fixing any nominal level $q \in (0, 1)$, now we turn to state our main result for SLOPE under this random design. 

To avoid any ambiguity, the theorem below considers a sequence of
problems indexed by $(k_j, n_j, p_j)$, where the number of variables
$p_j \goto \infty$, $k_j/p_j \goto 0$ and $(k_j\log p_j)/n_j \goto
0$. From now on, we shall omit the subscript.
\begin{theorem}\label{thm:gauss_minimax}
  Fix $0 < q < 1$ and set
  $\bm \lambda = \sigma (1+\epsilon) \bm \lambbh(q)$ for some
  arbitrary constant $0 < \epsilon < 1$.  Suppose $k/p \goto 0$ and
  $(k\log p)/n \goto 0$. Then
\begin{equation}
\label{eq:gauss_minimax} 
\sup_{\|\bm\beta\|_0 \le k} \P \left( \frac{\|
    \hat{\bm\beta}_{\textnormal{\tiny SLOPE}} - \bm\beta
    \|^2}{2\sigma^2 k \log(p/k)} > 1 + 3\epsilon \right)
\longrightarrow 0.
\end{equation}
\end{theorem}

For information, it is known that under some regularity conditions on
the design \cite{raskutti2011, verzelen2012}, the minimax risk is on
the order of $O(\sigma^2 k\log(p/k))$, without a tight matching in the
lower and upper bounds.  Against this, our main result states that
SLOPE, which does not use any information about the sparsity level,
achieves a squared loss bounded by $(1+o(1)) \, 2\sigma^2k\log(p/k)$
with large probability. This is the best any procedure can do as we
show next.
\begin{theorem}\label{thm:lower_beta} 
  Under the assumptions of Theorem \ref{thm:gauss_minimax}, for any $
  \epsilon >0$, we have
\begin{equation}\nonumber%%\label{eq:lowerbound_general}
  \inf_{\hat{\bm\beta}}\sup_{\|\bm\beta\|_0 \le k}\P\left( \frac{\|\hat{\bm\beta} - \bm{\beta}\|^2}{2\sigma^2k\log(p/k)} > 1 - \epsilon \right) \longrightarrow 1.
\end{equation}
\end{theorem}
Similar results dealing with arbitrary designs can be found in the literature, compare Theorem 1 in \cite{ye2010}. However, the notable difference is that our theorem captures the exact constants in addition to the rate.

Taking Theorems~\ref{thm:gauss_minimax} and \ref{thm:lower_beta} together demonstrate that in a probabilistic sense $2\sigma^2k\log(p/k)$ is the fundamental limit for the squared loss
and that SLOPE achieves it. It is also likely that our methods would
yield corresponding bounds for the expected squared loss but this
would involve technical issues having to do with the bounding of the
loss on rare events. This being said, Theorem \ref{thm:gauss_minimax}
provides a more accurate description of the squared error than a
result in expectation since it asserts that the error is at most
$2\sigma^2 k \log(p/k)$ with high probability. The proof of this fact
presents several novel elements not found in the literature.

The condition $(k \log p)/n \goto 0$ is natural and cannot be
fundamentally sharpened. To start with, our results imply that SLOPE
perfectly recovers $\bm \beta$ in the limit of vanishing noise. In the
high-dimensional setting where $p > n$, this connects with the
literature on compressed sensing, which shows that in the noiseless
case, $n \ge 2(1+o(1)) k \log (p/k)$ Gaussian samples are necessary
for perfect recovery by $\ell_1$ methods in the regime of interest
\cite{donoho2009counting, donoho2010exponential}. Our condition is a
bit more stringent but naturally so since we are dealing with noisy
data.

We hope that it is clear that results for orthogonal designs do not
imply results for Gaussian designs because of (1) correlations between
the columns of the design and (2) the high dimensionality. Under an
orthogonal design, when there is no noise, one can recover $\bm \beta$
by just computing $\bm X' \bm y$. However, as discussed above it is
far less clear how one should do this in the high-dimensional regime
when $p \gg n$. As an aside, with noise it would be foolish to find
$\hat{\bm \beta}$ via $\slope{\bm X' \bm y}$; that is, by applying
$\bm X'$ and then pretending that we are dealing with an orthogonal
design. Such estimates turn out to have unbounded risks.

We remark that a preprint \cite{oscaranalysis} considers statistical
properties of a generalization of OSCAR \cite{oscar} that coincides
with SLOPE. The findings and results are very different from those
presented here; for instance, the selection of optimal weights $\lambda_i$ is
not discussed.

Finally, to see our main results under a slightly different light,
suppose we get a new sample $(\bm x^*, y^*)$, independent from the
`training set' $(\X, \bm y)$, obeying the linear model $y^* = \< \bm
x^*, \bm \beta\> + \sigma z^*$ with $\bm x \sim \mathcal{N}(0, n^{-1}
\bm I_p)$ and $z^* \sim \mathcal{N}(0,\sigma^2)$. Then for any
estimate $\hat{\bm \beta}$, the prediction $\hat y = \<\bm x^*,
\hat{\bm \beta}\>$ obeys
\[
\E(y^* - \hat y)^2 = n^{-1} \E \|\bm \beta - \hat{\bm \beta}\|^2 +
\sigma^2,
\]
so that, in some sense, SLOPE with BH weights actually yields the best
possible prediction.

\subsection{Back to multiple testing}

Although our emphasis is on estimation, we would nevertheless like to
briefly return to the multiple testing viewpoint. In \cite{slope,
  slopeold}, a series of experiments demonstrated empirical FDR
control whenever $\bm \beta$ is sufficiently sparse.  While this paper
does not go as far as proving that SLOPE controls the FDR in our
Gaussian setting, the ideas underlying the proof of Theorem
\ref{thm:gauss_minimax} have some implications for FDR control.  Our
discussion in this section is less formal.

% Yet, a missing link yet up
% to now is a provable FDR control for SLOPE under other designs except
% for orthogonal designs. In \cite{slope, slopeold}, a series of
% experiments under various designs demonstrate empirical FDR control
% with penalty sequences based on the \BH critical values. At this
% point, no theoretical guarantee has been provided.

Suppose we wish to keep the false discovery proportion (FDP)
$\mbox{FDP} = V/(R \vee 1) \le q$. Since the number of true
discoveries $R - V$ is at most $k$, the false discovery number $V = \{i: \beta_i = 0 \text{ and }
\hat{\beta}_{\textnormal{\tiny SLOPE},i} \neq 0\}$ {\em must} obey
\begin{equation}\label{eq:v_upper_bd}
  V \le \frac{q}{1-q} \, k.
\end{equation}
Interestingly, an intermediate result of the proof of Theorem
\ref{thm:gauss_minimax} implies that \eqref{eq:v_upper_bd} is
satisfied with probability tending to one if $k$ is sufficiently large
and $q$ is replaced by $(1+o(1))q$.
%%, since the number of false
% discoveries $V = \{i: \beta_i = 0 \text{ and }
% \hat{\beta}_{\textnormal{\tiny SLOPE},i} \neq 0\}$ satisfies
% \begin{equation}\label{eq:v_upper_bd}
%   V \le (1+o(1)) \, \frac{q}{1-q} \, k.
% \end{equation}
% To be fully specific,
% under the assumptions of Theorem \ref{thm:gauss_minimax} and
% additionally $k \goto \infty$, we can construct a sophisticated
% superset $S^\star$ of the support set of $\bm\beta$ that obeys (i)
% \[
% |S^\star| \le \frac{1+o(1)}{1-q}k,
% \]
% and (ii) $S^\star$ contains the support set of $\hat{\bm\beta}_{\textnormal{\tiny SLOPE}}$ with probability tending to one. These two properties combined yield that \eqref{eq:v_upper_bd} holds with high probability. Details can be found in Definition~\ref{def:T_star} and 
This is shown in Lemma~\ref{lm:T_contains}. Another consequence of our
analysis is that if the nonzero regression coefficients are larger
than $1.1 \, \sigma\lambbh_1(q)$ (technically, we can replace 1.1 with
any fixed number greater than one), then the true positive proportion
(the ratio between the number of true discoveries and $k$) approaches
one in probability. In this setup, we thus have FDR control in the
sense that
\[
\mbox{FDR}_{\textnormal{\tiny SLOPE}} \le (1 + o(1))q.
\] 

Figure \ref{fig:fdr_simu} demonstrates empirical FDR control at the
target level $q = 0.1$. Over 500 replicates, the averaged FDR is
$0.09$, and the averaged false discovery number $V$ is $9.4$, as
compared with $11.1$, the upper bound in \eqref{eq:v_upper_bd}. We
emphasize that \cite{slope,slopeold} also provide strong evidence that
FDR is also controlled for moderate signals.

\begin{figure}[!htbp]
\centering
\hfill
\begin{subfigure}[b]{0.49\textwidth}
\centering
\includegraphics[width=\textwidth]{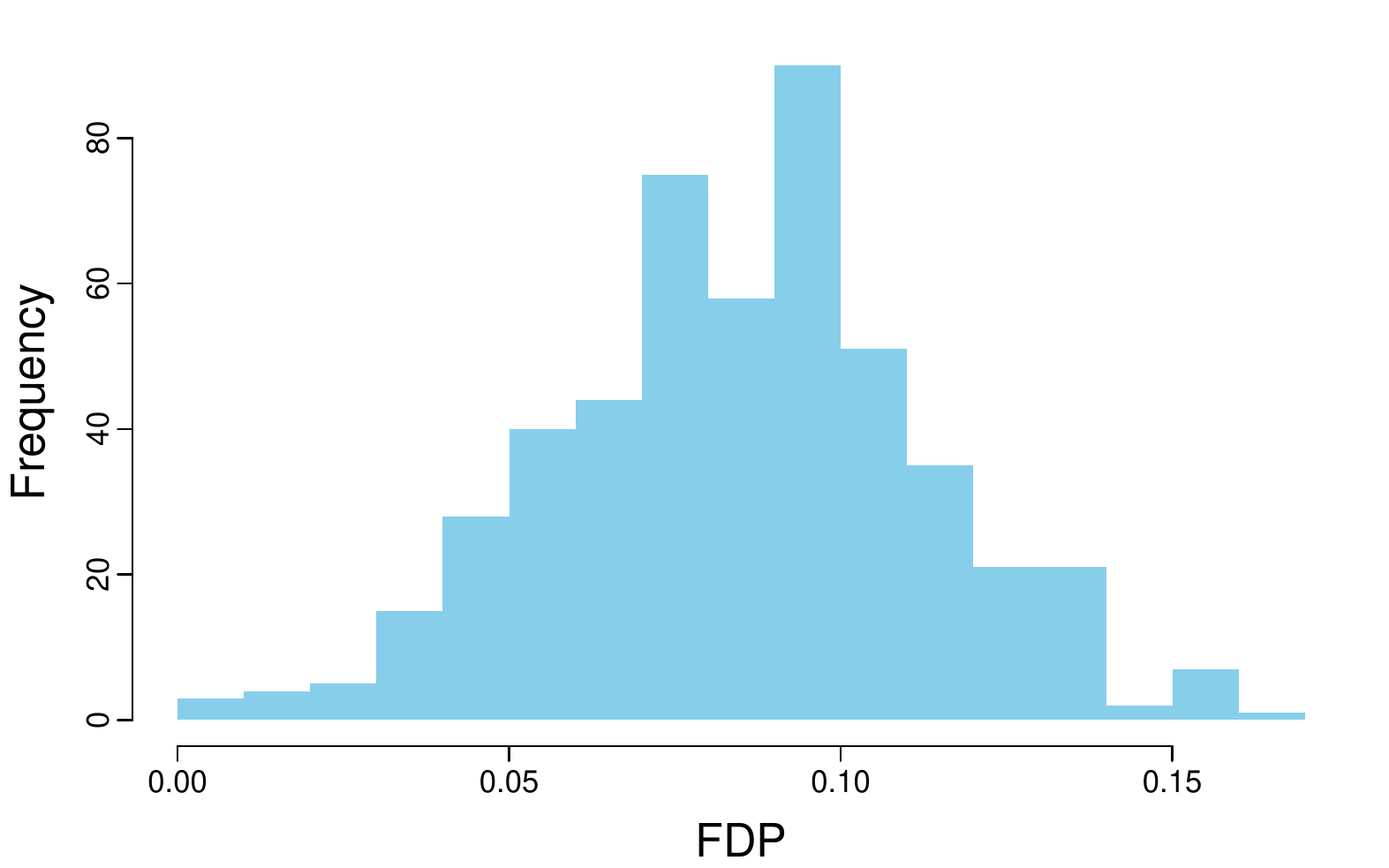}
\caption{Histogram of FDP.}
\label{fig:fdr_hist}
\end{subfigure}
%\hfill
\begin{subfigure}[b]{0.49\textwidth}
\centering
\includegraphics[width=\textwidth]{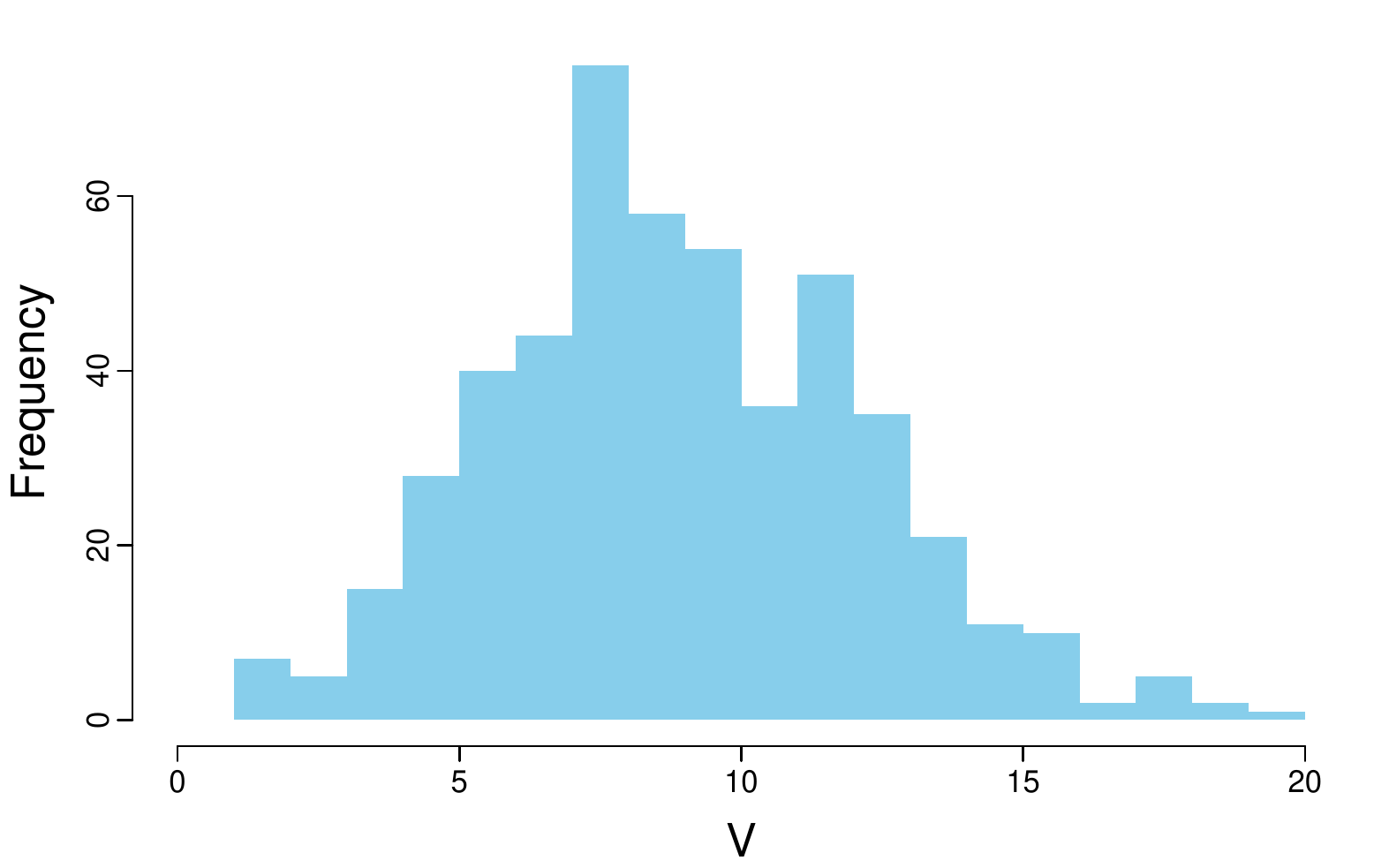}
\caption{Histogram of $V$.}
\label{fig:v_hist}
\end{subfigure}
\hfill
\caption{Gaussian design with $(n,p)=(8,000, 10,000)$ and $\sigma =
  1$. There are $k= 100$ nonzero coefficients with amplitudes
  $10\sqrt{2\log p}$. Here, the nominal level is $q = 0.1$ and
  $\bm\lambda = 1.1\bm\lambbh(0.1)$.}
\label{fig:fdr_simu}
\end{figure}

Since our paper proves that SLOPE does not make a large number of
false discoveries, the support of $\hat{\bm\beta}_{\textnormal{\tiny
    SLOPE}}$ is of small size, and thus we see that $\|\X
(\hat{\bm\beta}_{\textnormal{\tiny SLOPE}} - \bm \beta)\|^2$ is very
nearly equal to $\|\hat{\bm\beta}_{\textnormal{\tiny SLOPE}} - \bm
\beta\|^2$ since skinny Gaussian matrices are near
isometries. Therefore, we can carry our results over to the estimation
of the mean vector $\X \bm\beta$.

\begin{corollary}\label{coro:gauss_prediction_minimax}
  Under the assumptions of Theorem \ref{thm:gauss_minimax},
\[
\sup_{\|\bm\beta\|_0 \le k} \P\left( \frac{\|\bm X
    \hat{\bm\beta}_{\textnormal{\tiny SLOPE}} - \bm X
    \bm\beta\|^2}{2\sigma^2k\log(p/k)} > 1 + 3\epsilon \right) \longrightarrow
0.
\]
\end{corollary}

As before, there are matching lower bounds: for these, it
  suffices to restrict attention to estimates of the form $\hat{\bm
    \mu} = \X \hat{\bm \beta}$ since projecting any estimator
  $\hat{\bm\mu}$ onto the column space of $\X$ never increases the
  loss.
\begin{corollary}\label{cor:lowerbound_general_pred}
  Assume $k/p \goto 0$ and $p = O(n)$. Then
  \begin{equation}\nonumber
    \inf_{\hat{\bm\beta}} \sup_{\|\bm\beta\|_0 \le k}\P\left( \frac{\|\bm X \hat{\bm\beta} - \bm X \bm\beta\|^2}{2\sigma^2k\log(p/k)} > 1 - \epsilon \right) \longrightarrow 1
\end{equation}
\end{corollary}
Again, SLOPE is optimal for estimating the mean response, and achieves
an estimation error which is the same as that holding for the
regression coefficients themselves. % Furthermore, if we were informed
% about the support of $\bm\beta$, then regressing $\bm y$ onto the
% variables known to be in the model would have a MSE equal to $\sigma^2
% k$. The additional factor $2\log(p/k)$ may be interpreted as the
% statistical price paid for lacking such precious information. The
% maximum of this factor is at $k = 1$, where it takes the value $2\log
% p$; this is the risk inflation factor introduced by Foster and George
% \cite{foster1994}.

% \ejc{This should probably be part of a short discussion section
%   somewhere.}  {\color{red} We hope that follow up studies which give
%   concrete results for correlated random designs. For the time being,
%   we believe that these results would hold to \iid sub-Gaussian random
%   design\footnote{In the case of unknown $\sigma^2$, an upper bound
%     $\sigma'^2 \ge \sigma^2$ can often be easily obtained. Then, SLOPE
%     with penalty $\bm\lambda = \sigma'(1+\epsilon)\bm\lambbh$ obeys
%     all the previous theorems, provided that $2\sigma^2k\log(p/k)$ is
%     replaced by $2\sigma'^2k\log(p/k)$.}.}

\subsection{Organization and notations}
\label{sec:organization}

In the rest of the paper, we briefly explore possible alternatives to
SLOPE in Section \ref{sec:alternatives-slope}. Section
\ref{sec:slope-with-orth} concerns the estimation properties of SLOPE
under orthogonal designs and proves Theorem
\ref{thm:orth_main_general_intro}. We then turn to study SLOPE under
Gaussian random designs in Section \ref{sec:slope-under-random-all},
where both Theorem \ref{thm:gauss_minimax} and Corollary
\ref{coro:gauss_prediction_minimax} are proved. Last, we prove
corresponding lower bounds in Section \ref{sec:lowerbound}, including
Theorem \ref{thm:lower_beta}. Corollary
\ref{cor:lowerbound_general_pred} and auxiliary results are proved in
the Appendix.

% {\color{red} Hereafter, we denote $f_j \gtrsim g_j$ if both $f_j \ge C
%   g_j$ for some constant $C > 0$, and $f_j \asymp g_j$ if both $f_j
%   \gtrsim g_j$ and $g_j \gtrsim f_j$ are satisfied. }

Recall that $p, n, k$ are positive integers with $p \goto \infty$, but
not necessarily so for $k$.  We use $\overline{S}$ for the complement
of $S$. For any vector $\bm a$, define the support of $\bm a$ as
$\supp{\bm a}\triangleq \{i: a_i \ne 0\}$.  A bold-faced $\bm\lambda$
denotes a general vector obeying
$\lambda_1 \ge \lambda_2 \ge \cdots \ge \lambda_p \ge 0$, with at
least one strict inequality. For any integer $0 < m < p$,
$\bm\lambda^{[m]} \triangleq (\lambda_1, \ldots, \lambda_m)$ and
$\bm\lambda^{-[m]} \triangleq (\lambda_{m+1}, \ldots, \lambda_p)$. We
write $\lambe$ (the superscript is omitted to save space) for the
$\epsilon$-inflated \BH critical values,
\[
\lambda_{\epsilon, i} = (1 + \epsilon)\lambbh_i = (1 + \epsilon)\Phi^{-1}\left(1 - iq/(2p)  \right).
\]
Last and for simplicity, $\hat{\bm\beta}$ is the SLOPE estimate,
unless specified otherwise.

%%% Local Variables:
%%% mode: latex
%%% TeX-master: "paper"
%%% End:

\section{Alternatives to SLOPE?}
\label{sec:alternatives-slope}

It is natural to wonder whether there are other estimators, which can
potentially match the theoretical performance of SLOPE for sparse
regression. Although getting an answer is beyond the scope of this
paper, we pause to consider a few alternatives.

\subsection{Other $\ell_1$ penalized methods}

The Lasso,
\begin{equation}\nonumber%%%\label{eq:lasso}
  \underset{\bm b}{\mbox{minimize}} \quad \frac12 \|\bm y - \bm X \bm b\|^2 + \lambda \|\bm b\|_1, 
\end{equation}
serves as a building block for a lot of sparse estimation
procedures. If $\lambda$ is chosen non adaptively, then a value equal
to $(1-c) \cdot \sigma \sqrt{2\log p}$ for $0 < c < 1$ would cause a
large number of false discoveries even under the global null and, consequently, the risk when
estimating sparse signals would be high. This phenomenon can
already be seen in the orthogonal case
\cite{foster1994,johnstone}. This means that if we choose $\lambda$ in
a non-adaptive fashion then we would need to select $\lambda \ge
\sigma \sqrt{2\log p}$. Under the assumptions of Theorem \ref{thm:gauss_minimax} and setting
$\lambda = (1+c)\cdot \sigma\sqrt{2\log p}$ for an arbitrary positive
constant $c$ gives
\begin{equation}\label{eq:lasso_risk}
\sup_{\|\bm\beta\|_0 \le k}\P \left( \frac{\|\hat{\bm\beta}_{\textnormal{\tiny Lasso}} - \bm\beta\|^2}{2 \sigma^2 k\log p} > 1 \right) \goto 1.
\end{equation}
The proof is in Appendix \ref{sec:proofs-sect-refs}.  Hence the risk
inflation does not decreases as the sparsity level $k$ increases,
whereas it does for SLOPE. Note that when $p = n$ and $k = p^{1-\delta}$,
\[
\frac{2\sigma^2k\log p}{2\sigma^2 k\log(p/k)} \goto \frac{1}{\delta}. 
\]
The reason why the Lasso is suboptimal is that the bias is too large
(the fitted coefficients are shrunk too much towards zero). All in
all, by our earlier considerations and by letting $\delta \goto 0$
above, we conclude that no matter how we pick $\lambda$
non-adaptively, the ratio
\[
\frac{\mbox{max risk of Lasso}}{\mbox{max risk of SLOPE}} \goto \infty
\]
in the worst case over $k$.

Figures \ref{fig:lasso_slope} and \ref{fig:lasso_slope2} compare SLOPE with Lasso estimates for both strong and moderate signals. SLOPE is more accurate than the
Lasso in both cases, and the comparative advantage increases as $k$
gets larger. This is consistent with the reasoning that SLOPE has a
lower bias when $k$ gets larger.
% %%%%
% \begin{figure}[!htbp]
% \centering
% \begin{subfigure}[b]{0.48\textwidth}
% \centering
% \includegraphics[width=\textwidth]{lasso_slope_strong.pdf}
% \caption{Strong signals.}
% \label{fig:lasso_slope}
% \end{subfigure}
% \hfill
% \begin{subfigure}[b]{0.48\textwidth}
% \centering
% \includegraphics[width=\textwidth]{lasso_slope_moderate.pdf}
% \caption{Moderate signals.}
% \label{fig:lasso_slope2}
% \end{subfigure}
% \caption{Gaussian design with $(n,p) = (500, 1000)$ and $\sigma =
%   1$. The risk $\E \|\hat{\bm\beta} - \bm\beta\|^2$ is averaged over
%   100 replicates. SLOPE uses $\bm\lambda = \bm\lambbh(q)$ and Lasso
%   uses $\lambda=\lambbh_{1}(q)$ with level $q = 0.05$. In (a), the
%   components have magnitude $10\lambbh_1$; in (b), the magnitudes are
%   set to $0.8\lambbh_1$. }
% \label{fig:intro_lasso}
% \end{figure}

Of course, one might want to select $\lambda$ in a data-dependent
manner, perhaps by cross-validation (see next section), or by
attempting to control a type-I error such as the FDR. For instance, we
could travel on the Lasso path and stop `at some point'. Some recent
procedures such as \cite{lockhart2012significance} make very strong
assumptions about the order in which variables enter the path and are
likely not to yield sharp estimation bounds such as
\eqref{eq:gauss_minimax}---provided that they can be analyzed. Others
such as \cite{lassotest} are likely to be far too
conservative. In a different direction, it would be interesting to compare SLOPE with the Lasso in different settings, where perhaps both $k/p$ and $n/p$ converge to positive constants. While some
  tools have been developed for the Lasso in this asymptotic regime
  \cite{lassorisk}, it is unclear how SLOPE would behave and even what
  a good sequence of weights $\{\lambda_i\}$ might be in this case.

\subsection{Data-driven procedures}
\label{sec:cross-valit-type}

While finding tuning parameters adaptively is an entirely new issue, a
data-driven procedure where the regularization parameter of the Lasso
is chosen in an adaptive fashion would presumably boost
performance. Cross-validation comes to mind whenever applicable, which
is not always the case as when $\bm y \sim \mathcal{N}(\bm \beta,
\sigma^2 \bm I_p)$. Cross-validation techniques are also subject to
variance effects and may tend to select over-parameterized models. To
make the selection of the tuning parameter as easy and accurate as
possible, we work in the orthogonal setting where we have available a
remarkable unbiased estimate of the risk.

\begin{figure}[!htp]
\centering
\begin{subfigure}[b]{0.48\textwidth}
\centering
\includegraphics[width=\textwidth]{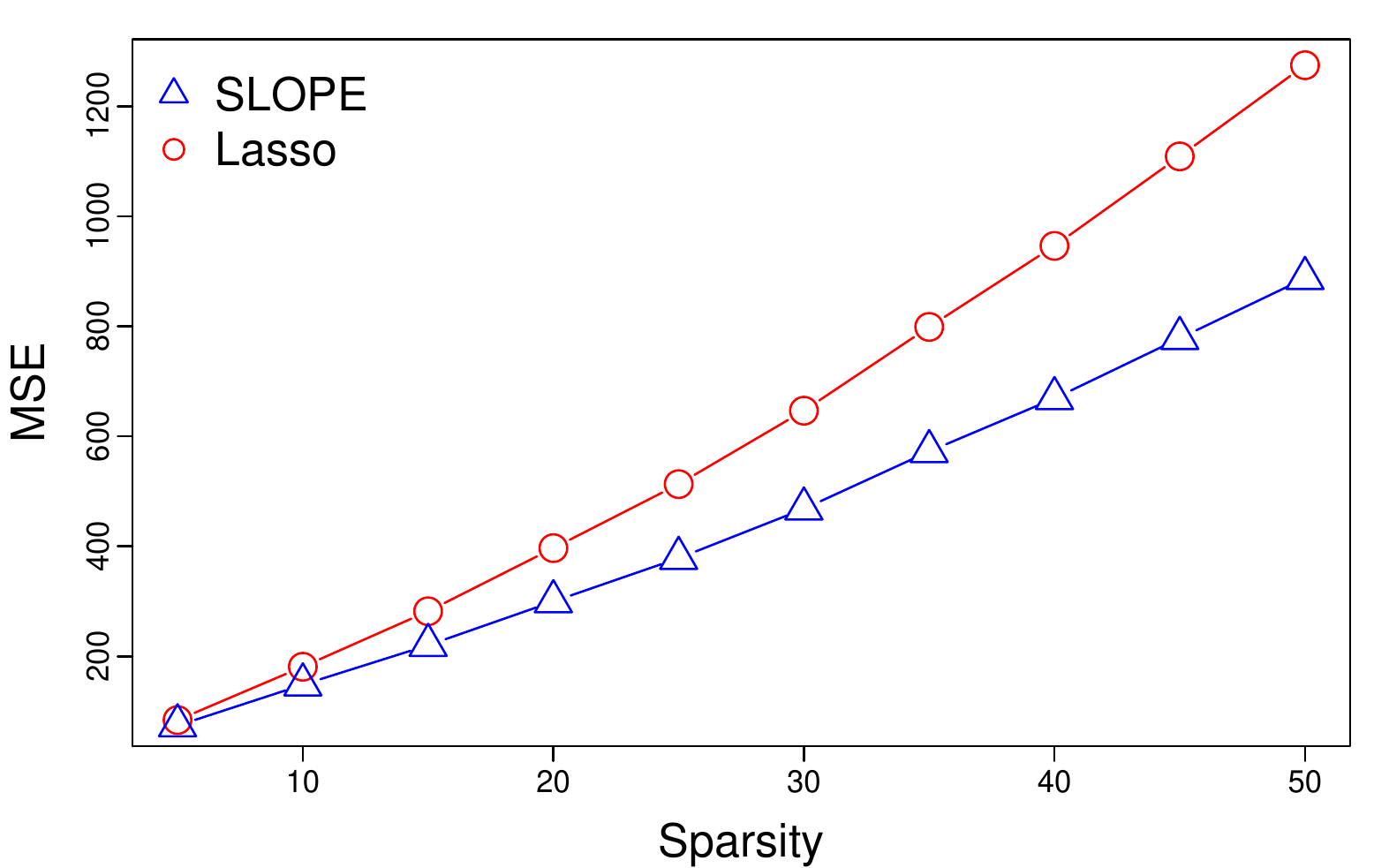}
\caption{Strong signals.}
\label{fig:lasso_slope}
\end{subfigure}
\hfill
\begin{subfigure}[b]{0.48\textwidth}
\centering
\includegraphics[width=\textwidth]{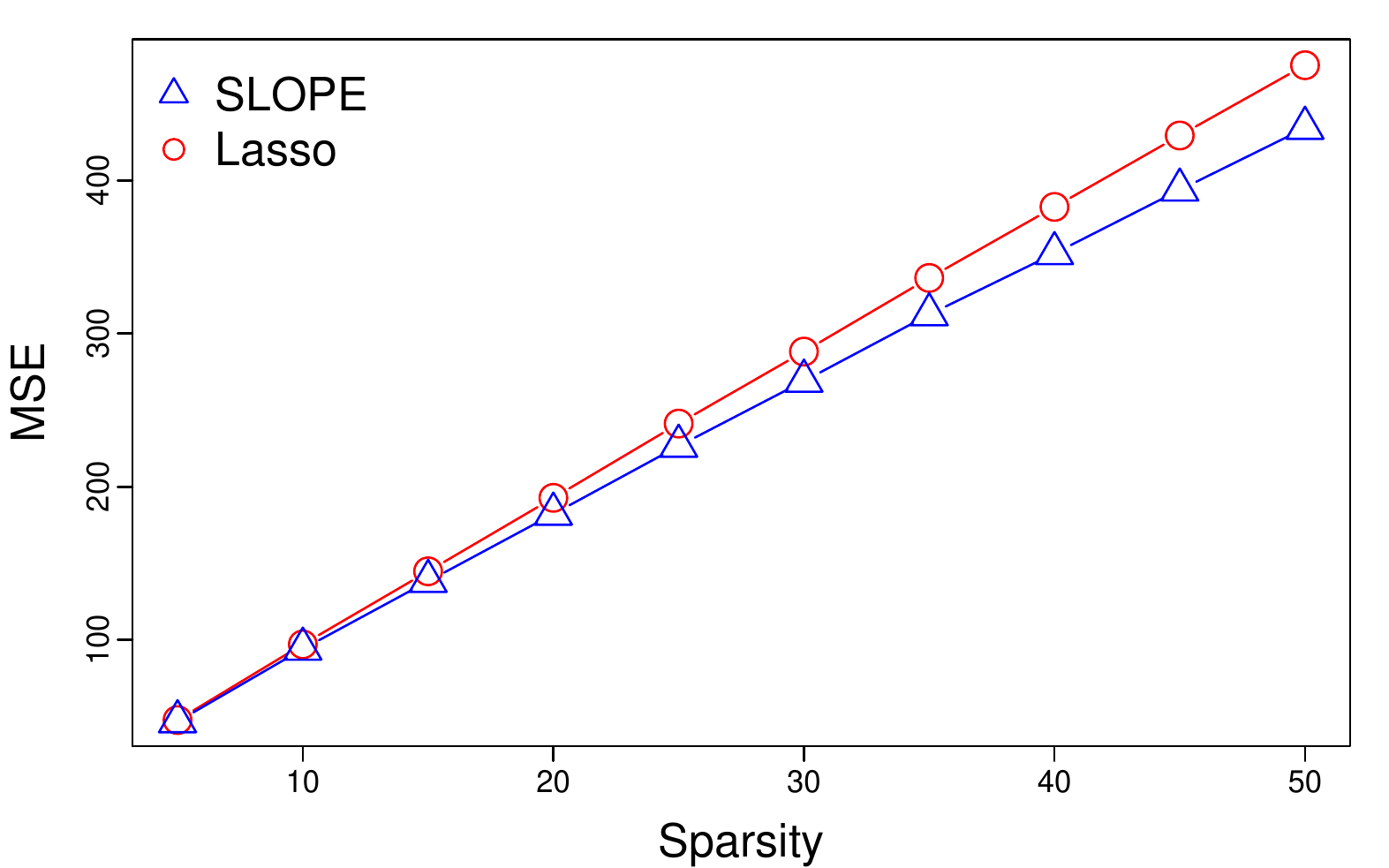}
\caption{Moderate signals.}
\label{fig:lasso_slope2}
\end{subfigure}
\begin{subfigure}[b]{0.48\textwidth}
\centering
\includegraphics[width = \textwidth]{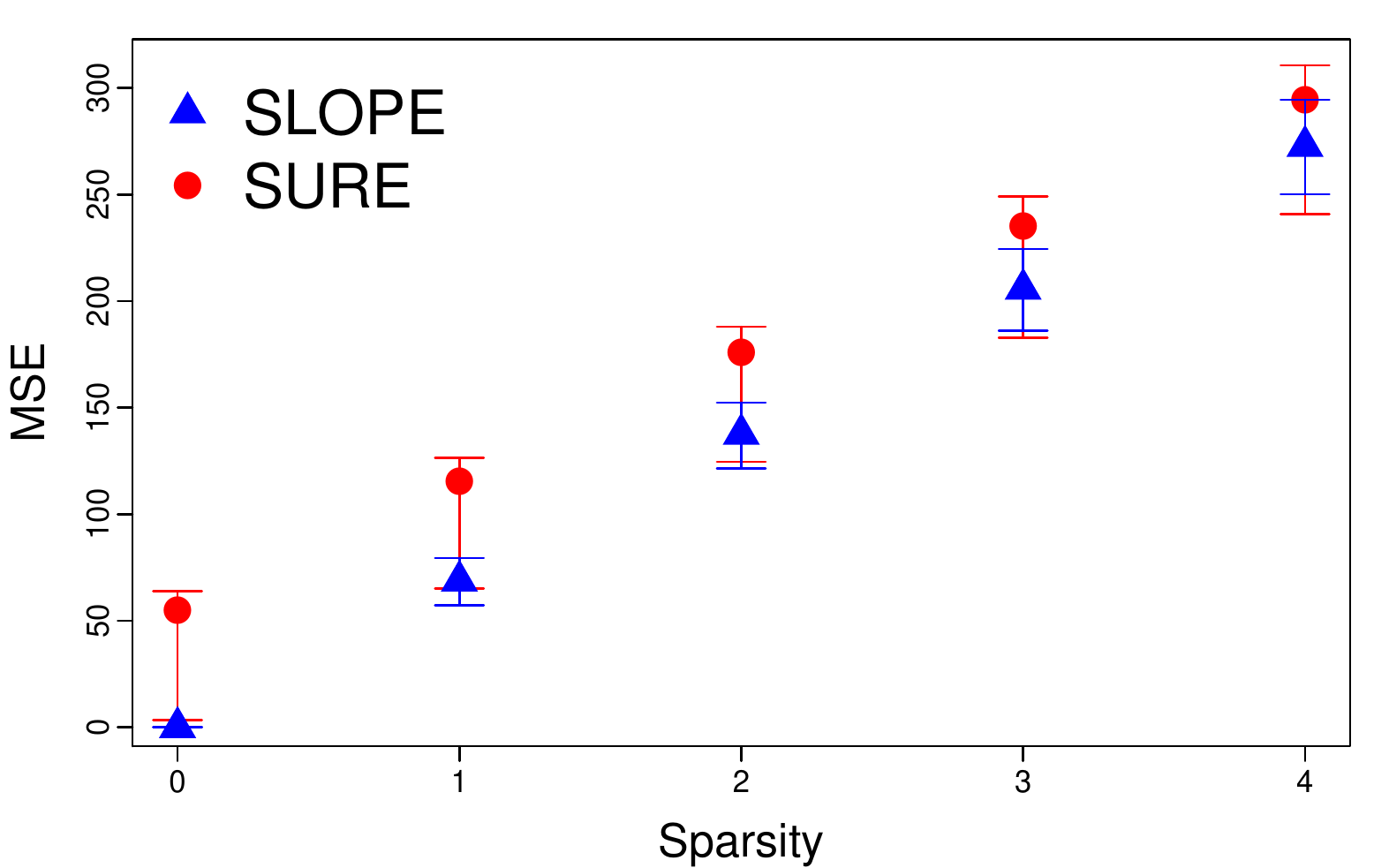}
\caption{Strong signals.}
\label{fig:sure_bar}
\end{subfigure}
\hfill
\begin{subfigure}[b]{0.48\textwidth}
\centering
\includegraphics[width = \textwidth]{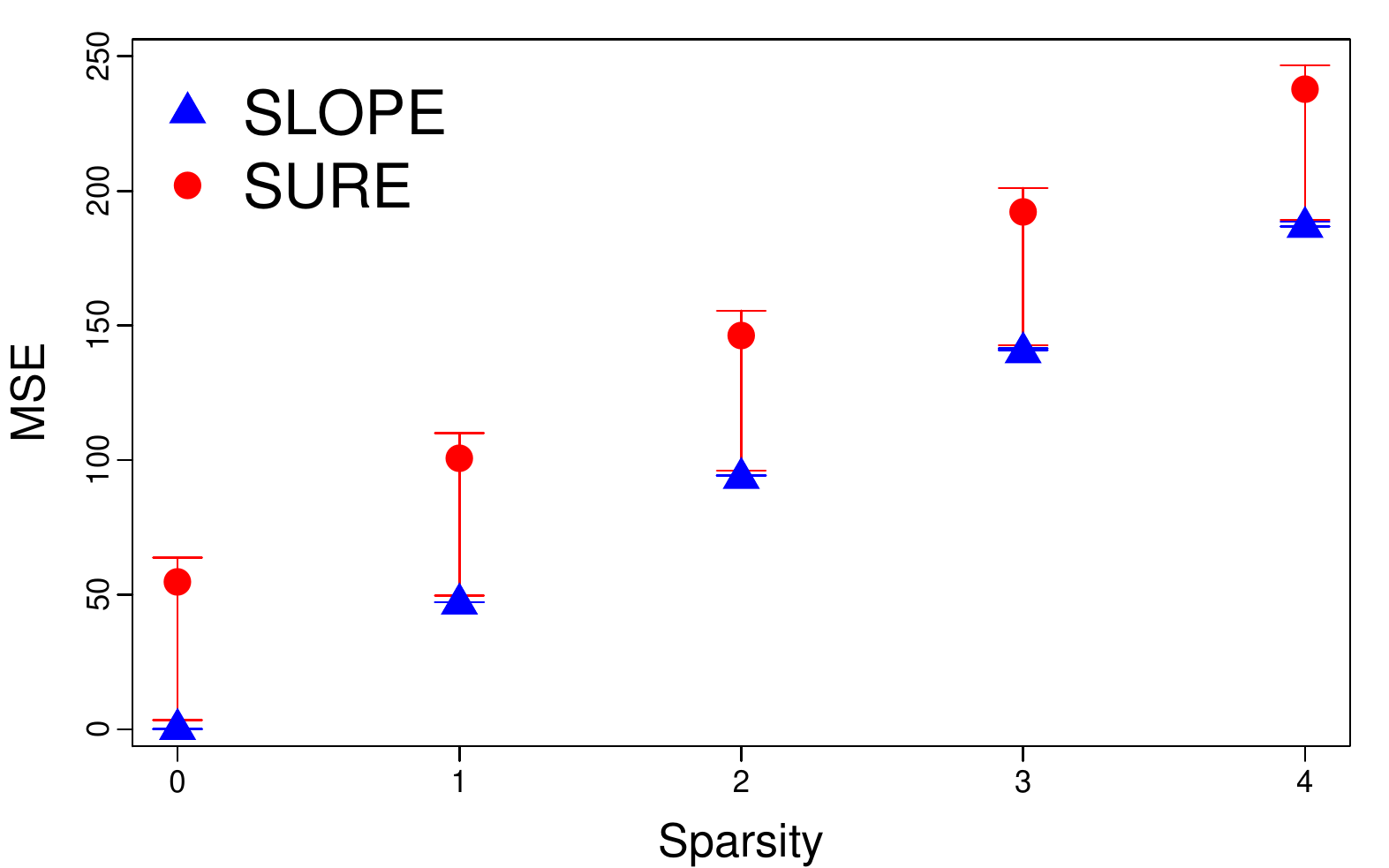}
\caption{Moderate signals.}
\label{fig:sure_bar2}
\end{subfigure}
\centering
\begin{subfigure}[b]{0.48\textwidth}
\centering
\includegraphics[width = \textwidth]{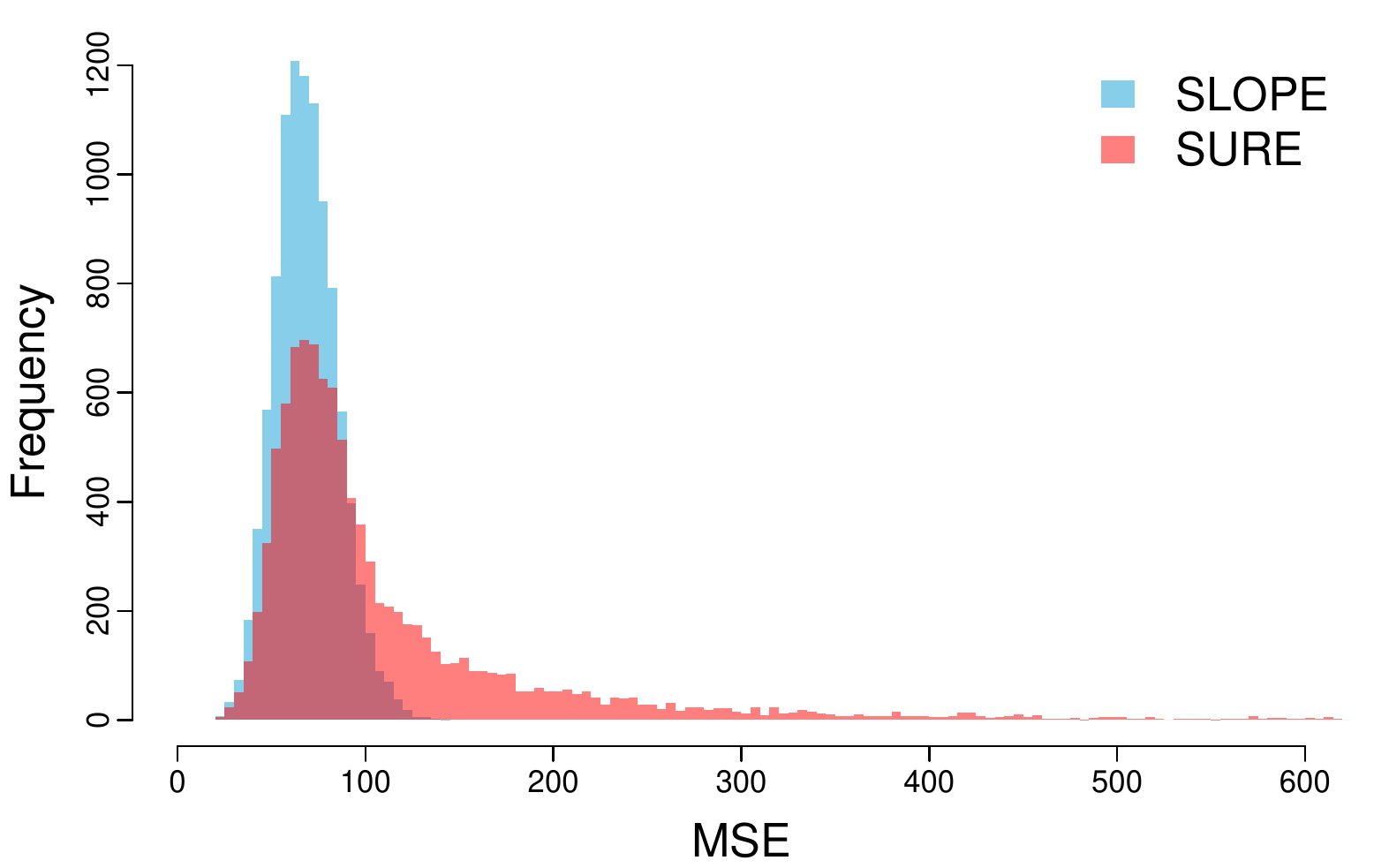}
\caption{Strong
  signals with $k = 1$.}
\label{fig:sure_hist}
\end{subfigure}
\hfill
\begin{subfigure}[b]{0.48\textwidth}
\centering
\includegraphics[width = \textwidth]{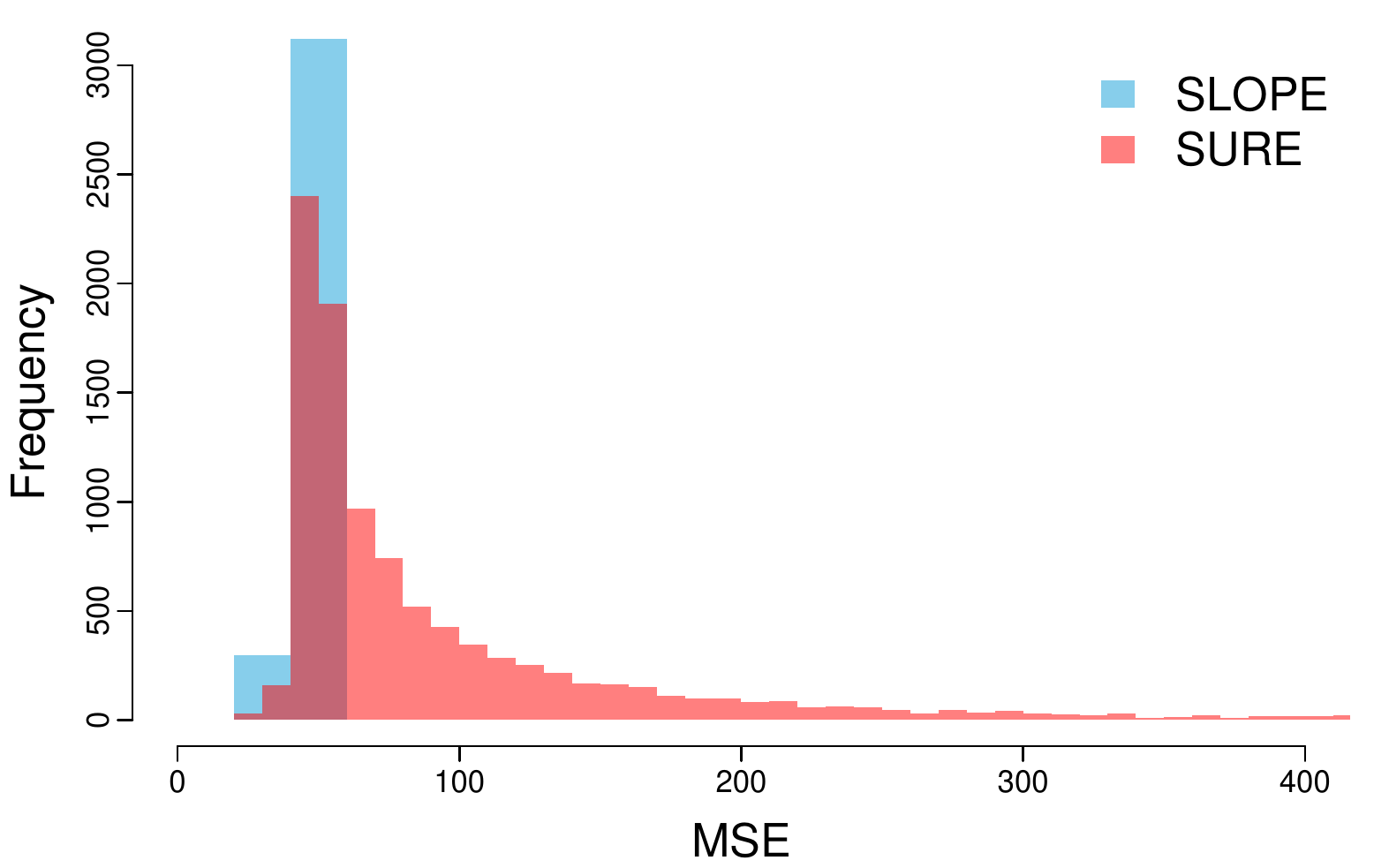}
\caption{Moderate
  signals with $k=1$.}
\label{fig:sure_hist2}
\end{subfigure}
\caption{(a) and (b) compares between SLOPE and Lasso under Gaussian design with $(n,p) = (500, 1000)$ and $\sigma = 1$. The risk $\E \|\hat{\bm\beta} - \bm\beta\|^2$ is averaged over
  100 replicates. SLOPE uses $\bm\lambda = \bm\lambbh(q)$ and Lasso
  uses $\lambda=\lambbh_{1}(q)$ with level $q = 0.05$. In (a), the components have magnitude $10\lambbh_1$; in (b), the magnitudes are set to $0.8\lambbh_1$. Next, (c, d, e, f) compare SLOPE with SURE under orthogonal design. Empirical distributions of $\|\hat{\bm\beta} - \bm\beta\|^2$ is obtained from 10,000 replicates.  Strong signals have nonzero $\beta_i$ set
  to $100\sqrt{2\log p}$ while this value is $0.8\sqrt{2\log p}$ for  moderate signals. In (c) and (d), the bars represent 75\% and 25\% percentiles.}
\label{fig:intro_sure}
\end{figure}

SURE thresholding \cite{SUREthreshold} for estimating a vector of
means from $\bm y \sim \mathcal{N}(\bm \beta, \sigma^2 \bm I_p)$ is a
cross-validation type procedure in the sense that the thresholding
parameter is selected to minimize Stein's unbiased estimate of risk
(SURE) \cite{steinsure}. For soft-thresholding at $\lambda$, SURE
reads
\[
\operatorname{SURE}(\lambda) = p\sigma^2 + \sum_{i=1}^p y_i^2 \wedge
\lambda^2 - 2\sigma^2\#\{i: |y_i| \le \lambda\}.
\]
One then applies the soft-thresholding rule at the minimizer $\hat
\lambda$ of $\operatorname{SURE}(\lambda)$. It has been observed
\cite{SUREthreshold, cai2009} that SURE thresholding loses performance
in cases of sparse signals $\bm\beta$, an empirical phenomenon which
can perhaps be made theoretically precise. Indeed, our own work in
progress aims to show that for any fixed sparsity $k$, SURE
thresholding obeys
\[
\frac{\sup_{\|\bm\beta\|_0 \le k}\E \|
  \hat{\bm\beta}_{\textnormal{\tiny SURE}} - \bm\beta
  \|^2}{\sup_{\|\bm\beta\|_0 \le k}\E \|
  \hat{\bm\beta}_{\textnormal{\tiny SLOPE}} - \bm\beta \|^2} \ge
(1+o(1))\frac{k+1}{k} > 1,
\]
where $k$ is allowed to take the value zero and $(k+1)/k = \infty$ in
this case. In particular, SURE has a risk that is infinitely larger
than SLOPE under the global null $\bm\beta = \vec 0$.

Figure \ref{fig:intro_sure} compares SLOPE with SURE in estimation
error. In Figures \ref{fig:sure_bar} and \ref{fig:sure_bar2}, we see
that SURE thresholding exhibits a squared error, which is consistently
larger in mean (risk) and variability. This difference is more
pronounced, the sparser the signal. Figures \ref{fig:sure_hist} and
\ref{fig:sure_hist2} display the error distribution for $k = 1$; we
see that the error of SURE thresholding is distributed over a longer
range.

\subsection{Variations on FDR thresholding}
\label{sec:other-bh-motivated}
As brought up earlier, the paper \cite{jiangadaptive} suggests a variation on FDR thresholding, where an adaptive smooth-thresholding rule is applied instead of a hard one. Such a procedure is still intrinsically limited to sequence models, and cannot be generalized to linear regression. On this subject, consider the
\textit{sequential FDR thresholding} rule,
\[
\hat{\bm\beta}_{\textnormal{\tiny Seq},i} = \sgn(y_i) \cdot \left( |y_i| -
\sigma\lambbh_{r(i)} \right)_+
\]
where $r(i)$ is the rank of $y_i$ when sorting the observations by
decreasing order of magnitude; that is, we apply soft-thresholding at
level $\sigma\lambbh_i$ to the $i$th largest observation (in
magnitude).  Under the same assumptions as in Theorem
\ref{thm:orth_main_general_intro}, this estimator also obeys
\begin{equation}\label{eq:seq_fdr_eq}
\sup_{\|\bm\beta\|_0 \le k}\E \|\hat{\bm\beta}_{\textnormal{\tiny Seq}} - \bm\beta\|^2= (1+o(1)) \, 2\sigma^2 k \log(p/k).
\end{equation}
The proof is in Appendix \ref{sec:proofs-sect-refs} and resembles
that of Theorem \ref{thm:orth_main_general_intro}. Even though the
worst case performance of this estimate matches that of SLOPE, it is
not a desirable procedure for at least two reasons. The first is that
it is not monotone; we may have $|y_i| > |y_j|$ and $|\hat \beta_j| >
|\hat \beta_i|$, which does not make much sense. A consequence is that
it will generally have higher risk. Also note that this estimator is
not continuous with respect to $\bm y$, since a small perturbation can
change the ordering of magnitudes and, therefore, the amount of
shrinkage applied to an individual component. The second reason is
that this procedure does not really extend to linear models.

%%% Local Variables:
%%% mode: latex
%%% TeX-master: "paper"
%%% End:

\section{Orthogonal designs}
\label{sec:slope-with-orth}

This section proves the optimality of SLOPE under orthogonal
designs. As we shall see, the proof is considerably shorter and
simpler than that in \cite{ABDJ} for FDR thresholding. One reason for
this is that SLOPE continuously depends on the observation vector while
FDR thresholding does not, a fact which causes serious technical
difficulties. The discontinuities of the FDR hard-thresholding
procedure also limits the range of its effectiveness (recall the
limits on the range of sparsity levels which state that the signal
cannot be too sparse or too dense) as false discoveries result in
large squared errors.

A reason for separating the proof in the orthogonal case is
pedagogical in that the argument is conceptually simple and, yet, some
of the ideas and tools will carry over to that of Theorem
\ref{thm:gauss_minimax}.  From now on and throughout the paper we set
$\sigma = 1$.

\subsection{Preliminaries}
\label{sec:preliminaries}

We collect some preliminary facts, which will prove useful, and begin
with a definition used to characterize the solution to SLOPE.
\begin{definition}
  A vector $\bm a \in \R^p$ is said to majorize $\bm b \in \R^p$ if
  for all $i = 1,\ldots, p$,
\begin{equation}\nonumber
|a|_{(1)}+\cdots + |a|_{(i)} \geq |b|_{(1)}+\cdots + |b|_{(i)}. 
\end{equation}
\end{definition}
This differs from a more standard definition---e.g.~see
\cite{olkin}---where the last inequality with $i = p$ is replaced by
an equality (and absolute values are omitted). We see that if $\bm a$
majorizes $\bm b$ and $\bm c$ majorizes $\bm d$, then the concatenated
vector $(\bm a, \bm c)$ majorizes $(\bm b, \bm d)$. For convenience,
we list below some basic but nontrivial properties of majorization and
of the prox to the sorted $\ell_1$ norm as defined in \eqref{eq:prox_intro}.  All the
proofs are deferred to the Appendix.
\begin{fact}\label{fc:norm}
If $\bm a$ majorizes $\bm b$, then $\norm{\bm a} \geq \norm{\bm b}$.
\end{fact}

\begin{fact}\label{fc:maj_zero}
  If $\bm \lambda$ majorizes $\bm a$, then $\slope{\bm a} = \bm 0$.
\end{fact}

\begin{fact}\label{fc:slope_maj}
The difference $\bm a - \slope{\bm a}$ is majorized by $\bm\lambda$.
\end{fact}

\begin{fact}\label{fc:monotone}\label{fc:reduce}
  Let $T$ be a nonempty proper subset of $\{1, \ldots, p \}$, and
  recall that $\bm{a}_T$ is the restriction of $\bm{a}$ to $T$ and
  $\bm\lambda^{-[m]} = (\lambda_{m+1}, \ldots, \lambda_p)$. Then
\[
\left\| \left[ \slope{\bm a} \right]_{\overline T} \right\| \le \left\| \slopex{\bm\lambda^{-[|T|]}}{\bm a_{\overline T}} \right\|.
\]
\end{fact}

\begin{lemma}\label{lm:algo-monotone}
For any $\bm a$, it holds that
\[
\|\slope{\bm a}\| \le \| (|\bm a| - \bm \lambda)_+\|, % \le \|(\bm a -
% \bm\lambda)_+\|.
\]
where $|\bm a|$ is the vector of magnitudes $(|a_1|, \ldots, |a_p|)$.
\end{lemma}
\begin{proof}[Proof of Lemma \ref{lm:algo-monotone}]
  The firm nonexpansiveness (e.g. see pp.131 of \cite{BoydProx}) of
  the prox reads
\[
\left\| \slope{\bm a} - \slope{\bm b} \right\|^2 \le (\bm a - \bm b)'\left(\slope{\bm a} - \slope{\bm b} \right)
\]
for all $\bm a, \bm b$. Taking $\bm b = \sgn(\bm a) \odot \bm\lambda$,
where $\odot$ is componentwise multiplication, and observing that
$\slope{\bm b} = \bm 0$ (Fact \ref{fc:maj_zero}) give 
\begin{align*}
  \left\| \slope{\bm a}\right\|^2 &\le \left\langle \sgn(\bm a) \odot (|\bm a| - \bm\lambda),  \slope{\bm a} \right\rangle \\
  &\le \left\langle (|\bm a| - \bm\lambda)_+, \sgn(\bm a) \odot \slope{\bm a} \right\rangle\\
  & \le \left\| ( |\bm a| - \bm\lambda)_+ \right\| \cdot \left\|\slope{\bm a} \right\|,
\end{align*}
where we use the nonnegativity of $\sgn(\bm a) \odot \slope{\bm a}$
and the Cauchy-Schwarz inequality. This yields the lemma. 
\end{proof}

\subsection{Proof of Theorem \ref{thm:orth_main_general_intro}}
\label{sec:proof-theor-refthm}

Let $S$ be the support of the vector $\bm \beta$,
$S = \supp{\bm\beta}$, and decompose the total mean-square error as
\[
\E \|\hat{\bm \beta} - \bm \beta\|^2 = \E \|\hat{\bm \beta}_S - \bm
\beta_S\|^2 + \E\| \hat{\bm \beta}_\barS - \bm \beta_\barS\|^2,
\]
i.e.~as a the sum of the contributions on and off support (in case
$\|\bm\beta\|_0 < k$, augment $S$ to have size $k$). Theorem
\ref{thm:orth_main_general_intro} follows from the following two
lemmas.

\begin{lemma}\label{lm:bound-on-support}
  Under the assumptions of Theorem \ref{thm:orth_main_general_intro},
  for all $k$-sparse vectors $\bm \beta$, 
\begin{equation}\nonumber
  \E \|\hat{\bm \beta}_S - \bm \beta_S\|^2 \le (1+o(1)) \, 2k \log(p/k). 
\end{equation}
\end{lemma}
\begin{proof}
  We know from Fact 3.3 that $\bm y - \hat{\bm\beta}$ is majorized by
  ${\bm \lambda} = \bm\lambbh$, which implies that $\bm{y}_S - \hat{\bm\beta}_S =
  \bm{\beta}_S + \bm{z}_S - \hat{\bm\beta}_S$ is majorized by
  ${\bm \lambda}^{[k]}$. The triangle inequality together with Fact
  \ref{fc:norm} give
\begin{equation*}\label{eq:orth_risk_s}
  \norm{\bm\beta_S - \hat{\bm\beta}_S} = \norm{\bm\beta_S +\bm{z}_S - \hat{\bm\beta}_S - \bm{z}_S} \le 
  \norm{\bm{y}_S - \hat{\bm\beta}_S} + \norm{\bm z_S} \le \norm{{\bm \lambda}^{[k]}} + \norm{\bm z_S}. 
\end{equation*}
This gives 
\begin{equation}\nonumber
\begin{aligned}
  \E   \norm{\bm\beta_S - \hat{\bm\beta}_S}^2  &\le  \sum_{i=1}^k(\lambbh_i)^2 + \E \|\bm z_S\|^2 + 2 \sqrt{\sum_{1\le i \le k}(\lambbh_i)^2} \E \|\bm z_S\|\\
  &\le  \sum_{i=1}^k(\lambbh_i)^2 + \E \|\bm z_S\|^2 + 2 \sqrt{\sum_{1\le i \le k}(\lambbh_i)^2 \E \|\bm z_S\|^2}\\
  & \le  \sum_{i=1}^k(\lambbh_i)^2 + k  + 2 \sqrt{k\sum_{1\le i \le k}(\lambbh_i)^2} \\
  & =(1+o(1))\, 2k\log(p/k),
\end{aligned}
\end{equation}
where the last step makes use of $\sum_{1\le i \le k}(\lambbh_i)^2 =
(1+o(1)) 2k\log(p/k)$ and $\log(p/k) \goto \infty$. 
\end{proof}

\begin{lemma}\label{lm:bound-out-support}
  Under the assumptions of Theorem \ref{thm:orth_main_general_intro},
  for all $k$-sparse vectors $\bm \beta$,
\begin{equation}
\label{eq:bound-out-support} 
\E \|\hat{\bm \beta}_\barS - \bm \beta_\barS\|^2 = o(1) \, 2k \log(p/k). 
\end{equation}
\end{lemma}
\begin{proof}
  It follows from Fact \ref{fc:monotone} that
\[
\|\hat{\bm \beta}_\barS\|^2 = \left\| [\slopex{\bm \lambda}{\bm
    y}]_{\overline S} \right\|^2 \le \left\| \slopex{\bm
    \lambda^{-[k]}}{\bm{z}_{\overline S}} \right\|^2.
\]
We proceed by showing that for $\bm\zeta \sim \mathcal{N}(\bm 0,
\bm{I}_{p-k})$, $\E\norm{\slopex{\bm \lambda^{-[k]}}{\bm\zeta}}^2 =
o(1) \, 2k\log(p/k)$. To do this, pick $A>0$ sufficiently large such that
$q(1+1/A)< 1$ in Lemmas \ref{lm:bound-out-support-close} and
\ref{lm:bound-out-support-far}, which then give
\[
  \sum_{i=1}^{p-k}\E \left( |\zeta|_{(i)} - \lambbh_{k+i} \right)_+^2 = o(1) \, 2k 
  \log(p/k).
\]
The conclusion follows from Lemma \ref{lm:algo-monotone} since
\[
  \E \left\| \slopex{\bm \lambda^{-[k]}}{\bm\zeta} \right\|^2 \le \sum_{i=1}^{p-k}\E \left( |\zeta|_{(i)} - \lambbh_{k+i} \right)_+^2 = o(1) \, 2k\log(p/k).
\]
\end{proof}

We conclude this section with a probabilistic bound on the squared
loss. The proposition below, whose argument is nearly identical to
that of Theorem \ref{thm:orth_main_general_intro}, shall be used as a
step in the proof of Theorem \ref{thm:gauss_minimax}. 
\begin{proposition}\label{prop:orth_main_general}
  Fix $0 < q < 1$ and set $\bm \lambda = (1+\epsilon) \bm \lambbh(q)$
  for some arbitrary $0 < \epsilon < 1$.  Suppose $k/p \goto 0$, then
  for each $\delta > 0$ and all $k$-sparse $\bm \beta$,
\begin{equation}\nonumber
  \P\left( \frac{\|\bm{\hat\beta} - \bm\beta\|^2}{2(1+\epsilon)^2k\log(p/k)} < 1 + \delta \right) \goto 1.
\end{equation}
Here, the convergence is uniform over $\epsilon$.
\end{proposition}
\begin{proof}
  We only sketch the proof. As in the proof of Lemma
  \ref{lm:bound-on-support}, we have
  \[
\|\bm \beta_S - \hat{\bm \beta}_S\| \le 
\norm{\lambe^{[k]}} + \norm{\bm z_S}. 
\]
Since $\norm{\lambe^{[k]}} = (1 + o(1)) \cdot (1+\epsilon)\sqrt{2k
  \log(p/k)}$ and $\|z_S\| = o_{\P}(\sqrt{2k \log (p/k)})$, we have that
for each $\delta > 0$,
\begin{equation}\nonumber
  \P\left( \frac{\| \hat{\bm\beta}_S - \bm\beta_S\|^2}{2(1+\epsilon)^2k\log(p/k)} < 1 + \delta/2 \right) \goto 1.
\end{equation}
Since $\bm \lambda$ has increased, it is only natural that the
off-support error remains under control. In fact,
\eqref{eq:bound-out-support} still holds, and the Markov inequality then
gives
\[
\P\left( \frac{\| \hat{\bm\beta}_\barS -
    \bm\beta_\barS\|^2}{2k\log(p/k)} < \frac{\delta}{2} \right) \goto
1.
\]
This concludes the proof. 
\end{proof}

%%% Local Variables:
%%% mode: latex
%%% TeX-master: "paper"
%%% End:

\section{Gaussian random designs}
\label{sec:slope-under-random-all}

When moving from an orthogonal to a non-orthogonal design, the
correlations between the columns of $\X$ and the high dimensionality
create much difficulty. This is already apparent when scanning the
literature on penalized sparse estimation procedures such as the
Lasso, SCAD \cite{fan2001variable}, the Dantzig selector
\cite{candes2007dantzig} and MC$+$ \cite{zhang2010mc}, see
e.g.~\cite{greenshtein2004,candes2006,zou2006,candes2009,bickel2009,van2009,
  wainwright2009,meinshausen2009,ye2010,lassorisk,donoho2011} for a
highly incomplete list of references. For example, a statistical
analysis of the Lasso often relies on several ingredients: first, the
Karush-Kuhn-Tucker (KKT) optimality conditions; second, appropriate
assumptions about the designs such as the Gaussian model we use here,
which guarantee a form of local orthogonality (known under the name of
restricted isometries or restricted eigenvalue conditions); third, the
selection of a penalty $\lambda$ several times the size of the
universal threshold $\sigma\sqrt{2\log p}$, which while introducing a
large bias yielding MSEs that cannot possibly approach the precise
bounds we develop in this paper, facilitates the analysis since it
effectively sets many coordinates to zero. 

Our approach must be different for at least two reasons. To begin
with, the KKT conditions for SLOPE are not easy to manipulate. Leaving
out this technical matter, a more substantial difference is that the
SLOPE regularization is far weaker than that of a Lasso model with a
large value of the regularization parameter $\lambda$. To appreciate
this distinction, consider the {\em orthogonal design} setting. In
such a simple situation, it is straightforward to obtain error
estimates about a hard thresholding rule set at---or several
times---the Bonferroni level. Getting sharp estimates for FDR
thresholding is entirely a different matter, compare the far longer proof in \cite{ABDJ}.

\subsection{Architecture of the proof}
\label{sec:architecture} 

Our aim in this section is to provide a general overview of the proof,
explaining the key novel ideas and intermediate results. At a high
level, the general structure is fairly simple and is as follows:
\begin{enumerate}
\item Exhibit an ideal estimator $\tilde{\bm \beta}$, which is easy to
  analyze and achieves the optimal squared error loss with high
  probability.
\item Prove that the SLOPE estimate $\hat{\bm \beta}$ is close to this
  ideal estimate. 
\end{enumerate}
We discuss these in turn and recall that throughout,
$\bm \lambda = (1+\epsilon) \bm \lambbh(q)$.

A solution algorithm for SLOPE is the proximal gradient method, which
operates as follows: starting from an initial guess $\bm b^{(0)} \in
\R^p$, inductively define
\[
\bm{b}^{(m+1)} = \slopex{t_m\bm\lambda}{\bm{b}^{(m)} - t_m\bm{X}'(\bm{X}\bm{b}^{(m)} - \bm y)}, 
\]
where $\{t_m\}$ is an appropriate sequence for step sizes.  It is
empirically observed that under sparsity constraints, the proximal
gradient algorithm for SLOPE (and Lasso) converges quickly provided we
start from a good initial point. Here, we propose approximating the
SLOPE solution by starting from the ground truth and applying just one
iteration; that is, with $t_0 = 1$, define 
\begin{equation}\label{eq:slope_orth_approx}
\tilde{\bm\beta} := \slopex{\bm\lambda}{\bm\beta + \bm{X}'\bm{z}}.
\end{equation}
This oracle estimator $\tilde{ \bm \beta}$ approximates the SLOPE
estimator $\hat{ \bm \beta}$ well---they are equal when the design is
orthogonal---and has statistical properties far easier to
understand. The lemma below is the subject of Section
\ref{sec:minimax-one-step}.
\begin{lemma}\label{lm:weak_conditional}
  Under the assumptions of Theorem \ref{thm:gauss_minimax}, for all
  $k$-sparse $\bm \beta$, we have
\[
\P\left( \frac{\| \tilde{\bm\beta}-\bm\beta\|^2}{(1+\epsilon)^2\, 2
    k\log(p/k)} < 1 + \delta \right) \goto 1,
\]
where $\delta > 0$ is an arbitrary constant.
\end{lemma}

\newcommand{\btilde}{{\tilde{\bm \beta}}}
\newcommand{\bhat}{{\hat{\bm \beta}}}

Since we know that $\btilde$ is asymptotically optimal, it suffices to
show that the squared distance between $\bhat$ and $\btilde$ is
negligible in comparison to that between $\btilde$ and $\bm \beta$.
This captured by the result below, whose proof is the subject of
Section \ref{sec:close}.
\begin{lemma}\label{lm:two_beta}
  Let $T \subset \{1, \ldots, p\}$ be a subset of columns assumed to
  contain the supports of $\bhat$, $\btilde$ and $\bm \beta$; i.e.~$T
  \supset \supp{\bhat} \cup \supp{\btilde} \cup \supp{\bm \beta}$.
  Suppose all the eigenvalues of $\X'_T \X_T$ lie in $[1-\delta,
  1+\delta]$ for some $\delta < 1/2$. Then
\begin{equation}\nonumber
\|\tilde{\bm\beta} - \hat{\bm\beta}\|^2 \le \frac{3\delta}{1 - 2\delta} \| \tilde{\bm\beta} - \bm\beta\|^2.
\end{equation}
In particular, $\bhat = \btilde$ under orthogonal designs.
\end{lemma}

We thus see that everything now comes down to showing that there is a
set of small cardinality containing the supports of $\bhat$, $\btilde$
and $\bm \beta$. While it is easy to show that
$\supp{\btilde} \cup \supp{\bm \beta}$ is of small cardinality, it is
delicate to show that this property still holds with the addition of
the support of the SLOPE estimate. Below, we introduce the
\textit{resolvent set}, which will prove to contain
$\supp{\bhat} \cup \supp{\btilde} \cup \supp{\bm \beta}$ with high
probability.
\begin{definition}[Resolvent set]\label{def:T_star}
  Fix $S = \supp{\bm \beta}$ of cardinality at most $k$, and an
  integer $\K$ obeying $k < \K < p$. The set $S^{\star} = S^\star(S,
  \K)$ is said to be a resolvent set if it is the union of $S$ and the
  $\K - k$ indices with the largest values of $|\bm{X}_i' \bm z|$
  among all $i \in \{1, \ldots, p\}\setminus S$.
\end{definition}

Under the assumptions of Theorem \ref{thm:gauss_minimax}, we shall see
in Section \ref{sec:find-supp-slope} that we can choose $\K$ in such a
way that on the one hand $\K$ is sufficiently small compared to $p$
and $n/\log p$, and on the other, the resolvent set $S^\star$ is still
expected to contain $\supp{\tilde{\bm\beta}}$ (easier) and
$\supp{\hat{\bm\beta}}$ (more difficult).  Formally, Lemma
\ref{lm:T_contains} below shows that
\begin{equation}
\label{eq:superset}
\inf_{\|\bm\beta\|_0 \le k}\P\left( \supp{\bm\beta} \cup
  \supp{\hat{\bm\beta}} \cup \supp{\tilde{\bm\beta}} \subset S^{\star}
\right) \goto 1.
\end{equation}
One can view the resolvent solution as a sophisticated type of a dual
certificate method, better known as primal-dual witness method
\cite{wainwright2009,candes2009,ravikumar2010} in the statistics
literature. A significant gradation in the difficulty of detecting the
support of the SLOPE solution a priori comes from the false
discoveries we commit because we happen to live on the edge, i.e.~work
with a procedure as liberal as can be.

With \eqref{eq:superset} in place, Theorem \ref{thm:gauss_minimax}
merely follows from Lemma \ref{lm:two_beta} and the accuracy of
$\btilde$ explained by Lemma \ref{lm:weak_conditional}; all the
bookkeeping is in Section \ref{sec:proof-theor-refthm:g}. Furthermore,
Corollary \ref{coro:gauss_prediction_minimax} is just one stone throw
away, please also see Section \ref{sec:proof-theor-refthm:g} for all
the necessary details.

% choice of a not too large $\bm\lambda$, rendering
% primal-dual witness method insufficient here. To address this issue,
% $\K-k$ additional indices are introduced. This enlargement allows the
% primal-dual witness method feasible again.

% If we can take $\delta$ above to be arbitrary small, this
% would conclude the proof of our main result. The trouble is that we
% are in high dimensions, and this is of course not possible. Indeed,
% the eigenvalues of $\bm X$ cannot possibly be well localized; when $p
% > n$, $\X' \X$ is actually singular. So what possible use can we make
% of Lemma \ref{lm:two_beta}?

% To see why this lemma is useful, suppose that a subset of variables $T
% \subset \{1, \ldots, p\}$ contains both the support of $\bhat$ and
% $\btilde$; i.e.~$T \supset \supp{\bhat} \cup \supp{\btilde}$. Then
% Lemma \ref{lm:two_beta} continues to hold provided that the
% eigenvalues of $\X_T' \X_T$ lie in $[1-\delta, 1 + \delta]$, where we
% recall that $\X_T$ is the restriction of $\X$ to column indices in
% $T$. 

\subsection{One-step approximation}
\label{sec:minimax-one-step}

The proof of Lemma \ref{lm:weak_conditional} is an immediate
consequence from Proposition \ref{prop:orth_main_general}.  In brief,
Borell's inequality---see Lemma \ref{lm:gauss_concentrate}---provides
a well-known deviation bound about chi-square random variables,
namely,
\begin{equation}\nonumber
  \P\left( \|\bm z\| \le (1 + \epsilon)\sqrt{n} \right) \ge 1 -  \e{-\epsilon^2 n/2} \goto 1
\end{equation}
since $\epsilon^2 n \goto \infty$.  Hence, to prove our claim, it suffices to
establish that
\begin{equation}
  \P\left( \frac{\|\prox{\lambe}{(\bm\beta+\bm{X}' \bm z)}-\bm\beta\|^2}{(1+\epsilon)^2 2k\log(p/k)} < 1 + \delta \ \Big{|}\  \|\bm z\| \le (1 + \epsilon)\sqrt{n} \right) \goto 1.\label{eq:one_step_upper2}
\end{equation}
Conditional on $\|\bm z\| = c\sqrt{n}$ for some $0 < c \le 1 +
\epsilon$, $\bm{X}' \bm z \sim \mathcal{N}(\bm 0, c^2 \bm I_p)$ and,
therefore, conditionally,
\begin{align*}
  \|\slopex{\lambe}{\bm\beta+\bm{X}' \bm z}-\bm\beta\| & \eqd
  \|\slopex{\lambe}{\bm\beta+ c \mathcal{N}(\bm 0, \bm I_p)} - \bm
  \beta\| \\
  & = c \|\slopex{\bm\lambda_{\epsilon'}}{\bm\beta/c+\mathcal{N}(\bm
    0, \bm{I}_p)}-\bm\beta/c\|
\end{align*}
for $\epsilon' = (1+\epsilon)/c - 1 \ge 0$. Hence, Proposition
\ref{prop:orth_main_general} gives
\[
\P \lb
\frac{\|\slopex{\bm\lambda_{\epsilon'}}{\bm\beta/c+\mathcal{N}(\bm 0,
    \bm{I}_p)}-\bm\beta/c\|^2}{(1+\epsilon')^2 \, 2k\log(p/k)} < 1 +
\delta \rb \goto 1. 
\]
Since $(1+\epsilon)^2/c^2 = (1+\epsilon')^2$, this is equivalent to
\[
\P \lb
\frac{c^2\|\slopex{\bm\lambda_{\epsilon'}}{\bm\beta/c+\mathcal{N}(\bm 0,
    \bm{I}_p)}-\bm\beta/c\|^2}{(1+\epsilon)^2 \, 2k\log(p/k)} < 1 +
\delta \rb \goto 1. 
\]
This completes the proof since we can deduce
\eqref{eq:one_step_upper2} by averaging over $\|\bm z\|$.

%\subsection{$\tilde{\bm\beta}$ and $\hat{\bm\beta}$ are close when $\bm X$ is nearly orthogonal}
\subsection{ {\boldmath$\tilde\beta$} and {\boldmath$\hat\beta$} are close when {\boldmath$X$} is nearly orthogonal}

\label{sec:close}

We prove Lemma \ref{lm:two_beta} in the case where $T = \{1, \ldots,
p\}$, first. Set $J_{\bm \lambda}(\bm b) = \sum_{1 \le i \le p}
\lambda_i |b|_{(i)}$, by definition $\bhat$ and $\btilde$ respectively
minimize
\begin{align*}
  L_1(\bm b) & := \half\|\X(\bm\beta - \bm b)\|^2 + \bm z' \X(\bm\beta-\bm b)+J_{\bm \lambda}(\bm b) \\
  L_2(\bm b) &:= \half\|\bm\beta - \bm b\|^2 + \bm z' \X(\bm\beta- \bm
  b)+J_{\bm \lambda}(\bm b).
\end{align*}
Next the assumptions about the eigenvalues of $\X' \X$ implies that
these two functions are related, 
\begin{equation*}
\begin{aligned}
L_2(\bm{\tilde\beta}) - \frac{\delta}{2}\|\bm\beta - \bm{\tilde\beta}\|^2 &\le L_1(\bm{\tilde\beta}) \le L_2(\bm{\tilde\beta}) + \frac{\delta}{2}\|\bm\beta - \bm{\tilde\beta}\|^2,\\
L_2(\hat{\bm\beta}) - \frac{\delta}{2}\|\bm\beta - \hat{\bm\beta}\|^2 &\le L_1(\hat{\bm\beta}) \le L_2(\hat{\bm\beta}) + \frac{\delta}{2}\|\bm\beta - \hat{\bm\beta}\|^2.
\end{aligned}
\end{equation*}
Chaining these inequalities gives 
\begin{equation}
\label{eq:ripp}
L_2(\bm{\tilde\beta}) + \frac{\delta\|\bm\beta - \bm{\tilde\beta}\|^2}{2} \ge L_1(\bm{\tilde\beta}) \ge L_1(\hat{\bm\beta}) 
\ge L_2(\hat{\bm\beta}) - \frac{\delta\|\bm\beta - \hat{\bm\beta}\|^2}{2}. 
\end{equation}
Now the strong convexity of $L_2$ also gives 
\begin{equation*}\label{eq:strong_convex}
L_2(\hat{\bm\beta}) \ge L_2(\bm{\tilde\beta}) + \frac{\|\bm{\tilde\beta} - \hat{\bm\beta}\|^2}{2},   
\end{equation*}
and plugging this in the right-hand side of \eqref{eq:ripp} yields
\begin{equation}\label{eq:beta_close}
\frac{\|\bm{\tilde\beta} - \hat{\bm\beta}\|^2}{2} - \frac{\delta\|\bm\beta - \hat{\bm\beta}\|^2}{2} \le \frac{\delta\|\bm\beta - \bm{\tilde\beta}\|^2}{2}.
\end{equation}
Since $\delta\|\bm\beta - \hat{\bm\beta}\|^2/2 \le
\delta\|\bm{\tilde\beta} - \hat{\bm\beta}\|^2 + \delta\|\bm\beta -
\bm{\tilde\beta}\|^2$ (this is essentially the basic inequality
$(a+b)^2 \le 2a^2 + 2b^2$), the conclusion follows.

We now consider the general case. Let $m$ be the cardinality of $T$
and for $\bm b \in \R^m$, set
$J_{{\bm \lambda}^{[m]}}(\bm b) = \sum_{1 \le i \le m} \lambda_i
|b|_{(i)}$,
and observe that by assumption, $\bhat_T$ and $\btilde_T$ are
solutions to the reduced problems
\begin{equation}
  \label{eq:reduced}
\underset{\bm b \in \R^{|T|}}{\mbox{argmin}}
\quad \half \|\bm y-\bm{X}_T \bm b\|^2 +
J_{{\bm \lambda}^{[m]}}(\bm b)
\end{equation}
and 
\[
\underset{\bm b \in \R^{|T|}}{\mbox{argmin}}
\quad \half \|\bm \beta_T + \X'_T \bm z -
\bm b\|^2 + J_{{\bm \lambda}^{[m]}}(\bm b).
\]
Using the fact that $\X \beta = \X_T \beta_T$, we see that $\bhat_T$
and $\btilde_T$ respectively minimize
\begin{align*}
  L_1(\bm b) & := \half\|\X_T(\bm\beta_T - \bm b)\|^2 + \bm z' \X_T(\bm\beta_T-\bm b)+J_{{\bm \lambda}^{[m]}}(\bm b) \\
  L_2(\bm b) &:= \half\|\bm\beta_T - \bm b\|^2 + \bm z' \X_T(\bm\beta_T- \bm
  b)+J_{{\bm \lambda}^{[m]}}(\bm b).
\end{align*}
From now on, the proof is just as before.

%%%%%%%%%%%%%%%%%%%%%%%%%

\subsection{Support localization}
\label{sec:find-supp-slope}

Below we write $\bm a \preceq \bm b$ as a short-hand for $\bm b$
majorizes $\bm a$ and 
\begin{equation}
\label{eq:Sdiamond}
S^\diamond = \supp{\bm\beta} \cup \supp{\hat{\bm\beta}} \cup
\supp{\tilde{\bm\beta}}. 
\end{equation}
\begin{lemma}[Reduced SLOPE]\label{lm:lifted}
  Let $\hat{\bm b}_T$ be the solution to the reduced SLOPE problem
  \eqref{eq:reduced}, which only fits regression coefficients with
  indices in $T$. If 
\begin{equation}
\label{eq:partialKKT}
\bm{X}'_{\overline T}(\bm y - \bm X_T\hat{\bm b}_T) \preceq
\bm\lambda^{-[|T|]}, 
\end{equation}
then it is the solution to the full SLOPE problem in the sense that
$\bhat$ defined as $\bhat_T = \hat{{\bm b}}_T$ and $\bhat_{\overline{T}} =
\bm 0$ is solution.
\end{lemma}
The inequality \eqref{eq:partialKKT}, which implies localization of
the solution, reminds us of a similar condition for the Lasso. In
particular, if $\lambda_1 = \lambda_2 = \cdots = \lambda_p$, then
SLOPE is the Lasso and \eqref{eq:partialKKT} is equivalent to
$\|\bm{X}'_{\overline T}(\bm y - \bm X_T\hat{\bm b}_T)\|_\infty \le
\lambda$.
In this case, it is well known that this implies that a solution to
the Lasso is supported on $T$, see
e.g.~\cite{wainwright2009,candes2009,ravikumar2010}.

The main result of this section is this: 
\begin{lemma}\label{lm:T_contains}
  Suppose 
  \[
\K \ge \max \left\{ \frac{1 + c}{1-q} \, k, k + d\right\}
\]
for an arbitrary small constant $c > 0$, where $d$ is a deterministic sequence diverging to
infinity\footnote{Recall that we are considering a sequence of
  problems with $(k_j, n_j, p_j)$ so that this is saying that $\K_j
  \ge \max(2(1-q)^{-1} k_j, k_j + d_j)$ with $d_j \goto \infty$.}  in
such a way that $\K/p\goto 0$ and $(\K \log p)/n \goto 0$.  Then
\[
\inf_{\|\bm\beta\|_0 \le k}\P\left( S^\diamond \subset S^{\star} \right) \goto 1.
\]
\end{lemma}

\begin{proof}[Proof of Lemma \ref{lm:T_contains}]
  By construction, $\supp{\bm\beta} \subset S^\star$ so we only need
  to show (i) $\supp{\hat{\bm\beta}} \subset S^\star$ and (ii)
  $\supp{\tilde{\bm\beta}} \subset S^\star$.  We begin with (i). By
  Lemma \ref{lm:lifted}, $\supp{\hat{\bm\beta}}$ is contained in
  $S^\star$ if
\[
\bm{X}'_{\overline{S^\star}}(\bm y - \bm
X_{S^\star}\hat{\bm\beta}_{S^\star}) \preceq \lambe^{-[\K]},
\]
which would follow from
\begin{equation}\label{eq:t_contains_maj1}
  \bm{X}'_{\overline{S^\star}}\bm{X}_{S^\star}(\bm\beta_{S^\star} - \hat{\bm\beta}_{S^\star})  \preceq \frac{\epsilon}{2}\left( \lambbh_{\K+1}, \ldots, \lambbh_p \right)
\end{equation}
and
\begin{equation}\label{eq:t_contains_maj2}
\bm{X}'_{\overline{S^\star}}\bm z \preceq (1+\epsilon/2)\left( \lambbh_{\K+1}, \ldots, \lambbh_p \right).
\end{equation}
Lemma \ref{lm:noise_maj} in Appendix concludes that
\eqref{eq:t_contains_maj1} holds with probability tending to one, since, by assumption, $\epsilon > 0$ is constant and $\sqrt{(\K\log p)/n} \goto 0$. To
show that \eqref{eq:t_contains_maj2} also holds with probability
approaching one, we resort to Lemma
\ref{lm:main_lift_maj}. Conditional on $\bm z$, $\bm X'_{\overline S}
\bm z \sim \mathcal{N}(0, \|\bm z\|^2/n \cdot \bm{I}_{p-k})$. By
definition, $\bm X'_{\overline{S^\star}} \bm z$ is formed from $\bm
X'_{\overline S} \bm z$ by removing its $\K - k$ largest entries in
absolute value. Denoting by $\zeta_1, \ldots, \zeta_{p-k}$ \iid
standard Gaussian random variables, \eqref{eq:t_contains_maj2} thus
boils down to
\begin{equation}\label{eq:lm_main_lift_maj}
  \left( |\zeta|_{(\K-k+1)}, |\zeta|_{(\K-k+2)}, \ldots, |\zeta|_{(p-k)} \right) \preceq \frac{(1+\epsilon/2)\sqrt{n}}{\|\bm z\|}\left( \lambbh_{\K+1}, \ldots, \lambbh_p \right).
\end{equation}
Borell's inequality (Lemma \ref{lm:gauss_concentrate}) gives 
\[
\P\left((1+\epsilon/2)\sqrt{n}/\|\bm z\| < 1 \right) = \P\left( \|\bm z\| - \sqrt{n} > \epsilon\sqrt{n}/2 \right) \le \e{-n\epsilon^2/8} \goto 0.
\]
The conclusion follows from Lemma \ref{lm:main_lift_maj}.

We turn to (ii) and note that 
\[
(\bm\beta + \bm{X}' \bm z)_{\overline{S^\star}} = \bm{X}'_{\overline{S^\star}} \bm z.
\]
Now our previous analysis implies $\bm X'_{\overline{S^\star}} \bm z
\preceq \lambe^{-[\K]}$ with probability tending to one. However, it
follows from Facts \ref{fc:reduce} and \ref{fc:maj_zero} that 
\[
\|\tilde{\bm\beta}_{\overline{S^\star}}\| = \|\slopex{\lambe}{\bm\beta + \bm{X}' \bm z}_{\overline{S^\star}}\| \le \|\slopex{\lambe^{-[\K]}}{\bm X'_{\overline{S^\star}}\bm z}\|  = \bm 0.
\] 
In summary, $\bm X'_{\overline{S^\star}} \bm z \preceq \lambe^{-[\K]}
\implies \supp{\tilde{\bm\beta}} \subset S^{\star}$. This concludes
the proof. 
\end{proof}

%%%%%%%%%%%%%%%%%%%%%%%%%
%%%%%%%%%%%%%%%%%%%%%%%%%%%

\subsection{Proof of Theorem \ref{thm:gauss_minimax} and Corollary \ref{coro:gauss_prediction_minimax}}
\label{sec:proof-theor-refthm:g}

Put
\[
\delta = \frac{1+3\epsilon}{(1+\epsilon)^2} - 1 = \frac{\epsilon - \epsilon^2}{(1+\epsilon)^2} > 0,
\]
and choose any $\delta' > 0$ such that
\[
(1+\delta')\left( \sqrt{3\delta'/(1 - 2\delta')} + 1 \right)^2 (1 + \delta/2) < (1 + \delta).
\]
Let $\mathscr{A}_1$ be the event $S^\diamond \subset S^\star$,
$\mathscr{A}_2$ that all the singular values of $\bm{X}_{S^\star}$ lie
in $\left[ \sqrt{1 - \delta'} , \sqrt{1 + \delta'} \right]$, and
$\mathscr{A}_3$ that
\begin{equation}\nonumber
  \frac{\|\tilde{\bm\beta} - \bm\beta\|^2}{(1+\epsilon)^2 \, 2k \log(p/k)} < 1 + \frac{\delta}{2}.
\end{equation}

We prove that each event happens with probability tending to one. For
$\mathscr{A}_1$, use Lemma \ref{lm:T_contains}, and set
\[
d = \min\left\{ \left\lfloor \sqrt{kn/\log p} \right\rfloor, \left\lfloor \sqrt{p} \right\rfloor \right\},
\]
which diverges to $\infty$, and 
\[
\K = \max\left\{ \left\lceil 2k/(1-q) \right\rceil, k+d \right\}.
\]
It is easy to see that $\K$ satisfies the assumptions of Lemma
\ref{lm:T_contains}, which asserts that $\P(\mathscr{A}_1) \goto 1$
uniformly over all $k$-sparse $\bm\beta$. For $\mathscr{A}_2$, since
$(\K \log p)/n \goto 0$ implies that $\K\log(p/\K)/n \goto 0$, then
taking $t$ sufficiently small in Lemma \ref{lm:slim_spec_refined}
gives $\P(\mathscr{A}_2) \goto 1$ uniformly over all $k$-sparse
$\bm\beta$. Finally, $\P(\mathscr{A}_3) \goto 1$ also uniformly over
all $k$-sparse $\bm\beta$ by Lemma \ref{lm:weak_conditional} since
$\epsilon^2 n \goto \infty$.

Hence,
$\P(\mathscr{A}_1 \cap \mathscr{A}_2 \cap \mathscr{A}_3) \goto 1$
uniformly over all $\bm\beta$ with sparsity at most $k$. Consequently,
it suffices to show that on this intersection,
\[
\frac{\|\hat{\bm\beta} - \bm\beta\|^2}{2 k \log(p/k)} < 1 + 3\epsilon, \quad \frac{\|\X \hat{\bm\beta} - \X \bm\beta\|^2}{2 k \log(p/k)} < 1 + 3\epsilon.
\]
On $\mathscr{A}_2 \cap \mathscr{A}_3$, all the eigenvalues values of
$\X_{S^\diamond}'\bm{X}_{S^\diamond}$ are between $1 - \delta'$ and
$1 + \delta'$. By definition, all the coordinates of
$\bm\beta, \hat{\bm\beta}$ and $\tilde{\bm\beta}$ vanish outside of
$S^\diamond$. Thus, Lemma \ref{lm:two_beta} gives
\begin{align*}
  \|\hat{\bm\beta} - \bm\beta\| \le \|\hat{\bm\beta} - \tilde{\bm\beta}\| + \|\tilde{\bm\beta} - \bm\beta\| & \le 
                                                                                                              \left( \sqrt{\frac{3\delta'}{1-2\delta'}} + 1 \right) \|\tilde{\bm\beta} - \bm\beta\| \\ & \le \left(\frac{1+\delta}{(1+\delta/2)(1+\delta')}\right)^{1/2} \|\tilde{\bm\beta} - \bm\beta\|.
\end{align*}
Hence, on $\mathscr{A}_1 \cap \mathscr{A}_2 \cap \mathscr{A}_3$, we have
\[
\frac{\|\hat{\bm\beta} - \bm\beta\|^2}{2 k \log(p/k)} \le \frac{1+\delta}{(1+\delta/2)(1+\delta')} \cdot \frac{\|\tilde{\bm\beta} - \bm\beta\|^2}{2 k \log(p/k)} < \frac{(1+\delta)(1+\epsilon)^2}{1+\delta'} < 1 + 3\epsilon,
\]
and similarly,
\begin{equation}\nonumber
\frac{\|\X \hat{\bm\beta} - \X \bm\beta\|^2}{2 k \log(p/k)} \le (1+\delta')\frac{\|\hat{\bm\beta} - \bm\beta\|^2}{2 k \log(p/k)} < (1+\delta')\frac{(1+\delta)(1+\epsilon)^2}{1+\delta'} = 1 + 3\epsilon.
\end{equation}
This finishes the proof.

%%%%%%%%%%%%%%%%%%%

% {\bf Remark.} If $(k\log p)/n \le \rho$ for some known
% $\rho_j \goto 0$ (for example, if $k = O(n/\log^2 p)$, we can find a
% sparsity-independent sequence $\epsilon_j \goto 0$, such that
% \[
% \max_{\|\bm\beta\|_0 \le k} \P \left( \frac{\|
%     \hat{\bm\beta}_{\textnormal{\tiny SLOPE}} - \bm\beta
%     \|^2}{2\sigma^2 k \log(p/k)} > 1 + \delta \right)
% \rightarrow 0
% \]
% for each constant $\delta > 0$. To achieve this goal, let
% \[
% d = \min\{\lfloor k/\sqrt{\rho} \rfloor, \lfloor p \rfloor \}
% \]
% and $\K$ be represented by $k, d$ the same as in the proof above. Letting $\epsilon = \rho^{1/5}$, Lemma \ref{lm:T_contains} is still valid here, since, implied by its proof, that it suffices to have $\epsilon/\sqrt{(\K\log p)/n} \goto \infty$. This choice of $\epsilon$ satisfies this, and in addition, $(1+\epsilon)^2$ can be absorbed since it vanishes. That is, for every constant $\delta > 0$, we have $1+3\epsilon < 1+\delta$ if $p$ is sufficiently large. This improvement also applies to Corollary \ref{coro:gauss_prediction_minimax}.

%%% Local Variables:
%%% mode: latex
%%% TeX-master: "paper"
%%% End:

\section{Lower bounds}
\label{sec:lowerbound}

We here prove Theorem \ref{thm:lower_beta}, the lower matching bound
for Theorem \ref{thm:gauss_minimax}, and leave the proof of Corollary
\ref{cor:lowerbound_general_pred} to Appendix
\ref{sec:proofs-sect-refs-5}.  Once again, we warm up with the
orthogonal design and develop tools that can be readily applied to the
regression case.

%%%%%%%%%%%%%%%%%%%%%%%%

\subsection{Orthogonal designs}
\label{sec:orthogonal-case-later}

Suppose $\bm y \sim \mathcal{N}(\bm \beta, \bm I_p)$. The first result
states that in this model, the squared loss for estimating 1-sparse
vectors cannot be lower than $2\log p$. The proof is in Appendix
\ref{sec:proofs-sect-refs-5}.
\begin{lemma}\label{lm:lowerbound_single}
  Let $\tau_p = (1 + o(1)) \sqrt{2\log p}$ be a sequence obeying
  $\sqrt{2\log p} - \tau_p \goto \infty$.  Consider the prior $\bm\pi$
  for $\bm\beta$, which selects a coordinate $i$ uniformly at random
  in $\{1, \ldots, p \}$, and sets $\beta_i = \tau_p$ and $\beta_j =
  0$ for $j \neq i$. For each $\epsilon > 0$,
\begin{equation}\nonumber
\inf_{\hat{\bm\beta}}\P_{\bm\pi}\lb \frac{\|\hat{\bm\beta} - \bm\beta\|^2}{2\log p} > 1 - \epsilon \rb \goto 1.
\end{equation}
\end{lemma}
Next, we state a counterpart to Theorem
\ref{thm:orth_main_general_intro}, whose proof constructs $k$
independent $1$-sparse recovery problems.
\begin{proposition}\label{prop:lowerbound_sequence}
Suppose $k/p \goto 0$. Then for any $\epsilon > 0$, we have
\begin{equation}\nonumber
\inf_{\hat{\bm\beta}}\sup_{\|\bm\beta\|_0 \le k}\P\left( \frac{\|\hat{\bm\beta} - \bm\beta\|^2}{2k\log(p/k)} > 1 - \epsilon \right) \rightarrow 1.
\end{equation}
\end{proposition}
\begin{proof}
  The fundamental duality between `min max' and `max min' gives
\begin{equation}\nonumber
  \inf_{\hat{\bm\beta}}\sup_{\|\bm\beta\|_0 \le k}\P\left( \frac{\|\hat{\bm\beta} - \bm\beta\|^2}{2k\log(p/k)} > 1 - \epsilon \right) \ge \sup_{\|\tilde{\bm\pi}\|_0 \le k} \inf_{\hat{\bm\beta}}\P_{\tilde{\bm\pi}}\left( \frac{\|\hat{\bm\beta} - \bm\beta\|^2}{2k\log(p/k)} > 1 - \epsilon \right).
\end{equation}
Above, $\tilde{\bm\pi}$ denotes any distribution on $\R^p$ such that
any realization $\bm\beta$ obeys $\|\bm\beta\|_0 \le k$, and
$\P_{\tilde{\bm\pi}}(\cdot)$ emphasizes that $\bm\beta$ follows the
prior $\tilde{\bm\pi}$, as earlier in Lemma
\ref{lm:lowerbound_single}. It is therefore sufficient to construct a
prior $\tilde{\bm \pi}$ with a right-hand side approaching one.

Assume $p$ is a multiple of $k$ (otherwise, replace $p$ with
$p_0 = k\lfloor p/k \rfloor$ and let $\bm\pi$ be supported on
$\{1, \ldots, p_0\}$). Partition $\{1, \ldots, p \}$ into $k$
consecutive blocks $\{1, \ldots, p/k\}$, $\{p/k + 1, \ldots, 2p/k\}$
and so on. Our prior is a product prior, where on each block, we
select a coordinate uniformly at random and sets its amplitude to
$\tau = (1+o(1)) \sqrt{\log(p/k)}$ and
$\sqrt{2\log (p/k)} - \tau \goto \infty$. Next, let $\hat{\bm\beta}$
be any estimator and write the loss
$\|\hat{\bm\beta} - \bm\beta\|^2 = L_1 + \cdots + L_k$, where $L_j$ is
the contribution from the $j$th block. The lemma is reduced to proving
\begin{equation}\label{eq:lower_orth_obj}
 \inf_{\hat{\bm\beta}}\P_{\bm\pi} \lb \frac{L_1 + \cdots + L_k}{2k\log(p/k)} > 1 - \epsilon \rb \goto 1.
\end{equation}

For any constant $\epsilon' > 0$, since $p/k \goto \infty$, Lemma
\ref{lm:lowerbound_single} claims that
\begin{equation}\label{eq:1_peak_ind}
  \inf_{\hat{\bm\beta}}\P_{\bm\pi} \lb 
  \frac{L_j}{2\log(p/k)} > 1 - \epsilon' \rb \goto 1
\end{equation}
uniformly over $j = 1, \ldots, k$ since distinct blocks are
stochastically independent. Set $$\bar L_j = \min\{ L_j, 2\log(p/k) \}
\le L_j.$$ On one hand,
\[
  \frac{ \E\left( \bar L_1 + \cdots + \bar L_k \right)}{2k\log(p/k)}
  \le (1-\epsilon) \cdot \P_{\bm\pi} \lb \frac{\bar L_1 + \cdots
      + \bar L_k}{2k\log(p/k)} \le 1 - \epsilon \rb +
    \P_{\bm\pi} \lb \frac{\bar L_1 + \cdots + \bar L_k}{2k\log(p/k)} >
    1 - \epsilon\right).
\]
On the other, 
\[
\frac{ \E\left( \bar L_1 + \cdots + \bar L_k \right)}{2k\log(p/k)}
\ge \frac{1 - \epsilon'}{k} \sum_{j=1}^k \P_{\bm\pi} \lb \frac{\bar L_j}{2\log(p/k)} > 1 - \epsilon' \rb.
\]
All in all, this gives 
\begin{equation}\nonumber
  \sup_{\hat{\bm\beta}}\P_{\bm\pi} \lb \frac{\bar L_1 + \cdots + \bar L_k}{2k\log(p/k)} \le 1 - \epsilon \rb \le \frac{1}{\epsilon} \cdot \left(1 - (1 - \epsilon')\inf_{\hat{\bm\beta}, j}\P_{\bm\pi} \lb \frac{\bar L_j}{2\log(p/k)} > 1 - \epsilon' \rb\right).
\end{equation}
Finally, take the limit $p \goto \infty$ in the above
inequality. Since $ \bar L_j/(2\log (p/k)) > 1 - \epsilon'$ if and only
if $L_j/(2\log (p/k)) > 1 - \epsilon'$, it follows from
\eqref{eq:1_peak_ind} that 
\begin{equation}\nonumber
  \limsup_{p\goto\infty} \sup_{\hat{\bm\beta}}\P_{\bm\pi} \lb \frac{\bar L_1 + \cdots + \bar L_k}{2k\log(p/k)} \le 1 - \epsilon \rb  \le  \frac{\epsilon'}{\epsilon}.
\end{equation}
We conclude by taking $\epsilon' \goto 0$.  
\end{proof}

\subsection{Random designs}
\label{sec:lower-bound-regr-later}

We return to the regression setup
$\bm y \sim \mathcal{N}(\X \bm \beta, \bm I_p)$, where $\X$ is our
Gaussian design.
\begin{lemma}\label{lm:regression_lower_single}
  Fix $\alpha \le 1$ and 
\[
\tau_{p,n} = \left( \sqrt{2\log p} - \log\sqrt{2\log p} \right)\left(
  1 - 2\sqrt{(\log p)/n} \right).
\]
Let $\bm\pi$ be the prior from Lemma \ref{lm:lowerbound_single} with
amplitude set to $\alpha \cdot \tau_{n,p}$. Assume $(\log p)/n \goto
0$. Then for any $\epsilon > 0$,
\begin{equation}\nonumber
  \inf_{\hat{\bm\beta}}\P_{\bm\pi}\lb \frac{\|\hat{\bm\beta} - \bm\beta\|^2}{\alpha^2 \cdot 2 \log p} > 1 - \epsilon \rb \goto 1.
\end{equation}
\end{lemma}

With this, we are ready to prove a stronger version of Theorem
\ref{thm:lower_beta}. % As an exception, $\sigma$ is not assumed to be one. 
\begin{theorem}\label{thm:reg_prob_lowerbound}[Stronger version of Theorem \ref{thm:lower_beta}]
  Consider $\bm y \sim \mathcal{N}(\X \bm \beta, \sigma^2 \bm I_p)$,
  where $\X$ is our Gaussian design, $k/p \goto 0$ and $\log(p/k)/n
  \goto 0$. Then for each $\epsilon > 0$,
\begin{equation}\nonumber
  \inf_{\hat{\bm\beta}}\sup_{\|\bm\beta\|_0 \le k}\P\left(\frac{\|\hat{\bm\beta} - \bm\beta\|^2}{\sigma^2 \cdot 2k\log(p/k)} > 1 - \epsilon \right) \rightarrow 1.
\end{equation}
\end{theorem}
\begin{proof} The proof follows that of Proposition
  \ref{prop:lowerbound_sequence}. As earlier, assume that $\sigma = 1$
  without loss of generality.  The block prior $\bm\pi$ and the
  decomposition of the loss $L$ are exactly the same as before except
  that we work with
\[
\tau = \left( \sqrt{2\log (p/k)} - \log\sqrt{2\log (p/k)} \right)\left( 1 - 2\sqrt{\log (p/k)/n} \right).
\]
Hence, it suffices to prove \eqref{eq:1_peak_ind} in the current
setting, which does not directly follow from Lemma
\ref{lm:regression_lower_single} because of correlations between the
columns of $\bm X$. Thus, write the linear model as
\begin{equation}\nonumber
\bm y = \bm X \bm\beta + \bm z = \bm X^{(1)}\bm\beta^{(1)} + \bm X^{-(1)}\bm\beta^{-(1)} + \bm z,
\end{equation}
where $\X^{(1)}$ (resp.~$\bm\beta^{(1)}$) are the first $p/k$ columns
of $\X$ (resp.~coordinates of $\bm \beta$) and $\X^{(-1)}$ all the
others.  Then 
\[
\tilde{\bm z} := \bm X^{-(1)}\bm\beta^{-(1)} + \bm z \sim
\mathcal{N}\left(\bm{0}, (\tau^2(k-1)/n + 1)\bm{I}_n \right),
\]
and is independent of $\bm X^{(1)}$ and $\bm\beta^{(1)}$. Since
$\tau^2(k-1)/n + 1 \ge 1$ and $n/\log(p/k) \goto \infty$, we can apply
Lemma \ref{lm:regression_lower_single} to
\[
\bm y = \bm X^{(1)}\bm\beta^{(1)} + \tilde{\bm z}.
\]
This establishes \eqref{eq:1_peak_ind}.
\end{proof}

%%% Local Variables:
%%% mode: latex
%%% TeX-master: "paper"
%%% End:

\section{Discussion}
\label{sec:discussion}

Regardless of the design, SLOPE is a concrete and rapidly computable
estimator, which also has intuitive statistical appeal.  For Gaussian
designs, taking Benjamini-Hochberg weights achieves asymptotic
minimaxity over large sparsity classes. Furthermore, it is likely that
our novel methods would allow us to extend our optimality results to
designs with \iid sub-Gaussian entries; for example, designs with
independent Bernoulli entries.  Since SLOPE runs without any knowledge
of the unknown degree of sparsity, we hope that taken together,
adaptivity and minimaxity would confirm the appeal of this procedure.

It would of course be of great interest to extend our results to a
broader class of designs. In particular, we would like to know what
types of results are available when the variables are correlated. In
such settings, is there a good way to select the sequence of weights
$\{\lambda_i\}$ when the rows of the design are independently sampled
from a multivariate Gaussian distribution with zero mean and
covariance $\bm \Sigma$, say? How should we tune this sequence for
fixed designs? This paper does not address such important questions,
and we leave these open for future research. 

Finally, returning to the issue of FDR control it would be interesting
to establish rigorously whether or not SLOPE controls the FDR in
sparse settings.

\section*{Acknowledgements}
W.~S. would like to thank Ma{\l}gorzata Bogdan and Iain Johnstone for helpful discussions. We thank Yuxin Chen, Rina Barber and Chiara Sabatti for their helpful comments about an early version of the manuscript. We would also like to thank the anonymous associate editor and two reviewers for many constructive comments.

% \begin{supplement} [id=suppA]
% %\sname{Supplement A}
% \stitle{Supplement to ``SLOPE is Adaptive to Unknown Sparsity and Asymptotically Minimax''}
% \slink[doi]{COMPLETED BY THE TYPESETTER}
% \sdatatype{.pdf}
% \sdescription{The supplementary materials contain proofs of some technical results in this paper.}
% \end{supplement}

\bibliographystyle{abbrv}
\bibliography{ref}

\appendix
\section{Proofs of technical results}
\label{sec:proofs-techn-results}

% {\color{red} A sequence of events
%   $\mathscr{E}_1, \mathscr{E}_2, \ldots$ is said to happen with
%   overwhelming probability if $\P(\mathscr{E}_i) \goto 1$ as
%   $i \goto \infty$.}

As is standard, we write $a_n \asymp b_n$ for two positive sequences
$a_n$ and $b_n$ if there exist two constants $C_1$ and $C_2$ (possibly
depending on $q$) such that $C_1 a_n \le b_n \le C_2 a_n$ for all
$n$. Also, we write $a_n \sim b_n$ if $a_n/b_n \goto 1$.

\subsection{Proofs for Section \ref{sec:alternatives-slope}}
\label{sec:proofs-sect-refs}

We remind the reader that the proofs in this subsection rely on some
lemmas to be stated later in the Appendix.
\begin{proof}[Proof of \eqref{eq:lasso_risk}]
  For simplicity, denote by $\hat{\bm\beta}$ the (full) Lasso solution
  $\hat{\bm\beta}_{\textnormal{\tiny Lasso}}$, and $\hat{\bm b}_S$ the
  solution to the reduced Lasso problem
\[
\underset{\bm b \in \R^k}{\mbox{minimize}}~ \frac12\|\bm y - \bm{X}_S \bm b\|^2 + \lambda\|\bm b \|_1,
\]
where $S$ is the support of the ground truth $\bm\beta$. We show that (i) \begin{equation}\label{eq:lasso_product}
\left\| \X_{\overline S}' \bm z \right\|_{\infty} \le (1+c/2)\sqrt{2\log p}
\end{equation}
and (ii) 
\begin{equation}\label{eq:lasso_inflat}
 \left\| \X'_{\overline S}\X_S(\Beta_S - \hat{\bm b}_S) \right\|_\infty < C\sqrt{(k\log^2 p)/n}
\end{equation}
for some constant $C$, both happen with probability tending to one.
Now observe that $\X_{\overline S}'(\bm y - \X_S \hat{\bm b}_S) =
\X_{\overline S}' \bm z + \X'_{\overline S}\X_S(\Beta_S - \hat{\bm
  b}_S)$.  Hence, combining \eqref{eq:lasso_product} and
\eqref{eq:lasso_inflat} and using the fact that $(k\log p)/n \goto 0$
give
\begin{equation}\nonumber
\begin{aligned}
\left\| \X_{\overline S}'(\bm y - \X_S \hat{\bm b}_S) \right\|_{\infty}  &\le  \left\| \X'_{\overline S}\X_S(\Beta_S - \hat{\bm b}_S) \right\|_\infty  + \left\| \X_{\overline S}' \bm z \right\|_{\infty}\\
& \le C\sqrt{(k\log^2 p)/n} + (1+c/2)\sqrt{2\log p} \\
& = o(\sqrt{2\log p}) + (1+c/2)\sqrt{2\log p}\\
& < (1+c)\sqrt{2\log p}
\end{aligned}
\end{equation}
with probability approaching one. This last inequality together with
the fact that $\hat{ \bm b}_S$ obeys the KKT conditions for the
reduced Lasso problem imply that padding $\hat{\bm b}_S$ with zeros on
$\barS$ obeys the KKT conditions for the full Lasso problem and is,
therefore, solution. % (This technique is sometimes called the
% primal-dual witness method in the literature on statistics.) %\ejc{This
% is saying that if $(k \log p)/n$ and $\lambda > (1 + c)\sqrt{2\log
%   p}$, then $\hat S \subset S$ right?} \wjs{Yes. The assumption
% $(k\log p)/n \goto 0$ is only used in applying
% \eqref{eq:lasso_inflat}.}

We need to justify \eqref{eq:lasso_product} and
\eqref{eq:lasso_inflat}. First, Lemmas \ref{lm:gauss_max_weak} and
\ref{lm:gauss_concentrate} imply \eqref{eq:lasso_product}. Next, to
show \eqref{eq:lasso_inflat}, we rewrite the left-hand side in 
\eqref{eq:lasso_inflat} as
\[
\X'_{\overline S}\X_S(\Beta_S - \hat{\bm b}_S) = \X'_{\overline S}\X_S(\X'_S\X_S)^{-1}(\X'_S(\bm y - \X_S\hat{\bm b}_S ) - \X_S' \bm z).
\]
By Lemma \ref{lm:bound_gradient}, we have that 
\[
\left\|  \X'_S(\bm y - \X_S\hat{\bm b}_S ) - \X_S' \bm z \right\| \le \sqrt{k}\lambda + \left\| \X_S' \bm z \right\| \le \sqrt{k}\lambda + \sqrt{32 k\log(p/k)} \le C' \sqrt{k\log p}
\]
holds with probability at least $1 - \e{-n/2} - (\sqrt{2}\mathrm{e}k/p)^k \goto 1$. 
%(The bound could be much tight. But does not affect the conclusion.) 
In addition, Lemma \ref{lm:slim_spec_refined} with $t = 1/2$ gives
\begin{equation}\nonumber
\left\| \bm{X}_{S}(\bm{X}_{S}'\bm{X}_{S})^{-1} \right\| \le \frac{1}{\sqrt{1-1/n} - \sqrt{\K/n} - 1/2} < 3
\end{equation}
with probability at least $1 - \e{-n/8} \goto 1$. Hence, from the last
two inequalities it follows that
\begin{equation}\label{eq:lasso_term_bound}
\left\| \X_S(\X'_S\X_S)^{-1}(\X'_S(\bm y - \X_S\hat{\bm b}_S ) - \X_S' \bm z) \right\| \le C'' \sqrt{k\log p}
\end{equation}
with probability at least $1 - \e{-n/2} - (\sqrt{2}\mathrm{e}k/p)^k -
\e{-n/8} \goto 1$. Since $\X'_{\overline S}$ is independent of
$\X_S(\X'_S\X_S)^{-1}(\X'_S(\bm y - \X_S\hat{\bm b}_S ) - \X_S' \bm
z)$, Lemma \ref{lm:gauss_max_weak} gives 
\begin{multline}\nonumber
\left\| \X_{\overline S}'\X_S(\X'_S\X_S)^{-1}(\X'_S(\bm y - \X_S\hat{\bm b}_S ) - \X_S' \bm z) \right\|_\infty \\
\le \sqrt{\frac{2\log p}{n}} \left\| \X_S(\X'_S\X_S)^{-1}(\X'_S(\bm y - \X_S\hat{\bm b}_S ) - \X_S' \bm z) \right\|
\end{multline}
with probability approaching one. Combining this with
\eqref{eq:lasso_term_bound} gives \eqref{eq:lasso_inflat}.

%%%%%%%%%

Let $\tilde{\bm b}_S$ be the solution to
\[
\underset{\bm b \in \R^k}{\mbox{minimize}}~ \frac12 \left\| \bm\beta_S + \X_S' \bm z - \bm b \right\|^2 + \lambda\|\bm b\|_1.
\]
To complete the proof of \eqref{eq:lasso_risk}, it suffices to
establish (i) that for any constant $\delta > 0$,
\begin{equation}\label{eq:lasso_lift_orth}
\sup_{\|\bm\beta\|_0 \le k}\P \lb \frac{\|\tilde{\bm b}_S - \bm\beta_S\|^2}{2(1+c)^2 k\log p} > 1-\delta \rb \goto 1,
\end{equation}
and (ii)
\begin{equation}\label{eq:lasso_triangle}
\|\tilde{\bm b}_S - \hat{\bm b}_S\| = o_{\P}\left(\| \tilde{\bm b}_S - \Beta_S \|\right)
\end{equation}
since \eqref{eq:lasso_lift_orth} and \eqref{eq:lasso_triangle} give
\begin{equation}\label{eq:lasso_beta_hat}
\sup_{\|\bm\beta\|_0 \le k}\P \lb \frac{\|\hat{\bm b}_S - \bm\beta_S\|^2}{2(1+c)^2 k\log p} > 1-\delta \rb \goto 1
\end{equation}
for each $\delta > 0$. Note that taking $\delta = 1 - 1/(1+c)^2$ in \eqref{eq:lasso_beta_hat} and using the fact that $\hat{\bm b}_S$ is solution to Lasso with probability approaching one finish the proof
%%%%

\textit{Proof of \eqref{eq:lasso_lift_orth}.}  Let $\beta_i =\infty$
if $i \in S$ and otherwise zero (treat $\infty$ as a sufficiently
large positive constant). For each $i \in S$, $\tilde{b}_{S,i} =
\beta_i + \bm{X}_i'\bm z - \lambda$, and
\[
|\tilde{b}_{S,i} - \beta_i| = |\bm{X}_i'\bm z - \lambda| \ge \lambda - |\bm{X}_i'\bm z|.
\]
On the event $\{\max_{i \in S} |\bm{X}_i'\bm z| \le \lambda\}$, which
happens with probability tending to one, this inequality gives
\begin{align*}
  \|\tilde{\bm b}_S - \bm\beta_S\|^2 \ge \sum_{i \in S} (\lambda - |\bm{X}_i'\bm z|)^2 & = k\lambda^2 - 2\lambda\sum_{i \in S} |\bm{X}_i'\bm z| + \sum_{i \in S} (\bm{X}_i'\bm z)^2\\
  & = (1 + o_{\P}(1))2(1+c)^2k\log p,
\end{align*}
where we have used that both $\sum_{i \in S} (\bm{X}_i'\bm z)^2$ and
$\sum_{i \in S} |\bm{X}_i'\bm z|$ are $O_{\P}(k)$. This proves the
claim.

\textit{Proof of \eqref{eq:lasso_triangle}.}  Apply Lemma
\ref{lm:two_beta} with $T$ replaced by $S$ (here each of $\hat{\bm
  b}_S, \tilde{\bm b}_S$ and $\bm\beta$ is supported on $S$). Since
$k/p \goto 0$, for any constant $\delta' > 0$, all the singular values
of $\X_S$ lie in $(1 - \delta', 1+\delta')$ with overwhelming
probability (see, for example, \cite{vershynin}). Consequently, Lemma
\ref{lm:two_beta} ensures \eqref{eq:lasso_triangle}.

% \ejc{Why do we need this piece since we have ``to complete the proof
%   of ...'' just above?}  To put an end to the proof of this theorem,
% let us make use of \eqref{eq:lasso_dual_primal} \ejc{??} and
% \eqref{eq:lasso_beta_hat}, which together yield
% \begin{equation}\nonumber
% \sup_{\|\bm\beta\|_0 \le k}\P \lb \frac{\|\hat{\bm\beta}_{\textnormal{\tiny Lasso}} - \bm\beta\|^2}{2(1+c)^2 k\log p} > 1-\delta \rb \goto 1
% \end{equation}
% for every constant $\delta > 0$. Taking $\delta = 1 - 1/(1+c)^2$ finishes the proof.

\end{proof}

%%%%%%%%% seq FDR
\begin{proof}[Proof of \eqref{eq:seq_fdr_eq}]
  We assume $\sigma = 1$ and put $\bm\lambda = \bm\lambbh$. As in the
  proof of Theorem \ref{thm:orth_main_general_intro}, we decompose the
  total loss as 
\[
|\hat{\bm\beta}_{\textnormal{\tiny Seq}} - \bm\beta \|^2 =
\|\hat{\bm\beta}_{\textnormal{\tiny Seq},S} - \bm\beta_S\|^2 +
\|\hat{\bm\beta}_{\textnormal{\tiny Seq},\overline S} -
\bm\beta_{\overline S}\|^2 = \|\hat{\bm\beta}_{\textnormal{\tiny
    Seq},S} - \bm\beta_S\|^2 + \|\hat{\bm\beta}_{\textnormal{\tiny
    Seq},\overline S}\|^2.
\]
The largest possible value of the loss off support is achieved when
$\bm y_{\overline S}$ is sequentially soft-thresholded by
$\bm\lambda^{-[k]}$. Hence, by the proof of Lemma
\ref{lm:bound-out-support}, we obtain
\[
\E \|\hat{\bm\beta}_{\textnormal{\tiny Seq},\overline S}\|^2 =
o\left(2k\log(p/k)\right)
\]
for all $k$-sparse $\bm\beta$.

Now, we turn to consider the loss on support. For any $i \in S$, the
loss is at most 
\[
\left( |z_i| + \lambda_{r(i)} \right)^2 = \lambda_{r(i)}^2 + z_i^2 + 2|z_i| \lambda_{r(i)}.
\]
Summing the above equalities over all $i \in S$ gives 
\[
\E \|\hat{\bm\beta}_{\textnormal{\tiny Seq},S} - \bm\beta_S\|^2 \le
\sum_{i=1}^k \lambda_{i}^2 + \sum_{i\in S} z_i^2 + 2\sum_{i\in S}
|z_i|\lambda_{r(i)}.
\]
Note that the first term $\sum_{i=1}^k \lambda_{i}^2 = (1 + o(1)) \,
2k\log(p/k)$, and the second term has expectation $\E \sum_{i\in S}
z_i^2 = k = o(2k\log(p/k))$, so that it suffices to show that
\begin{equation}\label{eq:seq_fdr_main}
  \E \left[2\sum_{i\in S} |z_i|\lambda_{r(i)} \right] = o\left(2k\log(p/k)\right).
\end{equation}
We emphasize that both $z_i$ and $r(i)$ are random so that
$\{\lambda_{r(i)}\}_{i \in S}$ and $\{z_i\}_{i \in S}$ may not be
independent. Without loss of generality, assume $S = \{1, \ldots, k\}$
and for $1 \le i \le k$, let $r'(i)$ be the rank of the $i$th
observation among the first $k$. Since $\bm \lambda$ is nonincreasing
and $r'(i) \le r(i)$, we have
\[
\sum_{1 \le i \le k} |z_i| \lambda_{r(i)} \le \sum_{1 \le i \le k}
|z_i| \lambda_{r'(i)} \le \sum_{1 \le i \le k} |z|_{(i)} \lambda_i,
\]  
where $|z|_{(1)} \ge \cdots \ge |z|_{(k)}$ are the order statistics of
$z_1, \ldots, z_k$.  The second inequality follows from the fact that
for any nonnegative sequences $\{a_i\}$ and $\{b_i\}$, $\sum_i a_i b_i
\le \sum_i a_{(i)} b_{(i)}$. Therefore, letting $\zeta_1, \ldots,
\zeta_k$ be \iid $\mathcal{N}(0, 1)$, \eqref{eq:seq_fdr_main} follows
from the estimate
\begin{equation}\label{eq:seq_fdr_main2}
\sum_{i=1}^k \lambda_i \E |\zeta|_{(i)} = o\left(2k\log(p/k)\right).
\end{equation}
To argue about \eqref{eq:seq_fdr_main2}, we work with the
approximations $\lambda_i \sim \sqrt{2\log(p/i)}$ and
$\E |\zeta|_{(i)} = O\left(\sqrt{2\log(2k/i)} \right)$ (see
e.g.~\eqref{eq:order_expec}), so that the claim is a consequence of
\[
\sum_{i=1}^k \sqrt{\log\frac{p}{i} \log\frac{2k}{i} } =
o\left(2k\log(p/k) \right), 
\]
which is justified as follows:
\begin{align*}
  \sum_{i=1}^k \sqrt{\log\frac{p}{i} \log\frac{2k}{i} }
  & \le k\int_0^1 \sqrt{\log\frac{p/k}{x} \log\frac2{x}} \d x \\
  & \le k\int_0^1 \sqrt{\log\frac{p}{k}} \sqrt{\log\frac2{x}} +
  \frac{\log\frac1{x}\sqrt{\log\frac2{x}}}{2\sqrt{\log(p/k)}} \d x\\
  & = C_1k\sqrt{\log\frac{p}{k}} + \frac{C_2k}{\sqrt{\log(p/k)}}
\end{align*}
for some absolute constants $C_1, C_2$. Since $\log(p/k) \goto
\infty$, it is clear that the right-hand side of the above display is
of $o(2k\log(p/k))$.
\end{proof}

\subsection{Proofs for Section \ref{sec:slope-with-orth}}
\label{sec:proofs-sect-refs-orth}

To begin with, we derive a dual formulation of the SLOPE program
\eqref{eq:slope_lambda}, which provides a nice geometrical
interpretation. This dual formulation will also be used in the proof
of Lemma \ref{lm:lifted}. Our exposition largely borrows from
\cite{slope}.

%We prove Lemma \ref{lm:lifted} from the dual perspective of \eqref{eq:slope_lambda}, from which Fact \ref{fc:slope_maj} easily follows. To be self-content, we restate this derivation in \cite{slope} as follows. 

Rewrite \eqref{eq:slope_lambda} as
\begin{equation}\label{eq:primal_slope}
\underset{\bm b, \bm r}{\mbox{minimize}} \quad \frac12\|\bm r\|^2 + \sum_{i} \lambda_i|b|_{(i)} \quad \mbox{subject to} \quad  \X\bm b + \bm r = \bm y,
\end{equation}
whose Lagrangian is
\[
\mathcal{L}(\bm{b}, \bm{r}, \bm{\nu}) : = \frac12\|\bm r\|^2 + \sum_{i} \lambda_i|b|_{(i)} - \bm\nu'(\X\bm{b} + \bm{r} - \bm{y}).
\]
Hence, the dual objective is given by
\begin{align*}
\inf_{\bm{b}, \bm{r}} ~ \mathcal{L}(\bm{b}, \bm{r}, \bm{\nu}) &= \bm\nu'\bm y - \sup_{\bm r}\left\{\bm\nu'\bm{r} - \frac12\|\bm r\|^2 \right\} - \sup_{\bm b}\left\{(\X'\bm\nu)'\bm{b} - \sum_{i} \lambda_i|b|_{(i)} \right\}\\
&= \bm\nu'\bm{y} - \frac12\|\bm\nu\|^2 -
\begin{cases}
0 &  ~ \bm\nu \in C_{\bm\lambda, \X}\\
+\infty & \mbox{otherwise},
\end{cases}
\end{align*}
where $C_{\bm\lambda, \X} := \{\bm\nu: \X' \bm\nu ~ \mbox{is majorized
  by}~ \bm\lambda\}$ is a (convex) polytope. %  \ejc{Are you sure that in
%   the def. of majorization, we would not need the last inequality to
%   be an equality?} \wjs{I am sure. The last inequality has to be equality if $\sum_{i} \lambda_i|b|_{(i)}$ is replaced by $\sum_{i} \lambda_i b_{(i)}$.}
It thus follows that the dual reads
\begin{equation}\label{eq:dual}
\underset{\bm\nu}{\mbox{maximize}} \quad \bm\nu'\bm{y} - \frac12\|\bm\nu\|^2 \quad \mbox{subject to}~ \bm\nu \in C_{\bm\lambda, \X}.
\end{equation}
The equality $\bm\nu'\bm{y} - \|\bm\nu\|^2/2 = - \|\bm y - \bm\nu\|^2/2 + \|\bm y\|^2/2 $ reveals that the dual solution $\hat{\bm\nu}$ is indeed the projection of $\bm y$ onto $C_{\bm\lambda, \X}$. The minimization of the Lagrangian over $\bm r$ is attained at $\bm r = \bm\nu$. This implies that the primal solution $\hat{\bm\beta}$ and the dual solution $\hat{\bm\nu}$ obey
\begin{equation}\label{eq:dual_stationary}
\bm y - \X\hat{\bm\beta} = \hat{\bm\nu}.
\end{equation}

%%Taking $\X$ to be the identity, we obtain a proof of Fact \ref{fc:slope_maj}. 

We turn to proving the facts.
\begin{proof}[Proof of Fact \ref{fc:norm}]
  Without loss of generality, suppose both $\bm a$ and $\bm b$ are
  nonnegative and arranged in nonincreasing order. Denote by
  $T_k^{\bm a}$ the sum of the first $k$ terms of $\bm a$ with
  $T_0^{\bm a} \triangleq 0$, and similarly for $\bm b$. We have
\[
\norm{\bm a}^2 = \sum^p_{k=1}a_k(T^{\bm a}_k - T^{\bm a}_{k-1}) =
\sum^{p-1}_{k=1}T^{\bm a}_k(a_k - a_{k+1}) + a_p T^{\bm a}_p \geq
\sum^{p-1}_{k=1}T^{\bm b}_k(a_k - a_{k+1}) + a_pT^{\bm b}_p =
\sum^p_{k=1}a_kb_k.
\]
Similarly,
\[
\norm{\bm b}^2 = \sum^{p-1}_{k=1}T^{\bm b}_k(b_k - b_{k+1}) +
b_pT^{\bm b}_p \le \sum^{p-1}_{k=1}T^{\bm a}_k(b_k - b_{k+1}) + b_pT^{\bm a}_p = \sum^p_{k=1}b_k(T^{\bm a}_k - T^{\bm a}_{k-1}) = \sum^p_{k=1}a_kb_k, 
\]
which proves the claim.
%  \\ & = \sum^p_{k=1}b_k(T^{\bm a}_k - T^{\bm a}_{k-1})\\
% &=\sum^{p-1}_{k=1}T^{\bm a}_k(b_k - b_{k+1}) + b_pT^{\bm a}_p \\ & \geq \sum^{p-1}_{k=1}T^{\bm b}_k(b_k - b_{k+1}) + b_pT^{\bm b}_p \\
% & = \norm{\bm b}^2.
% \end{align*}
\end{proof}

%%%%------------------
\begin{proof}[Proofs of Facts \ref{fc:maj_zero} and \ref{fc:slope_maj}]
  Taking $\X = \bm{I}_p$ in the dual formulation,
  \eqref{eq:dual_stationary} immediately implies that $\bm a -
  \slope{\bm a}$ is the projection of $\bm a $ onto the polytope
  $C_{\bm\lambda, \bm I_p}$. By definition, $C_{\bm\lambda, \bm I_p}$
  consists of all vectors majorized by $\bm\lambda$. Hence, $\bm a -
  \slope{\bm a}$ is always majorized by $\bm\lambda$. In particular,
  if $\bm a$ is majorized by $\bm\lambda$, then the projection $\bm a
  - \slope{\bm a}$ of $\bm a$ is identical to $\bm a$ itself. This
  gives $\slope{\bm a} = \bm 0$.
\end{proof}

%%%%%---------------------
\begin{proof}[Proof of Fact \ref{fc:monotone}]
  Assume $\bm a$ is nonnegative without loss of generality.  % {\color{red} First, we
%   show that, for all $1 \le i, j \le p$, $\left[ \slope{\bm a} \right]_i$ is nondecreasing function of $a_j$.}  \ejc{What does the last sentence mean?} \wjs{Corrected.} \ejc{As for the
%     proof, I believe it is easier to argue directly from (1.8)} \wjs{Correct me if I miss something. The cone $K := \{\bm b \in \R^p: ~ b_1 \ge b_2 \ge \cdots \ge b_p \ge 0\}$ is the monotone nonnegative cone. WLOG, assume $a_1 \ge \cdots \ge a_p \ge 0$, then the SLOPE solution is the projection of $\bm a - \bm\lambda$ onto $K$. Then, the claim we need is a consequence of
% \begin{equation}\label{eq:iso}
% \bm x \ge \bm y \Longrightarrow \mathcal{P}_K(\bm x) \ge \mathcal{P}_K(\bm y),
% \end{equation}
% where $\ge$ is the usual componentwise ordering, and $\mathcal{P}_K(\bm x)$ is the projection of $\bm x$ onto $K$. \eqref{eq:iso}is true, however, its proof is as complicated as the current proof of this lemma. On the other hand, it is well-known that
% \begin{equation}\nonumber
% \bm x \ge_K \bm y \Longrightarrow \mathcal{P}_K(\bm x) \ge_K \mathcal{P}_K(\bm y),
% \end{equation}
% see
% \url{http://www.convexoptimization.com/wikimization/index.php/Projection_on_Polyhedral_Cone}.
% }
It is intuitively obvious that
\[
\bm b \ge \bm a \quad \Longrightarrow \quad \slope{\bm b} \ge \slope{\bm
  a} ,
\]
where as usual $\bm b \ge \bm a$ means that $\bm b - \bm a
\in \R^p_+$. In other words, if the observations increase, the fitted
values do not decrease. To save time, we directly verify this claim
by using Algorithm 3 (FastProxSL1) from \cite{slope}.  By the
averaging step of that algorithm, we can see that for each $1 \le i, j
\le p$,
\[
\frac{\partial \left[ \slope{\bm a} \right]_i}{\partial a_j} = 
\begin{cases}
\frac1{\#\{1 \le k \le p: \ \left[ \slope{\bm a} \right]_k = \left[ \slope{\bm a} \right]_j\}}, & \quad
\slope{\bm a}_j = \slope{\bm a}_i > 0,\\
0, & \quad \mbox{otherwise}.  
\end{cases}
\]
This holds for all $\bm a \in \R^p$ except for a set of measure zero.
The nonnegativity of $\partial \left[ \slope{\bm a} \right]_i/\partial
a_j$ along with the Lipschitz continuity of the prox imply the
monotonicity property. A consequence is that $\| \left[ \slope{\bm a}
\right]_{\overline T}\|$ does not decrease as we let $a_i \goto
\infty$ for all $i \in T$. In the limit, $\|\left[ \slope{\bm a}
\right]_{\overline T}\|$ monotonically converges to
$\|\slopex{\bm\lambda^{-|T|}}{\bm a_{\overline T}}\|$.  This gives the
desired inequality.

\end{proof}

As a remark, we point out that the proofs of Facts \ref{fc:maj_zero}
and \ref{fc:slope_maj} suggest a very simple proof of Lemma
\ref{lm:algo-monotone}. Since $\bm a - \slope{\bm a}$ is the
projection of $\bm a$ onto $C_{\bm\lambda, \bm I_p}$,
$\|\slope{\bm a}\|$ is thus the distance between $\bm a$ and the
polytope $C_{\bm\lambda, \bm I_p}$. Hence, it suffices to find a point
in the polytope at a distance of $\|(|\bm a| - \bm\lambda)_+\|$ away
from $\bm a$. The point $\bm b$ defined as
$b_i = \min\{|a_i|, \lambda_i\}$ does the job.

%%%%%%%%%%%%%%%%%%%%%%%%
Now, we proceed to prove the preparatory lemmas for Theorem
\ref{thm:orth_main_general_intro}, namely, Lemmas
\ref{lm:bound-out-support-close} and
\ref{lm:bound-out-support-far}. The first two lemmas below can be
found in \cite{handbook}. 
% \begin{lemma}\label{lm:lipschitz}
% There exists an absolute constant $C > 0$ such that for any $0 < q_1 < q_2 \le 1/2$, we have
% \begin{equation}\nonumber
% \Phi^{-1}(1 - q_1) - \Phi^{-1}(1 - q_2) \le C \left( \sqrt{2\log 1/q_1} - \sqrt{2\log 1/q_2} \right).
% \end{equation}
% %$C$ can be chosen as $C = 1.476$.
% \end{lemma}

\begin{lemma}\label{lm:gamma_der}
Let $U$ be a $\mathrm{Beta}(a, b)$ random variable. Then
\begin{align*}
&\E\log U = (\log\Gamma(a))' - (\log\Gamma(a + b))',
\end{align*}
where $\Gamma$ denotes the Gamma function and $(\log\Gamma(x))'$ is the derivative with respect to $x$.
\end{lemma}

%&\E\log^2 U = \big((\log\Gamma(a))' - (\log\Gamma(a + b))'\big)^2 + (\log\Gamma(a))'' - (\log\Gamma(a+b))'',
%% &(\log\Gamma(m))'' = \sum_{k=m}^{\infty}\inverse{k^2} = \frac{1}{m} + O\Big(\frac{1}{m^2}\Big),
\begin{lemma}\label{lm:gamma-prop}
For any integer $m \ge 1$,
\begin{equation*}
(\log\Gamma(m))' = -\gamma + \sum_{j=1}^{m-1}\frac1{j} = \log m + O\Big(\frac{1}{m}\Big) ,
\end{equation*}
where $\gamma = 0.577215\cdots$ is the Euler constant.
\end{lemma}

%%%%%%%%%%%%%%%

\begin{lemma}\label{lm:bound-out-support-close}
Let $\bm\zeta \sim \mathcal{N}(\bm 0, \bm{I}_{p-k})$. Under the assumptions of Theorem \ref{thm:orth_main_general_intro}, for any constant $A > 0$,
\[
\frac1{2k\log(p/k)}\sum_{i = 1}^{\floor{Ak}}\E\left( |\zeta|_{(i)} - \lambbh_{k+i}\right)^2_+ \goto 0.
\]

\end{lemma}

\begin{proof}[Proof of Lemma \ref{lm:bound-out-support-close}]
  Write $\lambda_i = \lambbh_i$ for simplicity.  It is sufficient to
  prove a stronger version in which the order statistics
  $|\zeta|_{(i)}$ come from $p$ \iid $\mathcal{N}(0,1)$. The reason is
  that the order statistics will be stochastically larger, thus
  enlarging $\E\left( |\zeta|_{(i)} - \lambbh_{k+i}\right)^2_+$, since
  $(\zeta - \lambda)_+^2$ is nondecreasing in $\zeta$.  Applying the
  bias-variance decomposition, we get
\begin{equation}\label{eq:bias_var}
\E\left( |\zeta|_{(i)} - \lambda_{k+i}\right)^2_+ \le \E\left( |\zeta|_{(i)} - \lambda_{k+i}\right)^2 = \operatorname{Var}(|\zeta_{(i)}|) + \left(\E |\zeta_{(i)}| - \lambda_{k+i}\right)^2.
\end{equation}
We proceed to control each term separately. 

For the variance, a direct application of Proposition 4.2 in
\cite{orderconcentration} gives
\begin{equation}\label{eq:var_order_bound}
\operatorname{Var}(|\zeta_{(i)}|) = O\Bigl(\frac1{i\log(p/i)}\Bigr)
\end{equation}
for all $i \le p/2$. Hence,
\[
\sum_{i=1}^{\lfloor Ak \rfloor} \operatorname{Var}(|\zeta_{(i)}|)
= O\left( \sum_{i=1}^{\lfloor Ak \rfloor}
\frac1{i\log(p/i)} \right) = o(2k\log(p/k)),
\]
where the last step makes use of $\log(p/k) \goto \infty$. It remains to show that 
\begin{equation}\label{eq:bias_part}
  \sum_{i=1}^{\lfloor Ak \rfloor} \left(\E |\zeta_{(i)}| - \lambda_{k+i}\right)^2 = o(2k \log(p/k)).   
\end{equation}
Let $U_1, \ldots, U_p$ be \iid uniform random variables on $(0, 1)$
and $U_{(i)}$ be the $i^{\text{th}}$ smallest---please note that for a
change, the $U_i$'s are sorted in increasing order. We know that
$U_{(i)}$ is distributed as $\mathrm{Beta}(i, p+1-i)$ and that
$|\zeta|_{(i)}$ has the same distribution as
$\Phi^{-1}(1 - \order{U}/2)$. Making use of Lemmas
\ref{lm:gamma_der} and \ref{lm:gamma-prop} then gives
\[
\E |\zeta|_{(i)}^2 = \E \left[ \Phi^{-1}(1 - \order{U}/2) ^2\right] \sim \E \left[ 2\log(2/U_{(i)}) \right] = 2\log 2 + 2 \sum_{j=i}^p \frac1{j} = (1+o(1))2\log(p/i),
\]
where the second step follows from $(1 + o_{\P}(1))2\log(2/U_{(i)}) \le \Phi^{-1}(1 - \order{U}/2) ^2 \le 2\log(2/U_{(i)})$ for $i = o(p)$. As a result,
\begin{equation}\label{eq:order_expec}
\begin{aligned}
\E |\zeta_{(i)}| &\le \sqrt{\E |\zeta|^2_{(i)}} = (1+o(1))\sqrt{2\log(p/i)}\\
\E |\zeta_{(i)}| &= \sqrt{\E |\zeta|^2_{(i)} - \operatorname{Var}(|\zeta|_{(i)})} = (1+o(1))\sqrt{2\log(p/i)}.
\end{aligned}
\end{equation}
Similarly, since $k+i = o(p)$ and $q$ is constant, we have the approximation
\[
\lambda_{k+i} = (1+o(1))\sqrt{2\log(p/(k+i))},
\]
which together with \eqref{eq:order_expec} reveals that
\begin{equation}\label{eq:order_diff_expr}
\left(\E |\zeta_{(i)}| - \lambda_{k+i}\right)^2 \le (1+o(1)) \, 2\left[\sqrt{\log(p/i)} -  \sqrt{\log(p/(k+i))}\right]^2 + o(1)\log(p/i).
\end{equation}
The second term in the right-hand side contributes at most $o(1) \,
Ak\log(p/(Ak)) = o(1) \, 2k\log(p/k)$ in the sum
\eqref{eq:bias_part}. For the first term, we get
\[
\left[\sqrt{\log(p/i)} -  \sqrt{\log(p/(k+i))}\right]^2 = \frac{\log^2(1 + k/i)}{\left[\sqrt{\log(p/i)} +  \sqrt{\log(p/(k+i))}\right]^2} = o(1)\log^2(1 + k/i).
\]
Hence, it contributes at most
\begin{equation}
\label{eq:sq_log_sum}
\begin{aligned} 
  o(1)\sum_{i=1}^{\lfloor Ak \rfloor} \log^2(1 + k/i) & \le o(1) \sum_{i=1}^{\lfloor Ak \rfloor} k\int_{\frac{i-1}{k}}^{\frac{i}{k}}\log^2(1+1/x) \d x \\
  & = o(1)k \int_0^ A \log^2(1+1/x) \d x = o(2k\log(p/k)).
\end{aligned}
\end{equation}
Combining \eqref{eq:sq_log_sum}, \eqref{eq:order_diff_expr} and \eqref{eq:bias_part} concludes the proof.
\end{proof}

%%%%%%%%
\newcommand{\KL}{\operatorname{KL}}
\begin{lemma}\label{lm:bound-out-support-far}
Let $\bm\zeta \sim \mathcal{N}(\bm 0, \bm{I}_{p-k})$ and $A > 0$ be any constant satisfying $q(1+A)/A < 1$. Then, under the assumptions of Theorem \ref{thm:orth_main_general_intro},
\begin{equation}\nonumber
\frac1{2k\log(p/k)}\sum_{i=\ceil{Ak}}^{p-k}\E\left(|\zeta|_{(i)} - \lambbh_{k+i} \right)_+^2 \goto 0.
\end{equation}

\end{lemma}

\begin{proof}[Proof of Lemma \ref{lm:bound-out-support-far}]
  Again, write $\lambda_i = \lambbh_i$ for simplicity.  As in the
  proof of Lemma \ref{lm:bound-out-support-close} we work on a
  stronger version by assuming
  $\bm\zeta \sim \mathcal{N}(\bm 0, \bm{I}_p)$. Denote by
  $q' =q(1+A)/A$. For any $u \ge 0$, let
  $\alpha_u := \P(|\mathcal{N}(0,1)| > \lambda_{k+i} + u) = 2\Phi(-
  \lambda_{ k+i} - u)$.
  Then $\mathbb{P}(|\zeta|_{(i)} > \lambda_{ k+i} + u)$ is just the
  tail probability of the binomial distribution with $p$ trials and
  success probability $\alpha_u$. By the Chernoff bound, this
  probability is bounded as
\begin{equation}\label{eq:chernoff_bound}
\mathbb{P}(|\zeta|_{(i)} > \lambda_{ k+i} + u) \le \e{-p \KL\left( i/p\| \alpha_u \right)},
\end{equation}
where $\KL(a \| b) := a\log\frac{a}{b} + (1-a)\log\frac{1-a}{1-b}$ is the Kullback-Leibler divergence. Note that
\begin{equation}\label{eq:entropu}
\frac{\partial \KL(i/p \| b)}{\partial b} = -\frac{i/p}{b} + \frac{1-i/p}{1-b} \le -\frac{i}{pb} + 1
\end{equation}
for all $0 < b < i/p$. Hence, from \eqref{eq:entropu} it follows that
\begin{equation}\label{eq:mult_kl}
\begin{aligned} 
  \KL\left( i/p \| \alpha_u\right) - \KL\left( i/p \| \alpha_0 \right)
  = -\int^{\alpha_0}_{\alpha_u}\frac{\partial \KL}{\partial b} \d b
  & \geq \int^{\alpha_0}_{\alpha_u} \frac{i}{pb} - 1 \d b \\
  & \geq \int^{\alpha_0}_{ \e{-u\lambda_{k+i}}\alpha_0}\frac{i}{pb} -
  1 \d b \\ & = \frac{iu\lambda_{k+i}}{p} -
  \alpha_0\left(1-\e{-u\lambda_{k+i}} \right),
\end{aligned} 
\end{equation}
where the second inequality makes use of $\alpha_u \le \e{
  -u\lambda_{k+i} } \alpha_0$. With the proviso that $q(1+A)/A < 1$
and $i \ge Ak$, it follows that
\begin{equation}\label{eq:alpha0_bound}
\alpha_0 = q(k+i)/p \le q'i/p.
\end{equation}
Hence, substituting \eqref{eq:alpha0_bound} into \eqref{eq:mult_kl},
we see that \eqref{eq:chernoff_bound} yields
\begin{equation}\label{eq:kl_align}
\begin{aligned}
  \mathbb{P}(|\zeta|_{(i)} > \lambda_{ k+i} + u) & \le \e{-p\big(\KL(\frac{i}{p} \| \alpha_u) - \KL(\frac{i}{p} \| \alpha_0)\big)}\e{-p \KL(\frac{i}{p} \| \alpha_0)}\\
&  \le \e{-p\big(\KL(\frac{i}{p}\| \alpha_u) - \KL(\frac{i}{p}\| \alpha_0)\big)}\\ & \le 
  \exp\left(-iu\lambda_{k+i} + q'i\big(1-\exp\big(-u\lambda_{k+i}\big) \big)\right).
\end{aligned}
\end{equation}

With this preparation, we conlude the proof of our lemma as follows:
\begin{align*}
  \E\big(|\zeta|_{(i)} - \lambda_{k+i}\big)_+^2 & = \int_0^{\infty}\mathbb{P}\left( (|\zeta|_{(i)} - \lambda_{ k+i})_+^2 > x \right) \d x\\
  & = \int_0^{\infty}\mathbb{P}(|\zeta|_{(i)} > \lambda_{ k+i} +
  \sqrt{x} )\d x \\ & = 2\int_0^{\infty}u\mathbb{P}(|\zeta|_{(i)} >
  \lambda_{ k+i} + u)\d u, 
\end{align*}
and plugging \eqref{eq:kl_align} gives
\begin{align*}
  \E\big(|\zeta|_{(i)} - \lambda_{k+i}\big)_+^2
  & \le 2\int_0^{\infty}u \exp\Big(-iu\lambda_{k+i} + q'i\big(1-\exp\big(-u\lambda_{k+i}\big) \big)\Big)\d u\\
  & = \frac{2}{\lambda_{k+i}^2}\int_0^{\infty}x \e{-(x -
    q'(1-\e{-x}))i}\d x \\ & \le \frac{2}{\lambda_p^2}\int_0^{\infty}x
  \e{-(x - q'(1-\e{-x}))i}\d x. 
\end{align*}
This yields the upper bound
\begin{equation}\nonumber
\begin{aligned}
\sum_{i=\ceil{Ak}}^{p-k}\E\big(|\zeta|_{(i)} - \lambda_{k+i}\big)_+^2 &\le \frac{2}{\lambda_p^2}\sum_{i=\ceil{Ak}}^{p-k}\int_0^{\infty}x \e{-(x - q'(1-\e{-x}))i}\d x\\
&\le \frac{2}{\Phi^{-1}(1-q/2)^2}\sum_{i=1}^{\infty}\int_0^{\infty}x \e{-(x - q'(1-\e{-x}))i}\d x\\
& = \frac{2}{\Phi^{-1}(1-q/2)^2}\int_0^{\infty}\frac{x\e{-(x - q'(1-\e{-x}))}}{1 - \e{-(x - q'(1-\e{-x}))}}\d x.
\end{aligned}  
\end{equation}
Since the integrand obeys
\[
\lim_{x\goto 0}\frac{x\e{-(x - q'(1-\e{-x}))}}{1 - \e{-(x - q'(1-\e{-x}))}} = \frac1{1-q'}
\]
and decays exponentially fast as $x \goto \infty$, we conclude that
$\sum_{i=\ceil{Ak}}^{p-k}\E(|\zeta|_{(i)} - \lambda_{k+i})_+^2$ is
bounded by a constant. This is a bit more than we need since
$2k\log(p/k) \goto \infty$.
\end{proof}

\subsection{Proofs for Section \ref{sec:slope-under-random-all}}
\label{sec:proofs-sect-refs-3}

% The lemma below, often termed as Borell's inequality or Gaussian measure concentration inequality, is useful in bounding tail probabilities (see \cite{concentrationbook} for references). In this paper, we often use it to show that
% \[
% \P\left ( \chi_n > \sqrt{n} + t \right) \le \e{-\frac{t^2}{2}}
% \]
% for every $t > 0$, where we have also used the fact $\E\chi_n \le \sqrt{\E\chi_n^2} = \sqrt{n}$.
In this paper, we often use the Borell inequality to show that
$\P(\|\mathcal{N}(\bm 0, \bm{I}_n)\| > \sqrt{n} + t) \le
\exp(-t^2/2)$.
\begin{lemma}[Borell's inequality]\label{lm:gauss_concentrate}
  Let $\bm\zeta \sim \mathcal{N}(\bm 0, \bm{I}_n)$ and $f$ be an
  $L$-Lipschitz continuous function in $\R^n$. Then
\[
\P\left ( f(\bm\zeta) > \E f(\bm\zeta) + t \right) \le \e{-\frac{t^2}{2L^2}}
\]
for every $t > 0$.
\end{lemma}

\begin{lemma}\label{lm:gauss_max_weak}
Let $\zeta_1, \ldots, \zeta_p$ be \iid $\mathcal{N}(0, 1)$. Then
\[
\max_i |\zeta_i| \le \sqrt{2\log p}
\]
holds with probability approaching one.
\end{lemma}
The latter classical result can be proved in many different
ways. Suffices to say that it follows from a more subtle fact, namely,
that
\[
\sqrt{2\log p}\left( \max_i \zeta_i - \sqrt{2\log p} + \frac{\log\log
    p + \log 4\pi}{2\sqrt{2\log p}} \right)
\]
converges weakly to a Gumbel distribution \cite{gumbel}.

\begin{proof}[Proof of Lemma \ref{lm:lifted}]
  Let $\hat{\bm\beta}^{\textnormal{lift}}$ be the lift of $\hat{\bm
    b}_T$ in the sense that $\hat{\bm\beta}^{\textnormal{lift}}_T =
  \hat{\bm b}_T$ and $\hat{\bm\beta}^{\textnormal{lift}}_{\overline T}
  = \bm 0$ and let $|T| = m$. Further, set $\tilde{\bm\nu}:=\bm y -
  \X_T \hat{\bm b}_T = \bm y -
  \X\hat{\bm\beta}^{\textnormal{lift}}$. Applying \eqref{eq:dual} and
  \eqref{eq:dual_stationary} to the reduced SLOPE program, we get that
\[
\X_T'\tilde{\bm\nu} \preceq \bm\lambda^{[m]}.
\]
By the assumption, $\X_{\overline T}'\tilde{\bm\nu}$ is majorized by
$\bm\lambda^{-[m]}$. Hence, $\X'\tilde{\bm\nu}$---the concatenation of
$\X_T'\tilde{\bm\nu}$ and $\X_{\overline T}'\tilde{\bm\nu}$---is
majorized by $\bm\lambda = (\bm\lambda^{[m]},
\bm\lambda^{-[m]})$. This confirms that $\tilde{\bm\nu}$ is dual
feasible with respect to the full SLOPE program. If additionally we
show that
\begin{equation}\label{eq:dual_gap}
\frac12\|\bm y - \X \hat{\bm\beta}^{\textnormal{lift}}\|^2 + \sum_i\lambda_i|\hat\beta^{\textnormal{lift}}|_{(i)} = \tilde{\bm\nu}'\bm{y} - \frac12\|\tilde{\bm\nu}\|^2,
\end{equation}
then the strong duality claims that $\hat{\bm\beta}^{\textnormal{lift}}$ and $\tilde{\bm\nu}$ must, respectively, be the optimal solutions to the full primal and dual.

In fact, \eqref{eq:dual_gap} is self-evident. The right-hand side is
the optimal value of the reduced dual (i.e., replacing $\X$ and
$\bm\lambda$ by $\X_T$ and $\bm\lambda^{[m]}$ in \eqref{eq:dual}),
while the left-hand side agrees with the optimal value of the reduced
primal since
\[
\frac12\|\bm y - \X \hat{\bm\beta}^{\textnormal{lift}}\|^2 = \frac12\|\bm y - \X \hat{\bm b}_T\|^2 ~ \mbox{and} ~ \sum_{i=1}^p \lambda_i|\hat\beta^{\textnormal{lift}}|_{(i)} = \sum_{i=1}^{m} \lambda_i|\hat b_T|_{(i)}.
\]
Since the reduced primal only has linear equality constraints and is
clearly feasible, strong duality holds, and \eqref{eq:dual_gap}
follows from this.
\end{proof}

%%%%%%%%%%%%%%%
\begin{lemma}\label{lm:bound_gradient}
Let $1 \le \K < p$ be any (deterministic) integer, then
\begin{equation}\nonumber
\sup_{|T|=\K}\|\bm{X}_T' \bm z\| \le \sqrt{32\K\log(p/\K)}
\end{equation}
with probability at least $1-\e{-n/2}-(\sqrt{2}\mathrm{e} \K /p)^{\K}$. Above, the supremum is taken over all the subsets of $\{1,\ldots,p\}$ with cardinality $\K$.
\end{lemma}

\begin{proof}[Proof of Lemma \ref{lm:bound_gradient}]
Conditional on $\bm z$, it is easy to see that $\bm{X}' \bm z$ is distributed as \iid centered Gaussian random variables with variance $\|\bm z\|^2/n$. This observation enables us to write
\begin{equation}\nonumber
\bm{X}' \bm z \eqd \frac{\|\bm z\|}{\sqrt{n}}(\zeta_1, \ldots, \zeta_p),
\end{equation}
where $\bm\zeta := (\zeta_1,\ldots, \zeta_p)$ consists of \iid $\mathcal{N}(0,1)$ independent of $\|\bm z\|$. Hence, it is sufficient to prove that
\begin{equation}\nonumber
\|\bm z\| \le 2\sqrt{n}, \quad|\zeta|_{(1)}^2+\cdots + |\zeta|_{(\K)}^2 \le 8\K\log(p/\K)
\end{equation}
simultaneously with probability at least
$1-\e{-n/2}-\left(\sqrt{2}\mathrm{e}\K/p\right)^{\K}$. From Lemma
\ref{lm:gauss_concentrate}, we know that $\mathbb{P}(\|\bm z\| >
2\sqrt{n}) \le \e{-n/2}$ so we just need to establish the other
inequality. To this end, observe that
\begin{align*}
  \mathbb{P}\left( |\zeta|_{(1)}^2+\cdots + |\zeta|_{(\K)}^2 > 8\K\log(p/\K) \right) & \le \frac{\E\e{\frac{1}{4}\left( |\zeta|_{(1)}^2+\cdots + |\zeta|_{(\K)}^2 \right)}}{\e{2\K\log\frac{p}{\K}}}\\
  & \le \frac{\sum_{i_1< \cdots < i_{\K}}\E\e{\frac{1}{4}\left(
        |\zeta|_{i_1}^2+\cdots + |\zeta|_{i_{\K}}^2
      \right)}}{\e{2\K\log\frac{p}{\K}}} \\ & =\frac{{p\choose
      \K}2^{\K/2}}{\e{2\K\log\frac{p}{\K}}} \\ & \le
  \Big{(}\frac{\sqrt{2}\ee{1}\K}{p}\Big{)}^{\K}.
\end{align*}
\end{proof}

%%%%%%%%%%%%%%%%%

We record an elementary result which simply follows from
$\Phi^{-1}(1-c/2) \le \sqrt{2\log 1/c}$ for each $0 < c < 1$.
\begin{lemma}\label{lm:bh_seq_bound}
  Fix $0 < q < 1$. Then for all $1\le k \le p/2$,
\[
\sum_{i=1}^k \left(\lambbh_i \right)^2 \le C_q \cdot k\log(p/k),
\]
for some constant $C_q > 0$. 
\end{lemma}

In the next two lemmas, we use the \BH critical values $\bm\lambbh$ to
majorize sequences of Gaussian order statistics.  Again, $\bm a
\preceq \bm b$ means that $\bm b$ majorizes $\bm a$.
\begin{lemma}\label{lm:main_lift_maj}
Given any constant $c > 1/(1-q)$, suppose $\max\{ ck, k + d\} \le \K < p$ for any (deterministic) sequence $d$ that diverges to $\infty$. Let $\zeta_1, \ldots, \zeta_{p-k}$ be \iid $\mathcal{N}(0,1)$. Then
\[
\left( |\zeta|_{(\K-k+1)}, |\zeta|_{(\K-k+2)}, \ldots, |\zeta|_{(p-k)} \right) \preceq \left( \lambbh_{\K+1}, \lambbh_{\K+2}, \ldots, \lambbh_{p} \right)
\] 
with probability approaching one.
\end{lemma}

\begin{proof}[Proof of Lemma \ref{lm:main_lift_maj}]
  It suffices to prove the stronger case where
  $\bm \zeta \sim \mathcal{N}(\bm 0, \bm{I}_p)$. Let
  $U_1, \ldots, U_p$ be \iid uniform random variables on $[0, 1]$ and
  $U_{(1)} \le \cdots \le U_{(p)}$ the corresponding order
  statistics. Since
\[
(|\zeta|_{(\K-k+1)}, \ldots, |\zeta|_{(p-k)}) \eqd \left( \Phi^{-1}(1-U_{(\K-k+1)}/2), \ldots, \Phi^{-1}(1-U_{(p-k)}/2) \right),
\]
the conclusion would follow from
\[
\P\left( U_{(\K-k+j)} \ge q(\K + j)/p, ~ \forall j \in \{1, \ldots, p
  - \K\} \right) \goto 1.
\]

Let $E_1, \ldots, E_{p+1}$ be \iid exponential random variables with
mean 1 and denote by $T_i = E_1 + \cdots + E_i$. Then the order
statistics $U_{(i)}$ have the same joint distribution with $T_i/T_{p+1}$. Fixing
an arbitrary constant $q' \in (q, 1 - 1/c)$, we have 
\[
\P\left( U_{(\K-k+j)} \ge q(\K + j)/p, \, \forall j\right) \ge
\P\left( T_{\K-k+j} \ge q'(\K + j), \, \forall j\right) - \P\left(
  T_{p+1} > q'p/q \right).
\]
Since $\P\left( T_{p+1} > q'p/q \right) \goto 0$ by the law of large
numbers, it is sufficient to prove
\begin{equation}\label{eq:random_walk}
  \P\left( T_{\K-k+j} \ge q'(\K + j), \, \forall j \in \{1, \ldots, p
    - \K\} \right) \goto 1.
\end{equation}
This event can be rewritten as
\[
T_{\K-k+j} - T_{\K-k} - q'j \ge q'\K - T_{\K-k}
\]
for all $1\le j \le p - \K$. Hence, \eqref{eq:random_walk} is reduced to proving
\begin{equation}\label{eq:random_walk2}
\P\left( \min_{1\le j \le p-\K}T_{\K-k+j} - T_{\K-k} - q'j \ge q'\K - T_{\K-k} \right) \goto 1.
\end{equation}

As a random walk, $T_{\K-k+j} - T_{\K-k} - q'j$ has \iid increments with mean $1 - q' > 0$ and variance 1. Thus $\min_{1\le j \le p-\K}T_{\K-k+j} - T_{\K-k} - q'j$ converges weakly to a bounded random variable in distribution. Consequently, \eqref{eq:random_walk2} holds if one can demonstrate that $q'\K - T_{\K-k}$ diverges to $-\infty$ as $p \goto \infty$ in probability. To see this, observe that 
\[
q'\K - T_{\K-k} = \frac{q'\K}{\K-k}(\K-k) - T_{\K-k} \le \frac{q'c}{c-1}(\K-k) - T_{\K-k},
\]
where we use the fact $\K \ge ck$. Under our hypothesis $q'c/(c-1) <
1$, the process $\{ q'ct/(c-1) - T_t: t \in \mathbb{N} \}$ is a random
walk drifting towards $-\infty$. Recognizing that $\K - k \ge d \goto
\infty$, we see that $q'c(\K-k)/(c-1) - T_{\K-k}$ (weakly) diverges to
$-\infty$ since it corresponds to a position of the preceding random
walk at $t \goto \infty$. This concludes the proof.

\end{proof}

\begin{lemma}\label{lm:maj_all_order}
  Let $\zeta_1, \ldots, \zeta_{p-k}$ be \iid $\mathcal{N}(0,1)$. Then
  there exists a constant $C_q$ only depending on $q$ such that
\[
\left( \zeta_1, \ldots, \zeta_{p-k} \right) \preceqm C_q \cdot
\sqrt{\frac{\log p}{\log(p/k)}}\left( \lambbh_{k+1}, \ldots, \lambbh_p
\right)
\]
with probability tending to one as $p \goto \infty$ and $k/p \goto 0$.
\end{lemma}

\begin{proof}[Proof of Lemma \ref{lm:maj_all_order}]
  Let $U_1, \ldots, U_{p-k}$ be \iid uniform random variables on $[0,
  1]$ and replace $\zeta_i$ by $\Phi^{-1}(1 - U_i/2)$. Note that 
\[
\Phi^{-1}\left(1 - U_i/2\right) \le \sqrt{2\log\frac2{U_i}}, \quad
\lambbh_{k+i} \asymp \sqrt{2\log\frac{2p}{k+i}}; 
\]
Hence, it suffices to prove that for some constant $\kappa_q'$,
\begin{equation}\label{eq:maj_logs}
  \log(2/U_{(i)}) \log(p/k) \le \kappa_q' \cdot \log p \cdot \log(2p/(k+i))
\end{equation}
holds for all $i = 1, \ldots, p-k$ with probability approaching
one. Applying the representation given in the proof of Lemma
\ref{lm:main_lift_maj} and noting that $T_{p+1} = (1+o_{\P}(1))p$, we
see that \eqref{eq:maj_logs} is implied by
\begin{equation}\label{eq:maj_logs2}
\log(3p/T_i) \log(p/k) \le \kappa_q' \cdot  \log p \cdot \log(2p/(k+i)).
\end{equation}
We consider $i \le 4\sqrt{p}$ and $i > 4\sqrt{p}$ separately.

Suppose first that $i \le 4\sqrt{p}$. In this case,
\[
\log(2p/(k+i)) = (1+o(1)) \log(p/k).
\]
Thus \eqref{eq:maj_logs2} would follow from 
\begin{equation}\nonumber
  \log(3p/T_i) = O(\log p) 
\end{equation}
for all such $i$. This is, however, self-evident since $T_i \ge E_1
\ge 1/p$ with probability $1 - \e{-1/p} = o(1)$. 

Suppose now that $i > 4\sqrt{p}$. In this case, we make use of the
fact that $T_i > i/2 - \sqrt{p}$ for all $i$ with probability tending
to one as $p \goto \infty$. Then we prove a stronger result, namely,
\begin{equation}\nonumber
  \log\frac{3p}{i/2 - \sqrt{p}} \cdot \log\frac{p}{k} \le \kappa_q'\log p \cdot \log\frac{2p}{k+i}.
\end{equation}
for all $i > 4\sqrt{p}$. This follows from the two observations below:
\[
\log\frac{3p}{i/2 - \sqrt{p}} \asymp \log\frac{p}{i}, ~
\log\frac{2p}{k+i} \ge \min\left\{ \log\frac{p}{i}, \log\frac{p}{k}
\right\}.
\]
\end{proof}

%%%%%%%%%%%%%%%%%%%%%%%%%%%%%%%%
%%%%%%%%%%%%%%%%%%%%%%%%%%%%%%%%%%%%%%%%%%
In the proofs of the next two lemmas, namely, Lemma
\ref{lm:slim_spec_refined} and Lemma \ref{lm:noise_maj}, we introduce
an orthogonal matrix $\bm{Q} \in \R^{n\times n}$ that obeys
\[
\bm{Qz} = \left( \| \bm z \|, 0, \ldots, 0\right).
\]
In the proofs, $\bm Q$ is further set to be measurable with respect to
$\bm z$. Hence, $\bm Q$ is independent of $\X$. There are many options
available to construct such a $\bm Q$, including the Householder
transformation. Set 
\begin{equation}\nonumber
\bm{W} = 
\begin{bmatrix}
\tilde{\bm{w}}\\
\tilde{\bm W}
\end{bmatrix}
:=\bm Q \bm X,
\end{equation}
where $\tilde{\bm w} \in \R^{1 \times p}$ and $\tilde{\bm W} \in \R^{(n-1) \times p}$. The independence between $\bm Q$ and $\X$ suggests that $\bm W$ is still a Gaussian random matrix, consisting of \iid $\mathcal{N}(0,1/n)$ entries. Note that
\[
\X_i'\bm{z} = (\bm{Q}\X_i)'(\bm{Q}\bm{z}) = \|\bm z\|(\bm{Q}\X_i)_1 = \|\bm z\| \tilde{w}_i.
\]
This implies that $S^\star$ is constructed as the union of $S$ and the $\K - k$ indices in $\{1, \ldots, p\}\setminus S$ with the largest $|\tilde{w}_i|$. Since $\tilde{\bm w}$ and $\tilde{\bm W}$ are independent, we see that both $\tilde{\bm W}_{\overline{S^\star}}$ and $\tilde{\bm W}_{S^\star}$ are also Gaussian random matrices. These points are crucial in the proof of these two lemmas.

%%%%%%%%%%%%%%%%%%%%%%%%%%%%%%%%%%%%%%%%%%%%
\begin{lemma}\label{lm:slim_spec_refined}
  Let $k < \K < \min\{n, p \}$ be any (deterministic) integer. Denote
  by $\sigma_{\min}$ and $\sigma_{\max}$, respectively, the smallest
  and the largest singular value of $\X_{S^\star}$. Then for any $t >
  0$,
\[
\sigma_{\min} > \sqrt{1-1/n}-\sqrt{\K/n}-t
\] 
holds with probability at least $1 - \e{-nt^2/2}$. Furthermore,
\[
\sigma_{\max} < \sqrt{1-1/n}+\sqrt{\K/n} + \sqrt{8\K \log (p/\K)/n} + t
\]
holds with probability at least $1-\e{-nt^2/2} - (\sqrt{2}\ee{1}\K/p)^{\K}$.
\end{lemma}

\begin{proof}[Proof of Lemma \ref{lm:slim_spec_refined}]
  Recall that $\tilde{\bm W}_{S^\star} \in \R^{(n-1) \times \K}$ is a
  Gaussian design with \iid $\mathcal{N}(0,1/n)$ entries. Since
  $\bm{W}_{S^{\star}}$ and $\X_{S^\star}$ have the the same set of
  singular values, we consider $\bm{W}_{S^{\star}}$.  

  Classical theory on Wishart matrices (see \cite{vershynin}, for
  example) asserts that (i) all the singular values of $\tilde{\bm
    W}_{S^\star}$ are larger than $\sqrt{1-1/n}-\sqrt{\K/n}-t$ with
  probability at least $1-\e{-nt^2/2}$, and (ii) are all smaller than
  $\sqrt{1-1/n}+\sqrt{\K/n} + t$ with probability at least
  $1-\e{-nt^2/2}$.  Clearly, all the singular values larger of ${\bm
    W}_{S^\star}$ are at least as large as $\sigma_{\min}(\tilde{\bm
    W}_{S^\star})$. Thus, (i) yields the first claim. For the other,
  Lemma \ref{lm:bound_gradient} asserts that the event $\|\tilde{\bm
    w}_{S^\star} \| \le \sqrt{8\K\log(p / \K)}$ happens with
  probability at least $1 - (\sqrt{2}\ee{1}\K/p)^{\K}$. On this event,
\[
\|\bm{W}_{S^{\star}}\| \le \sqrt{\| \tilde{\bm W}_{S^\star}\|^2 + 8\K\log(p/\K)}, 
\]
where $\|\cdot\|$ denotes the spectral norm. Hence, (ii) gives
\[
\|\bm{W}_{S^{\star}}\| \le \| \tilde{\bm W}_{S^\star}\| +
\sqrt{8\K\log(p / \K)}\le \sqrt{1-1/n}+\sqrt{\K/n} + t+\sqrt{ 8\K
  \log(p / \K)}
  \]
  with probability at least $1-\e{-nt^2/2} -
  (\sqrt{2}\ee{1}\K/p)^{\K}$.

\end{proof}

%%%%%%%%%%%%%%%%%%%%%%%%%%%%

\begin{lemma}\label{lm:noise_maj}
  Denote by $\hat{\bm b}_{S^{\star}}$ the solution to the reduced
  SLOPE problem \eqref{eq:reduced} with $T = S^{\star}$ and
  $\bm\lambda = \bm\lambe$. Keep the assumptions from Lemma
  \ref{lm:main_lift_maj}, and additionally assume $\K/\min\{n, p\}
  \goto 0$. Then there exists a constant $C_q$ only depending on $q$
  such that
\begin{equation}\nonumber
  \bm{X}'_{\overline{S^\star}}\bm{X}_{S^{\star}}(\bm\beta_{S^{\star}} - \hat{\bm b}_{S^{\star}})  \preceqm  C_q \cdot \sqrt{\frac{\K \log p}{n}}\left( \lambbh_{\K+1}, \ldots, \lambbh_p \right)
\end{equation}
with probability tending to one. 
\end{lemma}

\begin{proof}[Proof of Lemma \ref{lm:noise_maj}]
  In this proof, $C$ is a constant that only depends on $q$ and whose
  value may change at each occurrence.  Rearrange the objective term
  as
\begin{equation}\nonumber
\begin{aligned}
\X'_{\overline{S^\star}}\X_{S^{\star}}(\Beta_{S^{\star}} - \hat{\bm b}_{S^{\star}}) &= \X'_{\overline{S^\star}}\X_{S^{\star}}(\X_{S^{\star}}'\X_{S^{\star}})^{-1}(\X_{S^{\star}}'(\bm{y}-\X_{S^{\star}}\hat{\bm b}_{S^{\star}})-\X_{S^{\star}}' \bm{z})\\
&=\X'_{\overline{S^\star}}\bm{Q}'\bm{Q}\X_{S^{\star}}(\X_{S^{\star}}'\X_{S^{\star}})^{-1}(\X_{S^{\star}}'(\bm y-\X_{S^{\star}}\hat{\bm b}_{S^{\star}})-\X_{S^{\star}}'\bm{z})\\
&= \X'_{\overline{S^\star}}\bm{Q}' \bm\xi,
\end{aligned}
\end{equation}
where
\[
\bm\xi := \bm{Q}\X_{S^{\star}}(\X_{S^{\star}}'\X_{S^{\star}})^{-1}\left(\X_{S^{\star}}'(\bm y-\X_{S^{\star}}\hat{\bm b}_{S^{\star}})-\X_{S^{\star}}' \bm z \right).
\]
For future usage, note that $\bm\xi$ only depends on $\widetilde{\bm w}$ and $\widetilde{\bm W}_{S^\star}$ and is, therefore, independent of $\widetilde{\bm W}_{\overline{S^\star}}$.

We begin by bounding $\|\bm\xi\|$. It follows from the KKT condition
of SLOPE that $\X_{S^{\star}}'(\bm y-\X_{S^{\star}}\hat{\bm
  b}_{S^{\star}})$ is majorized by $\bm\lambda^{[\K]}$. Hence, it
follows from Fact~\ref{fc:norm} that
\begin{equation}\label{eq:maj_lambx}
\left\| \X_{S^{\star}}'(\bm y-\X_{S^{\star}}\hat{\bm b}_{S^{\star}})\right\| \le \| \bm \lambda^{[\K]}\|.
\end{equation}
Lemma \ref{lm:slim_spec_refined} with $t = 1/2$ gives
\begin{equation}\label{eq:inv_sing}
\left\| \bm{X}_{S^{\star}}(\bm{X}_{S^{\star}}'\bm{X}_{S^{\star}})^{-1} \right\| \le \left(\sqrt{1-1/n} - \sqrt{\K/n} - 1/2 \right)^{-1} < 2.01
\end{equation}
with probability at least $1 - \e{-n/8}$ for sufficiently large $p$, where in the last step we have used $\K/n \goto 0$. Hence, from \eqref{eq:maj_lambx} and \eqref{eq:inv_sing} we get
\begin{equation}\label{eq:lm_noise_maj_0}
\begin{aligned}
  \|\bm\xi\| &\le \left\| \bm{X}_{S^{\star}}(\bm{X}_{S^{\star}}'\bm{X}_{S^{\star}})^{-1} \right\| \cdot \left\| \X_{S^{\star}}'(\bm y-\X_{S^{\star}}\hat{\bm b}_{S^{\star}})-\X_{S^{\star}}' \bm z \right\|\\
  &\le 2.01\left(\|{\bm \lambda}^{[\K]}\| + 4\sqrt{2\K\log(p/\K)}\right)\\
  &\le 2.01\left( (1+\epsilon)\sqrt{C} + 4\sqrt{2} \right)
  \sqrt{\K\log(p/\K)}\\
  & = C \cdot \sqrt{\K\log(p/\K)}
\end{aligned}
\end{equation}
with probability at least $1 - \e{-n/2} -
(\sqrt{2}\mathrm{e}\K/p)^{\K} - \e{-n/8} \goto 1$; we used Lemma
\ref{lm:bound_gradient} in the second line and Lemma
\ref{lm:bh_seq_bound} in the third. \eqref{eq:lm_noise_maj_0} will
help us in finishing the proof.

Write
\begin{equation}\label{eq:inflat_two_term}
\X'_{\overline{S^\star}}\X_{S^{\star}}(\Beta_{S^{\star}} - \hat{\bm b}_{S^{\star}}) = \X_{\overline{S^\star}}' \bm{Q}'\bm\xi = \bm{W}_{\overline{S^\star}}' \bm\xi = \left( \tilde{\bm w}'_{\overline{S^\star}}, \bm 0 \right) \bm\xi + \left(\bm 0, \tilde{\bm W}'_{\overline{S^\star}} \right) \bm\xi.
\end{equation}
It follows from Lemma \ref{lm:main_lift_maj} that $\tilde{\bm
  w}_{\overline {S^\star}}$ is majorized by $\left( \lambbh_{\K+1},
  \lambbh_{\K+2}, \ldots, \lambbh_{p} \right)/\sqrt{n}$ in
probability. As a result, the first term in the right-hand side obeys
\begin{equation}\label{eq:maj_low_w}
  \left( \tilde{\bm w}'_{\overline{S^\star}}, \bm 0 \right) \bm\xi = \xi_1\cdot \tilde{\bm w}'_{\overline{S^\star}} \preceqm \|\bm\xi\| \cdot \tilde{\bm w}'_{\overline{S^\star}}
  \preceqm C \cdot \sqrt{\frac{\K}{n}\log\frac{p}{\K}}\left( \lambbh_{\K+1}, \lambbh_{\K+2}, \ldots, \lambbh_{p} \right)
\end{equation}
with probability tending to one. For the second term, by exploiting
the independence between $\bm\xi$ and $\tilde{\bm
  W}_{\overline{S^\star}}$, we have
\[
\left(\bm 0, \tilde{\bm W}'_{\overline{S^\star}} \right) \bm\xi \eqd \sqrt{\frac{\xi_2^2 + \cdots + \xi_n^2}{n}}(\zeta_1, \ldots, \zeta_{p-\K}),
\]
where $\zeta_1, \ldots, \zeta_{p-\K}$ are \iid $\mathcal{N}(0,
1/n)$. Since $\K/p \goto 0$, applying Lemma \ref{lm:maj_all_order}
gives 
\[
\left( \zeta_1, \ldots, \zeta_{p-\K} \right) \preceqm C \cdot \sqrt{\frac{\log p}{\log(p/\K)}}\left( \lambbh_{\K+1}, \ldots, \lambbh_p \right)
\]
with probability approaching one.  Hence, owing to
\eqref{eq:lm_noise_maj_0}, 
\begin{equation}\label{eq:maj_up_w}
  \left(\bm 0, \tilde{\bm W}'_{\overline{S^\star}} \right) \bm\xi \preceqm C \cdot \sqrt{\frac{\K\log p}{n}}\left( \lambbh_{\K+1}, \ldots, \lambbh_p \right)
\end{equation}
holds with probability approaching one.  Finally, combining
\eqref{eq:maj_low_w} and \eqref{eq:maj_up_w} gives that
\begin{align*}
  \X'_{\overline{S^\star}}\X_{S^{\star}}(\Beta_{S^{\star}} - \hat{\bm b}_{S^{\star}})  & =  \left( \tilde{\bm w}'_{\overline{S^\star}}, \bm 0 \right) \bm\xi + \left(\bm 0, \tilde{\bm W}'_{\overline{S^\star}} \right) \bm\xi \\
  &\preceq C \cdot \left(\sqrt{\frac{\K}{n}\log\frac{p}{\K}} + \sqrt{\frac{\K\log p}{n}}\right) \cdot \left( \lambbh_{\K+1}, \ldots, \lambbh_p \right) \\
  &\preceqm C \cdot \sqrt{\frac{\K\log p}{n}}\left( \lambbh_{\K+1},
    \ldots, \lambbh_p \right)
\end{align*}
holds with probability tending to one.
\end{proof}

\subsection{Proofs for Section \ref{sec:lowerbound}}
\label{sec:proofs-sect-refs-5}

\begin{lemma}\label{lm:many_large_bonf}
  Keep the assumptions from Lemma \ref{lm:lowerbound_single} and let
  $\zeta_1,\ldots, \zeta_p$ be \iid $\normaldist(0, 1)$. Then
\[
\# \left\{2\le i \le p: \zeta_i > \tau + \zeta_1 \right\} \goto \infty
\] 
in probability.
\end{lemma}

\begin{proof}[Proof of Lemma \ref{lm:many_large_bonf}]
  With probability tending to one, $\tau' := \tau + \zeta_1$ also
  obeys $\tau'/\sqrt{2\log p} \goto 1$ 
% \wjs{It means that the ratio
%  is tending to one. This notion is also used in ABDJ.}
  and $\sqrt{2\log p} - \tau' \goto \infty$.  This shows that we only
  need to prove a simpler version of this lemma, namely, $\#
  \left\{1\le i \le p: \zeta_i > \tau \right\} \goto \infty$ in probability.

  Put $\Delta = \sqrt{2\log p} - \tau = o(\sqrt{2\log p})$ and $a =
  \P(\xi_1 > \tau)$. Then, $\# \left\{1\le i \le p: \zeta_i > \tau
  \right\}$ is a binomial random variable with $p$ trials and success
  probability $a$. Hence, it suffices to demonstrate that $ap \goto
  \infty$. To this end, note that 
  \begin{align*}
    a = 1 - \Phi(\tau) \sim \frac1{\tau}\frac1{\sqrt{2\pi}}\e{-\frac{\tau^2}{2}} & \asymp \frac1{\sqrt{2\log p}}\e{-\log p - \Delta^2/2 + \Delta\sqrt{2\log p}}\\
    & = \frac1{p\sqrt{2\log p}}\e{(1+o(1))\Delta\sqrt{2\log p}},
\end{align*}
which gives
\[
ap \asymp \frac1{\sqrt{2\log p}}\e{(1+o(1))\Delta\sqrt{2\log p}}. 
\]
Since $\Delta \goto \infty$ (in fact, it is sufficient to have
$\Delta$ bounded away from 0 from below), we have
\[
\e{(1+o(1))\Delta\sqrt{2\log p}}/\sqrt{2\log p} \goto \infty,
\] 
as we wish.
\end{proof}

\begin{proof}[Proof of Lemma \ref{lm:lowerbound_single}]
For sufficiently large $p$, $2(1-\epsilon)\log p \le (1-\epsilon/2)\tau^2$. Hence, it is sufficient to show
\begin{equation}\nonumber
\P_{\bm\pi} \lb \|\hat{\bm\beta} - \bm\beta\|^2 \le (1 - \epsilon/2)\tau^2 \rb \goto 0
\end{equation}
uniformly for all estimators $\hat{\bm\beta}$. Letting $I$ be the
random coordinate, 
\begin{equation}\nonumber
\|\hat{\bm\beta} - \bm\beta\|^2 =  \sum_{j\ne I} \hat\beta_j^2  + (\hat\beta_I - \tau)^2 = \|\hat{\bm\beta}\|^2 + \tau^2 - 2\tau\hat\beta_I,
\end{equation}
which is smaller than or equal to $(1 - \epsilon/2)\tau^2$ if and only if
\begin{equation}\nonumber
\hat\beta_I \ge \frac{2\|\hat{\bm\beta}\|^2 + \epsilon\tau^2}{4\tau}.
\end{equation}
Denote by $A = A(\bm y;\hat{\bm\beta})$ the set of all $i\in \{1, \ldots, p\}$ such that $\hat\beta_i \ge (2\|\hat{\bm\beta}\|^2 + \epsilon\tau^2)/(4\tau)$, and let $\hat b$ be the minimum value of these $\hat\beta_i$. Then
\begin{equation}\nonumber
  \hat b \ge \frac{2\|\hat{\bm\beta}\|^2 + \epsilon\tau^2}{4\tau} \ge \frac{2|A|\hat{b}^2 + \epsilon\tau^2}{4\tau} \ge \frac{2\sqrt{2|A| \hat b^2 \cdot \epsilon\tau^2 }}{4\tau},
\end{equation}
which gives
\begin{equation}\label{eq:a_card_bound}
|A| \le 2/\epsilon.
\end{equation}
Recall that $\|\hat{\bm\beta} - \bm\beta\|^2 \le (1-\epsilon/2)\tau^2$ if and only if $I$ is among these $|A|$ components. Hence,
\begin{equation}\label{eq:peak_prob_bound}
\P_{\bm\pi} \lb \|\hat{\bm\beta} - \bm\beta\|^2 \le (1 - \epsilon/2)\tau^2 \Big{|} \bm y \rb = \P_{\bm\pi}(I \in A | \bm y) 
= \sum_{i\in A}\P_{\bm\pi}(I = i|\bm y) = \frac{\sum_{i\in A} \e{\tau y_i}}{\sum_{i=1}^p \e{\tau y_i}},
\end{equation}
where we use the fact that $A$ is almost surely determined by $\bm
y$. Since \eqref{eq:peak_prob_bound} is maximal if $A$ is the set of
indices with the largest $y_i$'s, \eqref{eq:peak_prob_bound} and
\eqref{eq:a_card_bound} together yield
\begin{multline}\nonumber
\P_{\bm\pi} \lb \|\hat{\bm\beta} - \bm\beta\|^2 \le (1 - \epsilon/2)\tau^2 \rb \\
\le \P_{\bm\pi}\left( y_I = \tau + z_I \ \mbox{is at least the} ~\lceil 2/\epsilon\rceil^{\text{th}} ~  \mbox{largest among} ~ y_1, \ldots, y_p\right) \goto 0,
\end{multline}
where the last step is provided by Lemma \ref{lm:many_large_bonf}.

\end{proof}

%%%%%%%%%%%%%%% design

\begin{proof}[Proof of Lemma \ref{lm:regression_lower_single}]
  To closely follow the proof of Lemma \ref{lm:lowerbound_single},
  denote by $A = A(\bm y, \X; \hat{\bm\beta})$ the set of all $i\in
  \{1, \ldots, p\}$ such that $\hat\beta_i \ge (2\|\hat{\bm\beta}\|^2
  + \epsilon\alpha^2\tau^2)/(4\alpha\tau)$, and keep the same notation
  $\hat b$ as before. Then $\|\hat{\bm\beta} - \bm\beta\|^2 \le
  (1-\epsilon/2)\alpha^2\tau^2$ if and only if $I \in A$. Hence,
\begin{equation}\label{eq:peak_prob_bound_reg}
\begin{aligned}
  \P_{\bm\pi} \lb \|\hat{\bm\beta} - \bm\beta\|^2 \le  (1 - \epsilon/2)\alpha^2\tau^2 \Big{|} \bm y, \X\rb & = \P_{\bm\pi}(I \in A |\bm y, \X) \\
 &  = \sum_{i\in A}\P_{\bm\pi}(I = i|\bm y, \X) \\
& = \frac{\sum_{i\in
      A}\exp\bigl(\alpha\tau \X_i'\bm y - \alpha^2\tau^2
        \|\X_i\|^2/2\bigr)}{\sum_{i=1}^p \exp\bigl(\alpha\tau \X_i'\bm y - \alpha^2\tau^2
        \|\X_i\|^2/2\bigr)}
\end{aligned}
\end{equation}
and this quantity is maximal if $A$ is the set of indices $i$ with the
largest values of $\X_i'\bm y/\alpha - \tau \|\X_i\|^2/2$. As shown in
Lemma \ref{lm:lowerbound_single}, $|A| \le 2/\epsilon$, which gives
\begin{multline}\label{eq:2_delta_large_reg}
\P_{\bm\pi} \lb \|\hat{\bm\beta} - \bm\beta\|^2 \le (1 - \epsilon/2)\alpha^2\tau^2 \rb \\
\le \P_{\bm\pi}\left( \X_I'\bm y/\alpha - \tau \|\X_I\|^2/2 \ \mbox{is at least the}\ \lceil 2/\epsilon\rceil^{\text{th}} \ \mbox{largest} \right).
\end{multline}

We complete the proof by showing that the probability in the
right-hand side of \eqref{eq:2_delta_large_reg} is negligible
uniformly over all estimators $\bm{\hat\beta}$ as $p \goto \infty$. By
the independence between $I$ and $\X, \bm z$, we can assume $I = 1$
while evaluating this probability. With this in mind, we aim to show
that there are sufficiently many $i$'s such that
\[
\X_i'\bm y/\alpha - \frac{\tau}{2} \|\X_i\|^2 - \X_1'\bm y/\alpha +
\frac{\tau}{2} \|\X_1\|^2 = \X_i'\left( \bm z/\alpha + \tau \X_1
\right) - \frac{\tau}{2}\|\X_i\|^2 - \X_1'\bm{z}/\alpha -
\frac{\tau}{2}\|\X_1\|^2
\]
is positive.  Since
\[
\X_1'\bm{z}/\alpha + \frac{\tau}{2}\|\X_1\|^2 = O_{\P}(1/\alpha) + \frac{\tau}{2}\left( 1 + O_{\P}(1/\sqrt{n}) \right),
\]
it suffices to show that
\begin{equation}\label{eq:lower_app_str}
\#\left\{ 2 \le i \le p: \X_i'\left( \bm z/\alpha + \tau \X_1 \right)  -  \frac{\tau}{2}\|\X_i\|^2 > \frac{C_1}{\alpha} + \frac{\tau}{2} + \frac{C_2\tau}{\sqrt{n}}\right\} \le \lceil 2/\epsilon \rceil - 1
\end{equation}
holds with vanishing probability for all positive constants $C_1, C_2$. By the independence between $\X_i$ and $\bm z/\alpha + \tau \X_1$, we can replace $\bm z/\alpha + \tau \X_1$ by $(\|\bm z/\alpha + \tau \X_1\|, 0, \ldots, 0)$ in \eqref{eq:lower_app_str}. That is, 
\[
\X_i'\left( \bm z/\alpha + \tau \X_1 \right)  -  \frac{\tau}{2}\|\X_i\|^2 \eqd \left\| \bm z/\alpha + \tau \X_1 \right\| X_{i,1} - \frac{\tau}{2}X_{i,1}^2 - \frac{\tau}{2}\|\X_{i,-1}\|^2,
\]
where $\X_{i,-1} \in \R^{n-1}$ is $\X_i$ without the first entry. To
this end, we point out that the following three events all happen with
probability tending to one:
\begin{equation}\label{eq:lower_small05}
\begin{gathered}
\#\{2 \le i \le p: \|\X_{i,-1}\| \le 1\}/p \goto 1/2 , \\
\max_i X_{i,1}^2 \le \frac{2\log p}{n},\\
\left\| \bm z/\alpha + \tau \X_1 \right\| \ge \left( \sqrt{n} - \sqrt{\log p} \right)/\alpha .
\end{gathered}
\end{equation}
Making use of this and \eqref{eq:lower_app_str}, we only need to show
that
\begin{multline*}
  N \triangleq \#\left\{ 2 \le   i \le 0.49p  : \frac{1}{\alpha}\left( 1 - \sqrt{(\log p)/n} \right) \sqrt{n}X_{i,1} >  \frac{\tau\log p}{n} + \frac{\tau}{2} + \frac{C_1}{\alpha} + \frac{\tau}{2} + \frac{C_2\tau}{\sqrt{n}}\right\}\\
  \#\left\{ 2 \le i \le 0.49p: \frac{1}{\alpha}\left( 1 - \sqrt{(\log
        p)/n} \right) \sqrt{n}X_{i,1} > \tau + \frac{\tau\log p}{n} +
    \frac{C_1}{\alpha} + \frac{C_2\tau}{\sqrt{n}} \right\}
\end{multline*}
obeys 
\begin{equation}
  \label{eq:lower_app_04}
  N \le \lceil 2/\epsilon \rceil - 1
\end{equation}
with vanishing probability. The first line of
  \eqref{eq:lower_small05} shows that there are at least $0.49p$ many
  $i$'s such that $\|\X_{i,-1}\| \le 1$ and we assume they correspond
  to indices $2 \le i \le 0.49p$ without loss of generality. (Note
  that $N$ is independent of all $\X_{i,-1}$'s.)  Observe that
\[
\tau' := \frac{\tau + \tau(\log p)/n + C_1/\alpha + C_2\tau/\sqrt{n}}{\left( 1 - \sqrt{(\log p)/n}\right)/\alpha} = \alpha\left( 1 +  2\sqrt{\frac{\log p}{n}} \right)\tau  + O(1)
\]
for sufficiently large $p$ (to ensure $(\log p)/n$ is small). Hence,
plugging the specific choice of $\tau$ and using $\alpha \le 1$, we
obtain 
\[
\tau' \le \left( 1 +  2\sqrt{(\log p)/n} \right)\tau  + O(1) \le \sqrt{2\log p} - \log\sqrt{2\log p} + O(1),
\]
which reveals that $\sqrt{2\log (0.49p)} - \tau' = \sqrt{2\log p} - \tau'  + o(1) \goto \infty$. Since $\sqrt{n}X_{n,i}$ are \iid $\mathcal{N}(0,1)$, Lemma \ref{lm:many_large_bonf} validates \eqref{eq:lower_app_04}.
\end{proof}

%%%%%%%%%%

\begin{proof}[Proof of Corollary \ref{cor:lowerbound_general_pred}]
  Let $c > 0$ be a sufficiently small constant to be determined
  later. It is sufficient to prove the claim with $p$ replaced by a
  possibly smaller value given by $p^{\star} := \min\{\lfloor cn
  \rfloor, p\}$ (if we knew that $\beta_i = 0$ for $p^\star+1 \le i
  \le p$, the loss of any estimator $\X \hat{\bm \beta}$ would not
  increase after projecting onto the linear space spanned by the first
  $p^\star$ columns). Hereafter, we assume $\X \in \R^{n \times
    p^{\star}}$ and $\bm\beta \in \R^{p^{\star}}$.  Observe that $p =
  O(n)$ implies $p = O(p^{\star})$ and, therefore,
\begin{equation}\label{eq:lower_coro_pred}
\log(p^{\star}/k) \sim \log(p/k).
\end{equation}
In particular, $k/p^{\star} \goto 0$ and $n/\log(p^{\star}/k) \goto \infty$. This suggests that we can apply Theorem \ref{thm:reg_prob_lowerbound} to our problem, obtaining
\begin{equation}\nonumber
\inf_{\hat{\bm\beta}}\sup_{\|\bm\beta\|_0 \le k}\P\left(\frac{\|\hat{\bm\beta} - \bm\beta\|^2}{2k\log (p^\star/k)} > 1 - \epsilon' \right) \rightarrow 1.
\end{equation}
for every constant $\epsilon' > 0$. Because of
\eqref{eq:lower_coro_pred}, we also have
\begin{equation}\label{eq:lower_coro_reeuced}
\inf_{\hat{\bm\beta}}\sup_{\|\bm\beta\|_0 \le k}\P\left(\frac{\|\hat{\bm\beta} - \bm\beta\|^2}{2k\log(p/k)} > 1 - \epsilon' \right) \rightarrow 1
\end{equation}
for any $\epsilon' > 0$. 

Since $p^\star/n \le c \le 1$, the smallest singular value of the
Gaussian random matrix $\X$ is at least $1 - \sqrt{c} + o_{\P}(1)$
(see, for example, \cite{vershynin}). This result, together with
\eqref{eq:lower_coro_reeuced}, yields
\begin{equation}\nonumber
\inf_{\hat{\bm\beta}}\sup_{\|\bm\beta\|_0 \le k}\P\left(\frac{\|\X\hat{\bm\beta} - \X\bm\beta\|^2}{2k\log(p/k)} > (1-\sqrt{c})^2(1 - \epsilon') \right) \rightarrow 1
\end{equation}
for each $\epsilon' > 0$. Finally, choose $c$ and $\epsilon'$
sufficiently small such that $(1-\sqrt{c})^2(1 - \epsilon') > 1 -
\epsilon$.
\end{proof}

%%% Local Variables:
%%% mode: latex
%%% TeX-master: "paper"
%%% End:

\end{document}